\documentclass{amsart}

\usepackage{geometry}
\numberwithin{equation}{section}
\numberwithin{figure}{section}

\usepackage{amsmath, amsthm, amssymb, amstext, amsfonts}
\usepackage{enumerate}
\usepackage{moreenum}
\usepackage[T1]{fontenc} 
\usepackage[applemac]{inputenc}

\usepackage{graphicx}
\usepackage[dvipsnames]{xcolor}
\usepackage{tikz}
\usetikzlibrary{arrows,patterns,hobby}
\usepackage[bookmarks,hyperfootnotes=false, psdextra=true]{hyperref}
\usepackage[nameinlink, capitalise, noabbrev]{cleveref}
\colorlet{darkishRed}{red!60!black}
\colorlet{darkishBlue}{blue!60!black}
\colorlet{darkishGreen}{green!50!black}
\colorlet{lightishGreen}{green!70!black}
\hypersetup{
    draft = false,
    bookmarksopen=true,
    colorlinks,
    linkcolor={darkishBlue},
    citecolor={lightishGreen},
    urlcolor={darkishBlue}
}

\usepackage[msc-links,abbrev]{amsrefs} 

\theoremstyle{plain}

\newtheorem{thm}{Theorem}[section]
\crefname{thm}{Theorem}{Theorems}
\newtheorem{lem}[thm]{Lemma}
\crefname{lem}{Lemma}{Lemmas}
\newtheorem{cor}[thm]{Corollary}
\newtheorem{prop}[thm]{Proposition}

\newtheorem{prob}[thm]{Problem}

\newtheorem*{prob*}{Problem}
\newtheorem{conj}[thm]{Conjecture}
\newtheorem{obs}[thm]{Observation}
\crefname{obs}{Observation}{Observations}

\newtheorem{mainresult}{Theorem}

\crefname{mainresult}{Theorem}{Theorems}

\theoremstyle{definition}

\newcommand{\R}{\ensuremath{\mathbb{R}}}
\newcommand{\N}{\ensuremath{\mathbb{N}}}
\newcommand{\Z}{\ensuremath{\mathbb{Z}}}

\newcommand{\cC}{\ensuremath{\mathcal{C}}}
\newcommand{\cD}{\ensuremath{\mathcal{D}}}
\newcommand{\cE}{\ensuremath{\mathcal{E}}}

\newcommand{\cV}{\ensuremath{\mathcal{V}}}

\newcommand{\Aut}{\textnormal{Aut}}

\newcommand{\sub}{\subseteq}

\def\qt{quasi-tran\-si\-tive}
\def\qi{quasi-iso\-metric}
\def\qiy{quasi-iso\-me\-try}

\def\lf{locally finite}

\newcommand{\comment}[1]{}

\def\?#1{\vadjust{\vbox to 0pt{\vss\vskip-8pt\leftline{%
     \llap{\hbox{\vbox{\pretolerance=-1
     \doublehyphendemerits=0\finalhyphendemerits=0
     \hsize16truemm\tolerance=10000\small
     \lineskip=0pt\lineskiplimit=0pt
     \rightskip=0pt plus16truemm\baselineskip8pt\noindent
     \hskip0pt        
     #1\endgraf}\hskip7truemm}}}\vss}}}

\newenvironment{txteq*}
  {
    \begin{equation*}
    \begin{minipage}[c]{0.85\textwidth} 
    \em                                
  }
  {\end{minipage}\end{equation*}\ignorespacesafterend}

\let\eps=\varepsilon
\let\phi=\varphi

\newcommand{\defn}[1]{{\color{darkishRed}{\emph{#1}}}}
\newcommand{\defnm}[1]{{\color{darkishRed}{#1}}}

\usepackage{enumitem}

\newtheorem{innercustomthm}{}

\newenvironment{customthm}[1]
  {\renewcommand\theinnercustomthm{#1}\innercustomthm}
  {\endinnercustomthm}

\newcounter{claimcounter}[thm]
\newtheorem{claim}[claimcounter]{Claim}
\newtheorem*{claim*}{Claim}
\crefname{claim}{Claim}{Claims}

\newenvironment{claimproof}
{\begin{proof}
 [Proof.]
 \vspace{-1.5\parsep}
}
{\renewcommand{\qed}{\hfill $\Diamond$} \end{proof}}

\usepackage[skip=6pt]{subcaption} 

\makeatletter
\newcommand\thankssymb[1]{\textsuperscript{\@fnsymbol{#1}}}
\makeatother

\begin{document}
	
	\title{Asymptotic half-grid and full-grid minors}
	\author[Sandra Albrechtsen]{Sandra Albrechtsen}
	\author[Matthias Hamann]{Matthias Hamann\thankssymb{1}}
	\thanks{\thankssymb{1} Funded by the Deutsche Forschungsgemeinschaft (DFG) - Project No.\ 549406527.}
	\address{University of Hamburg, Department of Mathematics, Bundesstra{\ss}e 55 (Geomatikum), 20146 Hamburg, Germany}
	\email{\{sandra.albrechtsen,matthias.hamann\}@uni-hamburg.de}
	

	\begin{abstract}
		We prove that every locally finite, quasi-transitive graph with a thick end whose cycle space is generated by cycles of bounded length contains the full-grid as an asymptotic minor and as a diverging minor. This in particular includes all locally finite Cayley graphs of finitely presented groups that are not virtually free, and partially solves problems of Georgakopoulos and Papasoglu and of Georgakopoulos and Hamann.
		
		Additionally, we show that every (not necessarily quasi-transitive) graph of finite maximum degree which has a thick end and whose cycle space is generated by cycles of bounded length contains the half-grid as an asymptotic minor and as a diverging minor.
	\end{abstract}
	
	\maketitle
	
	\section{Introduction}\label{sec_intro}
	
	Fat minors are a coarse or metric variant of graph minors. They first appeared in works of Chepoi, Dragan, Newman, Rabinovich and Vaxes~\cite{FatK23Minor} and of Bonamy, Bousquet, Esperet, Groenland, Liu, Pirot and Scott~\cite{asymptoticdimminorclosed}. They play an important role in many (open) problems at the intersection of structural graph theory and coarse geometry -- an area which can be described as `coarse graph theory'. 
	
	A \defn{model} of a graph $X$ in a graph~$G$ is a collection of connected \emph{branch sets} and \emph{branch paths} in~$G$ such that after contracting each branch set to a vertex, and each branch path to an edge, we obtain a copy of~$X$. A model of $X$ is \defn{$K$-fat} (in $G$), for some~$K \in \N$, if its branch sets and paths are pairwise at least~$K$ apart, except that we do not require this for incident branch set-path pairs (see also \cref{subsec:fatminors} for the definition). We say that $X$ is a \defn{($K$-fat) minor} of $G$ if $G$ contains a ($K$-fat) model of $X$. The graph $X$ is an \defn{asymptotic minor} of $G$ if $X$ is a $K$-fat minor of $G$ for every $K \in \N$.
	An important advantage of asymptotic minors over the usual minors is that they are preserved under quasi-isometries, 
	and in particular, it does not depend on the choice of a finite generating set whether a Cayley graph of a finitely generated group contains a fixed graph as an asymptotic minor~\cite{GP2023+}.
	
	Recently, Georgakopoulos and Papasoglu~\cite{GP2023+} suggested a systematic approach of graph problems from a coarse perspective, establishing the bases of what they called 'coarse graph theory', where they presented results and open problems regarding the interplay of geometry and graphs, many of which concern fat minors. These problems have already attracted quite some attention; some (partial) solutions can be found in \cites{FatK4Minor,M24+,coarsecacti,CounterexAgelosPanosConjecture,FatK23Minor,ExcludingKAleph0,radialpathwidth}. Our main contribution is a partial resolution of a problem of Georgakopoulos and Papasoglu about asymptotic grid minors in quasi-transitive graphs~\cite{GP2023+}*{Problem~7.3}. To state this problem, we first need some definitions.
	\medskip
	
	An \defn{end} of a graph $G$ is an equivalence class of rays where two rays in $G$ are equivalent if there are infinitely many pairwise disjoint paths between them in $G$. An end is
	\defn{thick} if it has infinitely many pairwise disjoint rays. The \defn{full-grid} is the graph on $\Z \times \Z$ in which two vertices $(m,n)$ and $(m',n')$ are adjacent if and only if $|m - m'| + |n - n'| = 1$, and the \defn{half-grid}\footnote{Note that usually the grid on $\N^2$ is referred to as the half-grid. However, for us it will be more convenient to work with the grid on $\N \times \Z$. It is easy to see that our results about the half-grid also hold for the grid on $\N^2$.} is its induced subgraph on $\N \times \Z$.
	
	One of the cornerstones of infinite graph theory is \emph{Halin's Grid Theorem}~\cite{halin65}*{Satz 4$'$}, which asserts that every graph with a thick end contains the half-grid as a minor.
	In contrast to that, graphs with a full-grid as a minor form a proper subclass of the graphs with a thick end: while it is clearly true that every graph with a full-grid minor has a thick end, the converse is false in general, as the half-grid itself already witnesses. However, as it turned out, if we only consider graphs which are \defn{\qt}, i.e.\ graphs whose vertex set has only finitely many orbits under its automorphism group, then these two graph classes coincide. Indeed, Georgakopoulos and the second author~\cite{GH24+} showed that every \lf, \qt\ graph with a thick end contains the full-grid as a minor.
	
	Georgakopoulos and Papasoglu \cite{GP2023+} asked whether this result can be generalised to the coarse setting in the following sense. 
	
	\begin{prob}{\cite{GP2023+}*{Problem~7.3}} \label{mainprob:HalfGrid:CayleyGraphs}
		Let $G$ be a locally finite 
		Cayley graph of a one-ended finitely generated group. Must the half-grid be an asymptotic minor of $G$? Must the full-grid be an asymptotic minor of~$G$?
	\end{prob}
	
	\noindent Note that every Cayley graph of a group is (quasi-)transitive. Moreover, the unique end of a one-ended, quasi-transitive graph is always thick~\cites{T1992,CLM2019}. We remark that \cref{mainprob:HalfGrid:CayleyGraphs} can also be thought of as a coarse version of the grid minor theorem from Robertson and Seymour \cite{GMV} (because locally finite, quasi-transitive graphs with unbounded tree width are exactly those which contain the infinite (half- and full-)grid as a minor \cites{GH24+,KM08}).
	\medskip
	
	Our main theorem partially answers both questions in the affirmative, under the additional assumption that $G$ is a \lf\ Cayley graph of a finitely presented group.
	In fact, we show the following result. 
	
	\begin{mainresult} \label{main:AsymptoticFullGrid}
		Let $G$ be a locally finite, \qt\ graph whose cycle space is generated by cycles of bounded length.
		If $G$ has a thick end, then the full-grid is an asymptotic minor of~$G$.
	\end{mainresult}
	
	\noindent (We refer the reader to \cref{sec:CycleSpace} for the definitions concerning the cycle space.)
	\medskip
	
	Note that \cref{main:AsymptoticFullGrid} includes all \lf\ Cayley graphs of finitely presented groups. 
	Examples such as inaccessible graphs and groups \cites{D1993,D2011} or Diestel-Leader graphs \cites{DL2001,EFW2012} indicate that the geome\-try of arbitrary \lf, \qt\ (or Cayley) graphs may be far more involved. 
	This is why generalising \cref{main:AsymptoticFullGrid} to \lf\ Cayley graphs of arbitrary finitely generated groups or even to all \lf, \qt\ graphs may be much harder, and will require a different approach to that presented in this paper (see the sketch of the proof in \cref{sec:NewDefs} for details).
	
	For the proof of \cref{main:AsymptoticFullGrid} we construct, for every such graph $G$, a single model of the full-grid (see \cref{thm:AsymptoticFullGrid}), which can be turned into a $K$-fat model of the full-grid, for every $K \in \N$, by deleting some of its branch sets and paths. Moreover, it can be turned into a model of the full-grid that \emph{diverges}:
	for any two diverging sequences of vertices and/or edges of the full-grid also their branch sets/paths diverge in~$G$ (see \cref{subsec:DivSubdivisions} for the definition).
	
	\begin{mainresult} \label{main:DivergingFullGrid}
		Let $G$ be a locally finite, \qt\ graph whose cycle space is generated by cycles of bounded length.
		If $G$ has a thick end, then the full-grid is a diverging minor of $G$.
	\end{mainresult}
	
	\noindent This partially solves a question of Georgakopoulos and the second author \cite{GH24+}*{Problem~4.1}.
	\medskip

	Kr\"on and M\"oller \cite{KM08}*{Theorem~5.5} proved that a \lf, \qt\ connected graph has no thick end if and only if it is \qi\ to a tree.
	Thus, instead of assuming that the graph $G$ in \cref{main:AsymptoticFullGrid,main:DivergingFullGrid} has a thick end, we may assume that $G$ is not \qi\ to a tree (see \cref{subsec:QuasiIsomToTrees} for details).
	\medskip
	
	As a first step in the proof of \cref{main:AsymptoticFullGrid}, we find the half-grid as an asymptotic minor. For this, we do not need the transitivity assumption on~$G$. Indeed, we prove the following theorem.
	
	\begin{mainresult} \label{main:HalfGrid:BoundedCycles}
		Let $G$ be a graph of finite maximum degree whose cycle space is generated by cycles of bounded length. 
		If $G$ has a thick end, then the half-grid is an asymptotic minor of $G$.
	\end{mainresult}
	
	\noindent Note that every graph satisfying the premise of \cref{main:AsymptoticFullGrid,main:DivergingFullGrid} has finite maximum degree as it is \lf\ and \qt.
	\medskip
	
	Similar to the proof of \cref{main:AsymptoticFullGrid}, we again construct a single model of the half-grid (see \cref{thm:HalfGrid:BoundedCycles}), which can be turned into a $K$-fat model of the half-grid, for every $K \in \N$, and into a diverging model of the half-grid.
	
	\begin{mainresult} \label{main:DivergingHalfGrid}
		Let $G$ be a graph of finite maximum degree whose cycle space is generated by cycles of bounded length.
		If $G$ has a thick end, then the half-grid is a diverging minor of $G$.
	\end{mainresult}
	
	\noindent This partially solves a question of Georgakopoulos and the second author \cite{GH24+}*{Problem 4.2}.
	\medskip
	
	This paper is structured as follows. 
	In \cref{sec:Prelims} we recall some important definitions. 
	\cref{sec:NewDefs} con\-sists of three parts. We first introduce some new definitions in \cref{subsec:UltraFatMinors,subsec:EscapingSubdivisions}. We then give a sketch of the proofs of \cref{main:AsymptoticFullGrid,main:DivergingFullGrid,main:HalfGrid:BoundedCycles,main:DivergingHalfGrid} in \cref{subsec:ProofSketch12,subsec:ProofSketch34}, where we also state \cref{thm:AsymptoticFullGrid,thm:HalfGrid:BoundedCycles}, our stronger results on half-grid and full-grid minors, which we already briefly mentioned above. In \cref{subsec:EscSubYieldsFatDivMinor,subsec:FormalProofOfMainResults} we derive \cref{main:AsymptoticFullGrid,main:DivergingFullGrid,main:HalfGrid:BoundedCycles,main:DivergingHalfGrid} from \cref{thm:AsymptoticFullGrid,thm:HalfGrid:BoundedCycles}.
	\cref{sec:DivAndGeodDR} contains some preparatory work about diverging and quasi-geodesic rays. We then prove \cref{thm:HalfGrid:BoundedCycles,thm:AsymptoticFullGrid} in \cref{sec:FG,sec:HG:BC}, respectively.
	We finish in \cref{sec:problems} by discussing some related problems.

	\section{Preliminaries} \label{sec:Prelims}
	
	Our notions mainly follow~\cite{Bibel}. In what follows, we recap some important definitions which we need later.
	
	Given sets $U' \subseteq U$ of vertices of a graph $G$, a component $C$ of $G-U$ \defn{attaches} to~$U'$ if $C$ has a neighbour in~$U'$.
	The \defn{boundary $\partial_G X$} of a subgraph $X$ of $G$ is the set $N_G(V(G- V(X)))$ of vertices of~$X$ that send in~$G$ an edge outside of~$X$. For example, the boundary $\partial_G C$ of a component $C$ of $G-U$ is $N_G(U) \cap V(C)$.
	
	A graph $G$ is \defn{\qt} if the automorphism group of~$G$ acts on $V(G)$ with only finitely many orbits, that is if $V(G)$ can be partitioned into finitely many sets $U_0, \dots, U_n$ such that for all $i \in \{0, \dots, n\}$ and $u, v \in U_i$ there exists an automorphism~$\phi$ of~$G$ such that $\phi(u) = v$.
	The \defn{stabilizer} of a subgraph~$X$ of~$G$ consists of precisely those automorphisms of~$G$ that map~$X$ to itself.

	\subsection{Paths, rays and combs}
	
	For two sets $A,B$ of vertices of $G$, an \defn{$A$--$B$ path} meets~$A$ precisely in its first vertex and~$B$ precisely in its last vertex.
	For a subgraph $H$ of $G$, an \defn{$H$-path} is a non-trivial path which meets $H$ precisely in its endvertices.
	
	A \defn{ray} is a one-way infinite path, and a \defn{double ray} is a two-way infinite path. A \defn{tail} of a (double) ray~$R$ is any ray $S \subseteq R$.
	If $R = r_0 r_1 \dots$ is a ray, then we denote by \defn{$r_iRr_j$} for $i, j \in \N$ the subpath $r_i \dots r_j$ of~$R$, and by \defn{$r_iR$} or \defn{$R_{\geq i}$} the tail $r_i r_{i+1} \dots$ of~$R$.
	Further, we denote by \defn{$Rr_i$} or \defn{$R_{\leq i}$} the subpath $r_0 \dots r_i$ of~$R$.
	We use these notions analogously for double rays; in particular, if $R = \dots r_{-1}r_0r_1 \dots$ is a double ray, then \defn{$Rr_i$} and \defn{$R_{\leq i}$} denote the tail $r_ir_{i-1} \dots$ of~$R$. 
	
	A \defn{comb} is a union of a ray $R$ with infinitely many pairwise disjoint finite paths which have precisely their first vertex on $R$; we call the last vertices of these paths the \defn{teeth} of the comb and refer to $R$ as its \defn{spine}.
	The following observation about combs in infinite graphs is well-known and follows immediately from the \emph{Star-Comb Lemma}; see e.g.\ \cite{Bibel}*{Lemma~8.2.2} for a proof.
	
	\begin{lem} \label{lem:StarComb} 
		Let $U$ be an infinite set of vertices in a \lf, connected graph~$G$. Then $G$ contains a comb with all teeth in $U$.
		
		In particular, every infinite, connected graph has a vertex of infinite degree or contains a ray.
	\end{lem}

	\subsection{(Hexagonal) grids}
	
	The \defn{full-grid}, denoted by \defn{$FG$}, is the graph on $\Z^2$ in which two vertices $(m,n)$ and $(m',n')$ are adjacent if and only if $|m - m'| + |n - n'| = 1$. 
	The \defn{hexagonal full-grid} is obtained from $FG$ by deleting every other rung, as shown in \cref{fig:HexGrid}.
	The \defn{(hexagonal) half-grid}, denoted by \defn{$HG$}, is the induced subgraph of the (hexagonal) full-grid on vertex set $\N \times \Z$. 
	
	We call the double rays $R^i$ of the (hexagonal) full- and half-grid its \defn{vertical double rays} and the edges~$e_{ij}$ its \defn{horizontal edges} (see \cref{fig:HexGrid}).
	
	\begin{figure}[ht]
		\begin{tikzpicture}
    \tikzset{edge/.style = {->,> = stealth}}
    
    \foreach \x in {-4,-3,-2,-1,0,1,2} {
    \draw[stealth-stealth] (\x,-0.4) to (\x,4.4);
    \foreach \y in {0,0.5,1,1.5,2,2.5,3,3.5,4} {
      \draw[fill,black] (\x,\y) circle (.05);
}}


\foreach \x in {-4,-2,0}{
    \foreach \y in {0,1,2,3,4} {
	\draw (\x,\y) to (\x+1,\y);
}}

\foreach \x in {-3,-1,1}{
    \foreach \y in {0.5,1.5,2.5,3.5} {
	\draw (\x,\y) to (\x+1,\y);
}}


\foreach \y in {0.5,1.5,2.5,3.5} {
	\draw (-4.5,\y) to (-4,\y);
}
         
\foreach \y in {0,1,2,3,4} {
	\draw (2,\y) to (2.5,\y);
}

         
    \node at (3,2) {$\dots$};
     \node at (-5,2) {$\dots$};

\foreach \x in {-3,-2,-1}{
    \node at (\x-0.8,-0.7) {\footnotesize{$R^{\x}$}};
     }
    \node at (-0.9,-0.7) {\footnotesize{$R^{0}$}};
\foreach \x in {1,2,3}{
    \node at (\x-0.9,-0.7) {\footnotesize{$R^{\x}$}};
     }


    \node at (-1.5,1.7) {\footnotesize{$e_{-1j}$}};
    \node at (-0.5,2.2) {\footnotesize{$e_{1j}$}};
    \node at (-0.5,1.2) {\footnotesize{$e_{1(j-1)}$}};
    \node at (-0.5,3.2) {\footnotesize{$e_{1(j+1)}$}};
    \node at (0.5,1.7) {\footnotesize{$e_{2j}$}};
    \node at (-2.5,2.2) {\footnotesize{$e_{-2j}$}};
    \node at (-3.5,1.7) {\footnotesize{$e_{-3j}$}};
    \node at (1.5,2.2) {\footnotesize{$e_{3j}$}};

\end{tikzpicture}
		\vspace{-3em}
		\caption{The hexagonal full-grid with vertical double rays $R^i$ and horizontal edges~$e_{ij}$.}
		\label{fig:HexGrid}
	\end{figure}
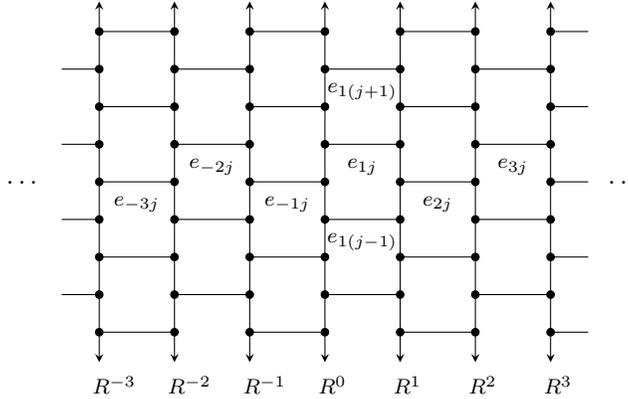

	\subsection{Cycle space}\label{sec:CycleSpace}
	
	Let $G$ be a graph. The \defn{edge space} of~$G$ is the vector space over the $2$-element field~$\mathbb{F}_2$ of all functions $E(G) \rightarrow \mathbb{F}_2$: its elements correspond to the subsets of $E(G)$ and vector addition corresponds to symmetric difference.
	The \defn{cycle space} of~$G$ is the subspace of the edge space of~$G$ spanned by all the cycles in~$G$~-- more precisely, by their edge sets; for simplicity, we will not distinguish between the edge sets in the cycle space and the subgraphs they induce in~$G$.
	
	We say that the cycle space of~$G$ is \defn{generated by cycles of bounded length} if there is some $n \in \N$ such that the cycles in~$G$ of length at most $n$ generate the cycle space of~$G$.
	
	\subsection{Ends}
	
	An \defn{end} $\eps$ of a graph~$G$ is an equivalence class of rays in~$G$ where two rays are equivalent if they are joined by infinitely many disjoint paths in~$G$ or, equivalently, if for every finite set~$U \subseteq V(G)$ both rays have tails in the same component of~$G - U$.
	A \defn{(double) $\eps$-ray} is a (double) ray whose tails are all contained in~$\eps$.
	An end is \defn{thick} if, for every $n \in \N$, there are $n$ pairwise disjoint $\eps$-rays. Halin \cite{halin65}*{Satz~1} showed that this is the case if and only if there are infinitely many pairwise disjoint $\eps$-rays.
	
	A finite set $U \subseteq V(G)$ \defn{distinguishes} two ends $\eps, \eps'$ of $G$ if no component of $G-U$ contains rays from both $\eps$ and $\eps'$. 
	A graph~$G$ is \defn{accessible} if there exists some $n \in \N$ such that every two distinct ends of~$G$ can be distinguished by a set of at most~$n$ vertices of~$G$.
	
	\begin{thm}{\cite{AccessibleCycleSpace}*{Corollary~3.2}} \label{thm:CycleSpaceAccessible}
		Every locally finite, \qt\ graph whose cycle space is generated by cycles of bounded length is accessible.
	\end{thm}

	\subsection{Distance and balls}
	
	Let $G$ be a graph.
	We write~\defn{$d_G(v, u)$} for the distance between the two vertices~$v$ and~$u$ in~$G$. 
	For two sets~$U$ and~$U'$ of vertices of~$G$, we write~\defn{$d_G(U, U')$} for the minimum distance of two elements of~$U$ and~$U'$, respectively.
	If one of~$U$ or~$U'$ is just a singleton, then we omit the braces, writing $d_G(v, U') := d_G(\{v\}, U')$ for $v \in V(G)$.
	If $X$ is a subgraph of $G$, then we abbreviate $d_G(U,V(X))$ as $d_G(U,X)$.
	
	Given a set~$U$ of vertices of~$G$, the \defn{ball (in~$G$) around~$U$ of radius $r \in \N$}, denoted by~\defn{$B_G(U, r)$}, is the set of all vertices in~$G$ of distance at most~$r$ from~$U$ in~$G$.
	If~$U = \{v\}$ for some~$v \in V(G)$, then we omit the braces, writing~$B_G(v, r)$ for the ball (in $G$) around~$v$ of radius~$r$.
	Additionally, we abbreviate the induced subgraph on $B_G(U,r)$ of $G$ with $\defnm{G[U,r]} := G[B_G(U,r)]$.
	If $X$ is a subgraph of $G$, then we abbreviate $B_G(V(X),r)$ and $G[V(X),r]$ as $B_G(X,r)$ and $G[X,r]$, respectively.
	
	A subgraph~$X$ of~$G$ is \defn{$c$-quasi-geodesic\footnote{Note that in general metric spaces one also allows for an additive error; this property is there called `$(c, b)$-quasi-geodesic'. So here a subgraph is $c$-quasi-geodesic if and only if it is $(c,0)$-quasi-geodesic. It is easily verifiable that these two notions of quasi-geodesic, with or without an additive error, are equivalent for graphs.} (in~$G$)} for some $c \in \N$ if for every two vertices $u, v \in V(X)$ we have $d_X(u, v) \leq c \cdot d_G(u, v)$.
	We call~$X$ \defn{quasi-geodesic} if it is $c$-quasi-geodesic for some $c \in \N$ and \defn{geodesic} if it is $1$-quasi-geodesic.\footnote{Geodesic subgraphs are also called \defn{isometric} subgraphs in the literature.}
	
	Two rays $R,S$ in $G$ \defn{diverge} if for every $n \in \N$ they have tails $R' \subseteq R$, $S' \subseteq S$ satisfying $d_G(R', S') > n$.
	A double ray $R$ in $G$ \defn{diverges} if every two disjoint tails of $R$ diverge.

	\subsection{Fat and diverging minors} \label{subsec:fatminors}
	
	Let $G, X$ be graphs.
	A \defn{model} $(\cV,\cE)$ of $X$ in $G$ is a collection $\cV$ of disjoint sets $V_x \subseteq V(G)$ for vertices $x$ of $X$ such that each $G[V_x]$ is connected, and a collection~$\cE$ of pairwise internally disjoint $V_{x_0}$--$V_{x_1}$ paths $E_{e}$ for edges $e=x_0x_1$ of $X$ which are disjoint from every $V_x$ with $x \neq x_0, x_1$.
	The sets $V_x$ are its \defn{branch sets} and the sets $E_e$ are its \defn{branch paths}.
	A model $(\cV, \cE)$ of $X$ in $G$ is \defn{$K$-fat} for $K \in \N$ if $d_G(Y,Z) \geq K$ for every two distinct $Y,Z \in \cV \cup \cE$ unless $Y = E_e$ and $Z = V_x$ for some vertex $x \in V(X)$ incident to $e \in E(X)$, or vice versa.
	The graph~$X$ is a \defn{($K$-fat) minor} of~$G$, denoted by \defn{$X\prec G$ ($X \prec_K G$)}, if~$G$ contains a ($K$-fat) model of~$X$. Moreover,~$X$ is an \defn{asymptotic minor} of~$G$, denoted by \defn{$X \prec_\infty G$}, if~$X$ is a $K$-fat minor of~$G$ for all $K \in \N$.
	Let $\eps$ be an end of $G$. If $X$ is a one-ended graph, then we write \defn{$X \prec_K^\eps G$} if $G$ contains a $K$-fat model $(\cV, \cE)$ of $X$ such that every ray in $\left(\bigcup_{x \in V(X)} G[V_x]\right) \cup \left(\bigcup_{e \in E(X)} E_e\right)$ is an $\eps$-ray. Similarly, we write \defn{$X \prec^\eps_\infty G$} if $X \prec^\eps_K G$ for all $K \in \N$.
	
	A model $(\cV, \cE)$ of~$X$ in~$G$ \defn{diverges} if for every two sequences $(x_n)_{n \in \N}$ and $(y_n)_{n\in \N}$ of vertices and/or edges of $X$ such that $d_X(x_n,y_n) \rightarrow \infty$, we have $d_G(U_n, W_n) \rightarrow \infty$ where $U_n := V_{x_n}$ if $x_n \in V(X)$ and $U_n := V(E_{x_n})$ if $x_n \in E(X)$ and analogously $W_n := V_{y_n}$ or $W_n := V(E_{y_n})$.

	\subsection{Fat and diverging subdivisions} \label{subsec:DivSubdivisions}
	
	A \defn{subdivision} of a graph~$X$ is a graph which arises from~$X$ by replacing every edge in~$X$ by a new path between its endvertices such that no new path has an inner vertex in $V(X)$ or on any other new path. The original vertices of $X$ are the \defn{branch vertices} of the subdivision and the new paths are its \defn{branch paths}.
	Let~$G$ be a graph and let~$H \subseteq G$ be a subdivision of~$X$ with branch vertices~$v_x$ for $x \in V(X)$ and branch paths $E_e$ for $e \in E(X)$. Then $H$ is \defn{$K$-fat} (in~$G$) if there are sets $V_x \subseteq V(H)$ with $v_x \in V_x$ for $x \in V(X)$ and paths $E'_e \subseteq E_e$ for $e \in E(X)$ such that $((V_x)_{x \in V(X)}, (E'_e)_{e \in E(X)})$ is a $K$-fat model of $X$.
	The subdivision~$H$ of~$X$ \defn{diverges} (in~$G$) if the model $((\{v_x\})_{x \in V(X)}, (E_e)_{e \in E(X)})$ of~$X$ in~$G$ diverges.

	\section{Further definitions and a sketch of the proof} \label{sec:NewDefs}
	
	In this section we first introduce ultra fat minors and escaping subdivisions of certain graphs (see \cref{subsec:UltraFatMinors,subsec:EscapingSubdivisions}). We then give in \cref{subsec:ProofSketch12,subsec:ProofSketch34} a sketch of the proofs of \cref{main:AsymptoticFullGrid,main:DivergingFullGrid,main:HalfGrid:BoundedCycles,main:DivergingHalfGrid}. There, we also state two stronger theorems, \cref{thm:AsymptoticFullGrid,thm:HalfGrid:BoundedCycles}, from which we then derive \cref{main:AsymptoticFullGrid,main:DivergingFullGrid,main:HalfGrid:BoundedCycles,main:DivergingHalfGrid} in \cref{subsec:EscSubYieldsFatDivMinor,subsec:FormalProofOfMainResults}.

	\subsection{Ultra fat minors} \label{subsec:UltraFatMinors}
	
	We say that a model $((V_i)_{i \in \N}, (E_{ij})_{i \neq j \in \N})$ of the countably infinite clique $K_{\aleph_0}$ in a graph $G$ is \defn{ultra fat} if
	\begin{itemize}
		\item $d_G(V_i, V_j) \geq \min\{i,j\}$ for all $i \neq j \in \N$,
		\item $d_G(E_{ij}, E_{k\ell}) \geq \min\{i, j, k, \ell\}$ for all $i, j, k, \ell \in \N$ with $\{i,j\} \neq \{k, \ell\}$, and 
		\item $d_G(V_i, E_{k\ell}) \geq \min\{i, k, \ell\}$ for all $i, k, \ell \in \N$ with $i \notin \{k, \ell\}$.
	\end{itemize}
	Further, we say that $K_{\aleph_0}$ is an \defn{ultra fat minor} of $G$, and write \defn{$K_{\aleph_0} \prec_{UF} G$}, if $G$ contains an ultra fat model of $K_{\aleph_0}$.
	The idea is that an ultra fat model of $K_{\aleph_0}$ in a graph~$G$ witnesses that~$G$ contains~$K_{\aleph_0}$ as an asymptotic minor. Indeed, if $((V_i)_{i \in \N}, (E_{ij})_{i \neq j \in \N})$ is an ultra fat model of $K_{\aleph_0}$ in $G$, then $((V_i)_{i \in \N_{\geq K}}, (E_{ij})_{i \neq j \in \N_{\geq K}})$ is a $K$-fat model of $K_{\aleph_0}$ in $G$. The following observation follows from the fact that every countable graph is a subgraph of~$K_{\aleph_0}$. 
	\begin{obs} \label{obs:UFImpliesAsym}
		If a graph $G$ contains $K_{\aleph_0}$ as an ultra fat minor, then it contains every countable graph as an asymptotic minor. 
		Moreover, if $K_{\aleph_0} \prec_{UF}^\eps G$ for some end $\eps$ of $G$, then also $X \prec_\infty^\eps G$ for every one-ended, countable graph $X$. \qed
	\end{obs}
	
	Moreover, the following proposition follows easily from the definitions.
	
	\begin{prop} \label{obs:UFImpliesDiv}
		If a graph $G$ contains $K_{\aleph_0}$ as an ultra fat minor, then it contains every connected, locally finite graph as a diverging minor, and in particular, it contains every connected graph of maximum degree at most $3$ as a diverging subdivision. 
		Moreover, if $K_{\aleph_0} \prec_{UF}^\eps G$ for some end $\eps$ of $G$, then we may choose the diverging minor / subdivision so that all its rays lie in $\eps$.
	\end{prop}

	\begin{proof}
		We only show the first part of the assertion, the `moreover'-part then follows easily.
		
		Let $(\cV, \cE)$ be an ultra-fat model of $K_{\aleph_0}$ in $G$, and let $X$ be some connected, locally finite graph. Then~$X$ is a subgraph of~$K_{\aleph_0}$. Let $(\cV', \cE')$ be a `submodel' of $(\cV, \cE)$ that corresponds to some copy of~$X$ in~$K_{\aleph_0}$. Since~$X$ is locally finite, we may assume, by taking subsets of the branch sets in~$\cV'$ if necessary, that all branch sets in~$\cV'$ are finite. 
		Now let $(x_n)_{n\in \N}$ and $(y_n)_{n \in \N}$ with $x_n, y_n \in V(X) \cup E(X)$ for all $n \in \N$ such that $d_X(x_n, y_n) \rightarrow \infty$. Let $U'_n, W'_n$ be the branch sets or paths of $(\cV', \cE')$ corresponding to $x_n, y_n$, respectively. 
		Now suppose for a contradiction that $d_G(U'_n, W'_n)$ does not tend to infinity. Then we may assume, by restricting to subsequences, that $d_G(U'_n, W'_n) = K$ for some $K \in \N$ and all $n \in \N$. This can only happen if one of $(x_n)_{n \in \N}$ and $(y_n)_{n \in \N}$, say $(x_n)_{n \in \N}$ is eventually constant because $(\cV, \cE)$ is ultra-fat and the branch sets and paths of $(\cV', \cE')$ are subsets of the branch sets and paths of $(\cV, \cE)$. Hence, we may assume that $x_n = z$ for all $n \in \N$ and some $z \in V(X) \cup E(X)$. 
		Since $d_G(x_n, y_n) \rightarrow \infty$ and $X$ is connected, it follows that (after possibly restricting to a subsequence of $(y_n)_{n \in \N}$) that the $y_n$'s are pairwise distinct.
		As $V_z$ is finite and $d_G(V_z, W'_n) = K$ for all $n \in \N$, there is some $v \in V_z$ and an infinite index set $I \subseteq \N$ such that $d_G(v, W'_n) = K$ for all $n \in I$. Hence, $d_G(W'_n, W'_m) \leq 2K$ for all $n,m \in \N$. But since $I$ is infinite and the $W'_n$'s are subsets of the branch sets and paths of $(\cV, \cE)$, this contradicts that $(\cV, \cE)$ is ultra-fat.
	\end{proof}

	\subsection{Escaping subdivisions} \label{subsec:EscapingSubdivisions}
	
	We call the double rays in a subdivision of the hexagonal half- or full-grid corresponding to the vertical double rays $R^i$ of the hexagonal half- or full-grid its \defn{vertical (double) rays}, and the branch paths corresponding to the horizontal edges $e_{ij}$ its \defn{horizontal paths}, and we usually denote the former by~$S^i$ and the latter by $P_{ij}$. 
	Whenever we introduce a subdivision of the hexagonal half- or full-grid with vertical double rays $S^i$ without specifying the vertex sets of the $S^i$'s, we tacitly assume that $S^i = \dots s^i_{-1} s^i_0 s^i_{1} \dots$ and that their tails $S^i_{\geq 0}$ are the image of the `upper' half of the vertical double ray~$R^i$ of the hexagonal half- or full-grid.
	
	Let $G$ be a graph and let $H \subseteq G$ be a subdivision of the hexagonal half-grid with vertical double rays~$S^i$ and horizontal paths $P_{ij}$ for all $i\in\N$ and $j\in\Z$. We say that $H$ is \defn{escaping} if $S^0$ is diverging and if there are $0 := M_0 < M_1 < \ldots \in \N$ such that $M_i > M_{i-1}+2i$ for all $i \geq 1$ and
	\begin{enumerate}[label=\rm{(\roman*)}]
		\item \label{itm:DefEscHG:S^i} $S^i \subseteq G[S^0, M_i] - B_G(S^0, M_{i-1}+2i)$ for all $i \in \N_{\geq 1}$, and
		\item \label{itm:DefEscHG:P_ij} $P_{1j} \subseteq G[S^0, M_1]$ and $P_{ij} \subseteq G[S^0, M_i] - B_G(S^0, M_{i-2}+i)$ for all $i \in \N_{\geq 2}$ and $j \in \Z$.
	\end{enumerate}
	
	A subdivision $H \subseteq G$ of the hexagonal full-grid with vertical double rays $S^i$ and horizontal paths $P_{ij}$ is \defn{escaping} if the $S^i$'s and $P_{ij}$'s with $i \geq 0$ form an escaping subdivision of the hexagonal half-grid as well as the $S^i$'s and $P_{ij}$'s with $i \leq 0$, and if there is some $M \in \N$ such that the $S^i$'s with $i > 0$ are contained in a different component of $G- B_G(S^0, M)$ than the $S^i$'s with $i < 0$.
	
	We remark that properties \ref{itm:DefEscHG:S^i} and \ref{itm:DefEscHG:P_ij} in particular ensure that the distances between the double rays and horizontal paths of an escaping subdivision increase with the $i$-coordinate; e.g.\ $d_G(S^i, S^{i'}) \geq 2i'$ for $i < i'$.
	So roughly speaking, an escaping subdivision (of the hexagonal half- or full-grid) has some good divergence properties with respect to the $i$-coordinate. In \cref{subsec:EscSubYieldsFatDivMinor} we show that, since $S^0$ diverges, the subdivision also has good divergence properties with respect to the $j$-coordinate (up to some extraction), and in fact contains a diverging subdivision, and, for every $K \in \N$, a $K$-fat subdivision of the hexagonal half- or full-grid, respectively.

	\subsection{Sketch of the proofs of \texorpdfstring{\cref{main:HalfGrid:BoundedCycles,main:DivergingHalfGrid}}{Theorems 3 and 4}} \label{subsec:ProofSketch34}
	
	We will prove \cref{main:HalfGrid:BoundedCycles,main:DivergingHalfGrid} simultaneously by showing the following stronger result. 
	
	\begin{thm} \label{thm:HalfGrid:BoundedCycles}
		Let $\eps$ be a thick end of a graph $G$ with finite maximum degree whose cycle space is generated by cycles of bounded length. 
		Then either $K_{\aleph_0} \prec_{UF}^\eps G$ or $G$ contains an escaping subdivision $H$ of the hexagonal half-grid whose rays all lie in $\eps$. 
	\end{thm}
	
	By \cref{obs:UFImpliesAsym,obs:UFImpliesDiv}, an ultra fat model of $K_{\aleph_0}$ contains a diverging and a $K$-fat model of the half-grid for every $K\in\N$. So to derive \cref{main:HalfGrid:BoundedCycles,main:DivergingHalfGrid} from \cref{thm:HalfGrid:BoundedCycles} it suffices to show that an escaping subdivision of the hexagonal half-grid also contains a diverging and a $K$-fat subdivision of the hexagonal half-grid (see \cref{subsec:FormalProofOfMainResults,subsec:EscSubYieldsFatDivMinor}).
	\medskip
	
	For the proof of \cref{thm:HalfGrid:BoundedCycles}, we first show that $G$ contains for every thick end $\eps$ a diverging double $\eps$-ray $R$ (see \cref{thm:DivergingRays}),
	and we then set $S^0 := R$. Second, we show that $G$ contains double rays $S^1, S^2, \ldots$ such that the $S^i$'s are contained in increasingly distant `thickened cylinders' around $R$ of the form $G[R, M_i] - B_G(R, M_{i-1}+2i)$ for some $M_0 < M_1 < \ldots \in \N$, as required by \ref{itm:DefEscHG:S^i} for the vertical double rays of an escaping subdivision of the hexagonal half-grid.
	Finally, we connect the $S^i$'s by infinitely many paths so that infinitely many of them either form the vertical double rays of an escaping subdivision of the hexagonal half-grid or form the branch sets of an ultra fat model of $K_{\aleph_0}$ (see \cref{lem:HalfGrid}). 
	
	Let us describe the second step in more detail. We will choose the $S^i$'s recursively, starting from the diverging double $\eps$-ray $R = \dots r_{-1}r_0r_1 \dots (=: S^0)$.
	For this, we first show that $C[\partial_G C, \lfloor\frac{\kappa-2}{2}\rfloor] = C \cap G[R, L + \lfloor \frac{\kappa}{2}\rfloor]$ is connected for every $L \in \N$ and every component $C$ of $G - B_G(R,L)$, where $\kappa \in \N$ is such that the cycle space of $G$ is generated by cycles of length $\leq \kappa$ (see \cref{lem:kappa/2NhoodIsConnected}). Note that this is almost the only part in the proofs of \cref{main:AsymptoticFullGrid,main:DivergingFullGrid,main:DivergingHalfGrid,main:HalfGrid:BoundedCycles} where we use the assumption that the cycle space is generated by cycles of bounded length; nevertheless, the assumption is crucial here, and the rest of the proof relies on this lemma.\footnote{We in fact prove stronger versions of \cref{main:AsymptoticFullGrid,main:DivergingFullGrid,main:DivergingHalfGrid,main:HalfGrid:BoundedCycles} (see \cref{subsec:FormalProofOfMainResults} below), which find the desired minors in a prescribed end. For this, in the case of \cref{main:AsymptoticFullGrid,main:DivergingFullGrid}, we need the assumption that the cycle space is generated by cycles of bounded length once more, to ensure that the graph is accessible, which is the only other place where this assumption is needed.} 
	
	We then show that for every $L \in \N$ some component~$C$ of $G - B_G(R,L)$ is `long', i.e.\ it has a neighbour in $B_G(R_{\geq j},L)$ and in $B_G(R_{\leq -j}, L)$ for all $j \in \N$ (see \cref{lem:LongCompsExist}). 
	Combining that $C$ is long and $C[\partial_G C, \lfloor\frac{\kappa-2}
	{2}\rfloor]$ is connected then allows us to find a double ray in $C[\partial_G C, \lfloor\frac{\kappa-2}
	{2}\rfloor]$, which thus also lies in $G[R,L + \lfloor\frac{\kappa}{2}\rfloor] - B_G(R, L)$. 
	Hence, we may proceed recursively by increasing the radius $L$ of the ball around $R$ by a summand of $\lfloor\frac{\kappa}{2}\rfloor + 2i$ in each step.

	\subsection{Sketch of the proof of \texorpdfstring{\cref{main:AsymptoticFullGrid,main:DivergingFullGrid}}{Theorems 1 and 2}} \label{subsec:ProofSketch12}
	
	Similarly to before, we prove \cref{main:AsymptoticFullGrid,main:DivergingFullGrid} simultaneously, by showing the following stronger result.
	
	\begin{thm} \label{thm:AsymptoticFullGrid}
		Let $\eps$ be a thick end of a \lf, \qt\ graph $G$ whose cycle space is generated by cycles of bounded length. Then either $K_{\aleph_0} \prec_{UF}^\eps G$ or $G$ contains an escaping subdivision of the hexagonal full-grid whose rays all lie in $\eps$.
	\end{thm}
	
	Similarly to above, \cref{obs:UFImpliesAsym,obs:UFImpliesDiv} together with results from \cref{subsec:EscSubYieldsFatDivMinor} below will show that it suffices to prove \cref{thm:AsymptoticFullGrid} in order to obtain \cref{main:AsymptoticFullGrid,main:DivergingFullGrid} (see \cref{subsec:FormalProofOfMainResults}).
	\medskip
	
	The proof of \cref{thm:AsymptoticFullGrid} builds on \cref{thm:HalfGrid:BoundedCycles}. 
	From the proof of \cref{thm:HalfGrid:BoundedCycles} it follows that we have more control over where the escaping subdivision of the hexagonal half-grid lies (see \cref{thm:HalfGrid:BoundedCycles:copy}, the detailed version of \cref{thm:HalfGrid:BoundedCycles}, in \cref{sec:HG:BC}). For this, let $\eps$ be a thick end of $G$, and let $R$ be a diverging double $\eps$-ray. Given a `thick' component $C$ of $G - B_G(R, K)$ for some $K \in \N$, that is one which contains a long component of $G-B_G(R, L)$ for every $L \geq K$,  we in fact obtain an escaping subdivision~$H$ of the hexagonal half-grid whose first vertical double ray is $R$ and which is `mostly' contained in $C$ (unless \cref{thm:HalfGrid:BoundedCycles:copy} yields an ultra fat model of $K_{\aleph_0}$, in which case we are immediately done).
	Now suppose that for some large enough $L \in \N$ there is another thick component~$D$ of $G - B_G(R, L)$. Then \cref{thm:HalfGrid:BoundedCycles} yields another escaping subdivision~$H'$ of the hexagonal half-grid whose first vertical double ray is $R$ and which is `mostly' contained in $D$ (or an ultra fat model of $K_{\aleph_0}$). Gluing~$H$ and~$H'$ together along their common first vertical double ray $R$ then yields the desired subdivision of the hexagonal full-grid (see \cref{lem:TwoThickCompsYieldFullGrid}). 
	
	It thus suffices to prove that $G$ contains a diverging double $\eps$-ray $R$ such that, for some large enough $K \in \N$, there are two distinct thick components of $G-B_G(R, K)$. This step is mainly divided into two lemmas (see \cref{lem:ThickAndHalfThickComponent,lem:TwoThickComponents}).
	We first show that if $R'$ is a double $\eps$-ray which is not only diverging but even quasi-geodesic, then it is enough that for some large enough $K \in \N$ there are distinct components $C \neq D$ of $G-B_G(R', K)$ such that~$C$ is thick but~$D$ is only `half-thick' (see \cref{lem:TwoThickComponents}) because then we can use the quasi-transitivity of~$G$ to find another quasi-geodesic double $\eps$-ray $R$ such that $G-B_G(R, K)$ has two distinct thick components (see \cref{lem:TwoThickComponents}).
	Here, a component of $G-B_G(R, K)$ is half-thick if it includes for every $L \geq K$ a component of $G-B_G(R, L)$ which is `half-long', i.e.\ which has neighbours in $B_G(R_{\geq n}, L)$ \emph{or} in $B_G(R_{\leq -n}, L)$ for all $n \in \N$.
	
	Next, we show that such a double ray $R'$ exists. For this, we first prove that~$G$ contains three $\eps$-rays $R_1, R_2, R_3$ that are pairwise disjoint except that they share the same first vertex and such that $R_1 \cup R_2 \cup R_3$ is quasi-geodesic (see \cref{lem:QuasiGeodesic3StarOfRays}). Applying the detailed version \cref{thm:HalfGrid:BoundedCycles:copy} of \cref{thm:HalfGrid:BoundedCycles} to the quasi-geodesic, and hence diverging, double ray $R_1 \cup R_2$ then yields an escaping subdivision $H$ of the hexagonal half-grid whose first vertical double ray is~$R_1 \cup R_2$. 
	Now for every $K \in \N$, by the definition of escaping,  $H$ will lie `mostly' in one component $C_K$ of $G - B_G(R, K)$, which then needs to be thick.
	We then analyse where~$R_3$ lies in relation to~$H$.
	If, for some large enough $L \in \N$, $R_3$ has a tail in a component $D_L \neq C_L$ of $G-B_G(R_1 \cup R_2, L)$, then we are done since $D_L$ needs to be half-thick as $R_3$ diverges from $R_1 \cup R_2$ but lies in the same end as~$R_1$ and~$R_2$. 
	
	Otherwise, again since $R_3$ diverges from $R_1 \cup R_2$, it has a tail in~$C_K$ for all $K \in \N$. We then distinguish two cases. First assume that $R_3$ is far away from $H$. Then, since $R_3$ has a tail in each $C_K$, we can connect~$R_3$ and $H$ with infinitely many pairwise disjoint paths. These paths together with $R_3$ then yield infinitely many $H$-paths that `jump over' $H$. We then use these paths together with $H$ to find an ultra fat model of~$K_{\aleph_0}$. Otherwise, $R_3$ lies close to~$H$. Then either~$R_3$ separates $H$ into an `upper half' containing (a tail of) $R_1$ and a `lower half' containing (a tail of) $R_2$, and then $R_1 \cup R_3$ (or symmetrically $R_2 \cup R_3$) is the desired double ray~$R'$, or there are infinitely many $H$-paths that `jump over' $R_3$, which then again yield an ultra fat model of~$K_{\aleph_0}$ (see \cref{lem:ThickAndHalfThickComponent}).

	\subsection{Obtaining fat and diverging minors from escaping subdivisions} \label{subsec:EscSubYieldsFatDivMinor}
	
	In this section we describe how one can turn an escaping subdivision $H$ of the hexagonal half-grid (full-grid) into a $K$-fat or diverging minor of the hexagonal half-grid (full-grid). 
	
	\begin{lem} \label{cor:EscHexGridYieldsAsympAndDivHexGrid}
		Let $G$ be a locally finite graph, and let $H \subseteq G$ be an escaping subdivision of the hexagonal half-grid (full-grid). Then the following assertions hold for all $K \in \N$: 
		\begin{enumerate}[label=\rm{(\roman*)}]
			\item \label{itm:EscHexGridYieldsAsympHexGrid} $H$ contains a subdivision of the hexagonal half-grid (full-grid) which is $K$-fat in~$G$, and
			\item \label{itm:EscHexGridYieldsDivHexGrid} $H$ contains a subdivision of the hexagonal half-grid (full-grid) which diverges in $G$.
		\end{enumerate}
	\end{lem}
	
	\noindent In fact, the subdivisions which we obtain from \cref{cor:EscHexGridYieldsAsympAndDivHexGrid}~\ref{itm:EscHexGridYieldsAsympHexGrid} and~\ref{itm:EscHexGridYieldsDivHexGrid} will have the property that their sets of vertical double rays are a subset of the sets of vertical double rays of $H$.
	\medskip
	
	Since we will have to delete some of the branch paths from $H$ in the proof of \cref{cor:EscHexGridYieldsAsympAndDivHexGrid}, we need the following auxiliary result.
	
	\begin{prop} \label{prop:HexGridAfterDeletingPaths}
		Let $H$ be a subdivision of the hexagonal half-grid (full-grid) with vertical double rays~$S^i$ and horizontal paths~$P_{ij}$. Let~$H'$ be obtained from~$H$ by deleting some of the~$P_{ij}$. If~$H'$ still contains, for every $i \in \N \setminus \{0\}$ ($i \in \Z\setminus\{0\}$), infinitely many $P_{ij}$ with $j \in \N$ and infinitely many $P_{ij}$ with $j \in \Z_{\leq 0}$, then~$H'$ contains a subdivision $H''$ of the hexagonal half-grid (full-grid) whose vertical double rays are the~$S^i$ and whose set of horizontal paths is a subset of the $P_{ij}$'s.
	\end{prop}
	
	\begin{proof}
		To obtain the desired graph $H''$, one may recursively select paths $P_{ij} \subseteq H'$ with sufficiently large~$|j|$ to represent the edges $f_k$ in the order indicated in \cref{fig:HalfHexGrid} (and similarly for the hexagonal full-grid in the order $f_1, f_2, f'_2, f_3, f'_3, f_4, f'_4, f_5, f_6, f_7, f'_7, \ldots$).
	\end{proof}

	\begin{figure}[ht]
		\begin{tikzpicture}
    \tikzset{edge/.style = {->,> = stealth}}
    
    \foreach \x in {0,1,2,3,4,5} {
    \draw[stealth-stealth] (\x,-0.4) to (\x,4.4);
    \foreach \y in {0,0.5,1,1.5,2,2.5,3,3.5,4} {
      \draw[fill,black] (\x,\y) circle (.05);
}}

\foreach \x in {-4,-3,-2,-1} {
    \draw[Gray, stealth-stealth] (\x,-0.4) to (\x,4.4);
    \foreach \y in {0,0.5,1,1.5,2,2.5,3,3.5,4} {
      \draw[fill,Gray] (\x,\y) circle (.05);
}}


\foreach \x in {1,3}{
    \foreach \y in {0.5,1.5,2.5,3.5} {
	\draw (\x,\y) to (\x+1,\y);
}}

\foreach \x in {-3,-1,}{
    \foreach \y in {0.5,1.5,2.5,3.5} {
	\draw[Gray] (\x,\y) to (\x+1,\y);
}}

\foreach \x in {0,2,4}{
    \foreach \y in {0,1,2,3,4} {
	\draw (\x,\y) to (\x+1,\y);
}}

\foreach \x in {-4,-2}{
    \foreach \y in {0,1,2,3,4} {
	\draw[Gray] (\x,\y) to (\x+1,\y);
}}

         
\foreach \y in {0.5,1.5,2.5,3.5} {
	\draw (5,\y) to (5.5,\y);
}

\foreach \y in {0.5,1.5,2.5,3.5} {
	\draw[Gray] (-4.5,\y) to (-4,\y);
}


    \node at (6,2) {$\dots$};
    \node[Gray] at (-5,2) {$\dots$};

\foreach \x in {0,1,2,3,4,5}{
    \node at (\x,-0.7) {\footnotesize{$R^{\x}$}};
     }
     
\foreach \x in {-4,-3,-2,-1}{
    \node[Gray] at (\x,-0.7) {\footnotesize{$R^{\x}$}};
     }


    \node at (0.5,-0.3) {\footnotesize{$f_{15}$}};
    \node at (0.5,0.7) {\footnotesize{$f_{6}$}};
    \node at (0.5,1.7) {\footnotesize{$f_{1}$}};
    \node at (0.5,2.7) {\footnotesize{$f_{5}$}};
    \node at (0.5,3.7) {\footnotesize{$f_{14}$}};

    \node at (1.5,0.2) {\footnotesize{$f_{13}$}};
    \node at (1.5,1.2) {\footnotesize{$f_{4}$}};
    \node at (1.5,2.2) {\footnotesize{$f_{3}$}};
    \node at (1.5,3.2) {\footnotesize{$f_{12}$}};

    \node at (2.5,0.7) {\footnotesize{$f_{11}$}};
    \node at (2.5,1.7) {\footnotesize{$f_{2}$}};
    \node at (2.5,2.7) {\footnotesize{$f_{10}$}};

    \node at (3.5,1.2) {\footnotesize{$f_{9}$}};
    \node at (3.5,2.2) {\footnotesize{$f_{8}$}};

    \node at (4.5,1.7) {\footnotesize{$f_{7}$}};

    \node[Gray] at (-1.5,1.7) {\footnotesize{$f'_{2}$}};
    \node[Gray] at (-0.5,1.2) {\footnotesize{$f'_{4}$}};
    \node[Gray] at (-0.5,2.2) {\footnotesize{$f'_{3}$}};
    \node[Gray] at (-3.5,1.7) {\footnotesize{$f'_{7}$}};
    \node[Gray] at (-2.5,2.2) {\footnotesize{$f'_{8}$}};
    \node[Gray] at (-2.5,1.2) {\footnotesize{$f'_{9}$}};
    \node[Gray] at (-1.5,2.7) {\footnotesize{$f'_{10}$}};
    \node[Gray] at (-1.5,0.7) {\footnotesize{$f'_{11}$}};
    \node[Gray] at (-0.5,3.2) {\footnotesize{$f'_{12}$}};
    \node[Gray] at (-0.5,0.2) {\footnotesize{$f'_{13}$}};
\end{tikzpicture}
		\vspace{-3em}
		\caption{The hexagonal half-grid (full-grid) with an enumeration of its horizontal edges as needed for the proof of \cref{prop:HexGridAfterDeletingPaths}.}
		\label{fig:HalfHexGrid}
	\end{figure}
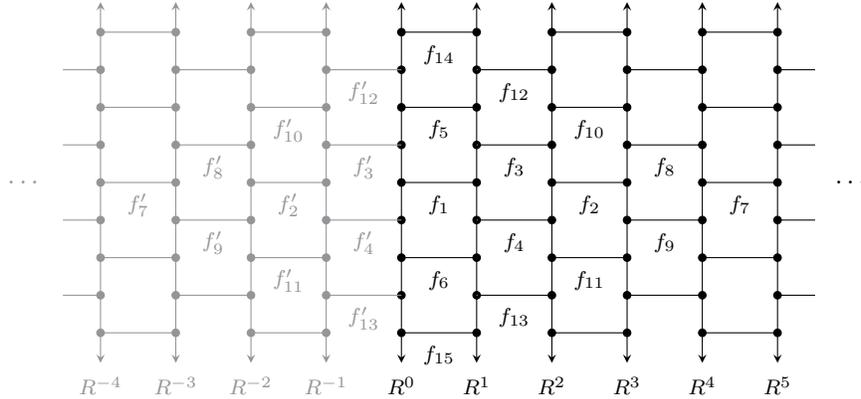

	For the proof of \cref{cor:EscHexGridYieldsAsympAndDivHexGrid}, we will use the next three lemmas, which show some useful additional properties of escaping subdivisions.
	
	\begin{lem} \label{lem:FurtherPropertiesOfEscapingHGs}
		Let $R = \dots r_{-1}r_0r_1 \dots$ be a diverging double ray in a locally finite graph $G$, let $L \in \N$, and let $S$ be a ray in $G[R, L]$. Suppose there are infinitely many pairwise disjoint $R_{\geq 0}$--$S$ paths $P_i$ in $G[R, L]$. Then $S$ has a tail $T$ such that $T \subseteq G[R_{\geq 0}, L]$, and all but finitely many $P_i$ are contained in $G[R_{\geq 0}, L]$.
	\end{lem}
	
	\begin{proof}
		Since $R$ diverges, there is some $n \in \N$ such that $R_{\geq n}$ and $R_{\leq -n}$ have distance at least $2L+2$ from each other; in particular, $G[R_{\geq n}, L]$, and $G[R_{\leq -n}, L]$ are disjoint and not joined by an edge. 
		Hence, they are separated in $G[R, L]$ by $B_G(r_{-n+1}Rr_{n-1}, L)$. Since $B_G(r_{-n+1}Rr_{n-1}, L)$ is finite as $G$ is locally finite, it follows that $S$ is eventually contained in either $G[R_{\geq n}, L]$ or $G[R_{\leq -n}, L]$ and that at most finitely many~$P_i$ meet $B_G(r_{-n+1}Rr_{n-1}, L)$. As the $P_i$'s are disjoint and start in $R_{\geq 0}$, and hence infinitely many $P_i$ start in $R_{\geq n}$, it follows that $S$ has a tail $T$ such that $T \subseteq G[R_{\geq n}, L] \subseteq G[R_{\geq 0}, L]$ and that all but finitely many $P_i$ lie in $G[R_{\geq 0}, L]$.
	\end{proof}

	\begin{cor} \label{cor:FurtherPropertiesOfEscapingHGs}
		Let $G$ be a locally finite graph and let $H$ be an escaping subdivision of the hexagonal half-grid with vertical double rays $S^i$ and horizontal paths $P_{ij}$ such that $S^0$ diverges. Then 
		\begin{enumerate}[label=\rm{(\roman*)}]
			\item \label{itm:DefEscHG:S^i:2} for all $i, k \in \N$ there is $\ell \in \N$ such that $S^i_{\geq \ell} \subseteq G[S^0_{\geq k}, M_i]$ and $S^i_{\leq -\ell} \subseteq G[S^0_{\leq -k}, M_i]$, and
			\item \label{itm:DefEscHG:P_ij:2} for all $i, k \in \N$ there is $\ell \in \N$ such that $P_{ij} \subseteq G[S^0_{\geq k}, M_i]$ and $P_{i(-j)} \subseteq G[S^0_{\leq -k}, M_i]$ for all $j \geq \ell$.
		\end{enumerate}
	\end{cor}
	
	\begin{proof}
		Let $i, k \in \N$ be given. Set $R := S^0$, where we enumerate $R = \dots r_{-1}r_0r_1\dots$ so that $r_0 = s^0_k$. 
		Applying \cref{lem:FurtherPropertiesOfEscapingHGs} to $R$, with $L := M_i$ and $S := S^i_{\geq 0}$ and the paths $P_j := P_{ij}$ for $j \geq 0$ yields some $m \in \N$ such that $S^i_{\geq m}$ and all $P_{ij}$ with $j \geq m$ are contained in $G[R_{\geq 0}, M_i] = G[S^0_{\geq k}, M_i]$. Similarly, we find some $n \in \N$ such that $S^i_{\leq -n}$ and all $P_{ij}$ with $j \leq -n$ are contained in $G[S^0_{\leq -k}, M_i]$. Then $\ell := \max\{m,n\}$ is as desired.
	\end{proof}

	The next lemma allows to find a diverging subdivision in an escaping subdivision of the hexagonal half- or full-grid.
	
	\begin{lem} \label{lem:DivSubHexGridOfEscHexGrid}
		Let $H$ be an escaping subdivision of the hexagonal half-grid (full-grid) in a locally finite graph $G$ with vertical double rays $S^i$ and horizontal paths $P_{ij}$. Then there exists for every $K\in\N$ an escaping subdivision $H' \subseteq H$ of the hexagonal half-grid (full-grid) whose vertical double rays are the~$S^i$ and whose horizontal paths $P'_{ij}$ are a subcollection of the $P_{ij}$'s such that $H'$ diverges and such that for every two non-incident edges of the hexagonal half-grid (full-grid) their images in $H'$ are at least $K$ apart in $G$ if they are contained in $H'_{K} := \left(\bigcup_{i \geq K} S^i\right) \cup \left(\bigcup_{i > K, j \in \Z} P'_{ij}\right)$.
	\end{lem}
	
	\begin{proof}
		We only give the proof for the hexagonal full-grid; the construction for the hexagonal half-grid is analogous.
		For the sake of this proof, we denote the horizontal edges $e_{ij}$ of the hexagonal full-grid by~$f_{i(2j)}$ if $i \in 2\Z+1$, and by $f_{i(2j-1)}$ if $i \in 2\Z$, and we enumerate the $P_{ij}$'s accordingly. Let $x^{i-1}_j, x^{i}_j$ denote the endvertices of $P_{ij}$ on $S^{i-1}$ and $S^i$, respectively.
		We will recursively select the branch paths $Q_{ij}$ of the edges~$f_{ij}$ amongst the $P_{k\ell}$'s such that 
		\begin{enumerate}[label=\rm{(\arabic*)}]
			\item \label{itm:DivHexGrid:1} $d_G(Q_{ij}, Q_{k\ell}) \geq \max\{|i|,|j|,|k|,|\ell|\}$ for all $i, j, k, \ell \in \Z$ such that $\{i,j\} \neq \{k, \ell\}$,
			\item \label{itm:DivHexGrid:2} $d_G(Q_{ij}, S^ky^k_{j-1} \cup y^k_{j+1}S^k) \geq \max\{|i|,|j|,|k|\}$ for all $i, j, k \in \Z$, and 
			\item \label{itm:DivHexGrid:3} $d_G(S^iy^i_{j-1}, y^{k}_{j}S^k) \geq \max\{|i|, |k|, |j|\}$ for all $i, j, k \in \Z$,
		\end{enumerate}
		where $y^{i-1}_{j}$ and $y^{i}_{j}$ denote the endvertices of $Q_{ij}$ on $S^{i-1}$ and $S^{i}$, respectively. Before explaining how our construction works, we show that the graph $H'$ defined by the union of the $S^i$'s and the $Q_{ij}$'s is an escaping subdivision satisfying the desired properties.
		Clearly, $H'$ is still an escaping subdivision of the hexagonal full-grid (whose horizontal paths $P'_{ij}$ are essentially the $Q_{ij}$'s, except that they are again enumerated as usual). It follows from \ref{itm:DivHexGrid:1}--\ref{itm:DivHexGrid:3} and \ref{itm:DefEscHG:S^i} of escaping subdivisions that $H'_K$ has the desired property. Moreover,~$H'$ diverges.
		Indeed, the distances in $G$ between the images $U_n, W_n$ in $H'$ of vertices and/or edges $a_n, b_n$ of the hexagonal full-grid which form diverging sequences $(a_n)_{n \in \N}, (b_n)_{n \in \N}$ in $H$ grow by \ref{itm:DivHexGrid:1}--\ref{itm:DivHexGrid:3} unless the $U_n$'s, $W_n$'s are of the form $Q_{ij}, y^k_{j-1}S^ky^k_{j}$ or $Q_{ij}, y^k_{j}S^ky^k_{j+1}$ or $y^i_{j-1}S^iy^i_{j}, y^k_{j-1}S^ky^k_{j}$. But their distances grow because of \ref{itm:DefEscHG:S^i} and \ref{itm:DefEscHG:P_ij} of escaping subdivisions; we omit the details.
		
		We now describe how we choose the paths $Q_{ij}$. First, we set $Q_{i0} := P_{i0}$ for all $i \in 2\N+1$ and $Q_{i0} := P_{(i-1)0}$ for $i \in 2\Z_{< 0}+1$. Note that it follows from \ref{itm:DefEscHG:S^i} and \ref{itm:DefEscHG:P_ij} of escaping subdivisions that the $Q_{i0}$'s satisfy \ref{itm:DivHexGrid:1}. For \ref{itm:DivHexGrid:2} and \ref{itm:DivHexGrid:3} there is nothing to check yet. Now let $n \in \N$ be given, and assume that we have already chosen paths~$Q_{ij}$ for all $|j| < n$. Let us now assume that $n$ is even. The case that $n$ is odd follows analogously and can be checked in parallel. We now describe how we choose the paths $Q_{in}$, where $i \in 2\Z+1$ since $n$ is even. The choice of the paths $Q_{i(-n)}$ can be done analogously after the choice of the $Q_{in}$'s.
		
		\begin{figure}[ht]
			\begin{tikzpicture}
    \tikzset{edge/.style = {->,> = stealth}}
    

    \begin{scope}[opacity=0.5]
		\foreach \x in{-6,-4,-2,0,2,4}{
			\draw[Gray, line width=6pt,line cap=round] (\x,1) to (\x+1,1);}
		\foreach \x in{-3,1,3}{
			\draw[Gray, line width=6pt,line cap=round] (\x,0.5) to (\x+1,0.5);}
		\foreach \x in{1}{
			\draw[Gray, line width=6pt,line cap=round] (\x,1.5) to (\x+1,1.5);}
		\foreach \x in{-3,-1,3}{
			\draw[Gray, line width=6pt,line cap=round] (\x,2.5) to (\x+1,2.5);}
		\draw[Gray, line width=6pt,line cap=round] (0,4) to (1,4);
		\draw[Gray, line width=6pt,line cap=round] (-5,1.5) to (-4,1.5);
		\draw[Gray, line width=6pt,line cap=round] (-5,-0.5) to (-4,-0.5);
		\draw[Gray, line width=6pt,line cap=round] (-1,-0.5) to (0,-0.5);
        \draw[Gray, line width=6pt,line cap=round] (6,1) to (6.5,1);
		\draw[Gray, line width=6pt,line cap=round] (5,1.5) to (6,1.5);
		\draw[Gray, line width=6pt,line cap=round] (5,-0.5) to (6,-0.5);
        \draw[Gray, line width=6pt,line cap=round] (-6,-0.5) to (-6.5,-0.5);
        \draw[Gray, line width=6pt,line cap=round] (-6,2.5) to (-6.5,2.5);
    \end{scope}

    \foreach \x in {-6,-5,5,6} {
    \draw[stealth-stealth] (\x,-0.9) to (\x,4.4);
}

\foreach \x in {-4} {
    \draw[-stealth] (\x,1.5) to (\x,4.4);
}

\foreach \x in {-3,-2,-1,4} {
    \draw[-stealth] (\x,2.5) to (\x,4.4);
}

\foreach \x in {0,1} {
    \draw[-stealth] (\x,4) to (\x,4.4);
}

    \draw (2,1.5) to (2,3);
    \draw (3,2.5) to (3,3);


\foreach \x in {1,3,5}{
    \foreach \y in {-0.5,0.5,1.5,2.5,3.5} {
	\draw (\x,\y) to (\x+1,\y);
}}

\foreach \x in {-5,-3,-1,}{
    \foreach \y in {-0.5,0.5,1.5,2.5,3.5} {
	\draw (\x,\y) to (\x+1,\y);
}}

\foreach \x in {0,2,4}{
    \foreach \y in {0,1,2,3,4} {
	\draw (\x,\y) to (\x+1,\y);
}}

\foreach \x in {-6,-4,-2}{
    \foreach \y in {0,1,2,3,4} {
	\draw (\x,\y) to (\x+1,\y);
}}

         
\foreach \y in {-0.5,0.5,1.5,2.5,3.5} {
	\draw (6,\y+0.5) to (6.5,\y+0.5);
}

\foreach \y in {-0.5,0.5,1.5,2.5,3.5} {
	\draw (-6.5,\y) to (-6,\y);
}


    \node at (7,2) {$\dots$};
    \node at (-7,2) {$\dots$};

	\foreach \x in{-4,-2,0,2}{
		\draw[blue, line width=1.4pt,line cap=round] (\x,1) to (\x+1,1);}
	\foreach \x in{-3,1,3}{
		\draw[blue, line width=1.4pt,line cap=round] (\x,0.5) to (\x+1,0.5);}
	\foreach \x in{1}{
		\draw[blue, line width=1.4pt,line cap=round] (\x,1.5) to (\x+1,1.5);}
	\foreach \x in{-3,-1,3}{
		\draw[blue, line width=1.4pt,line cap=round] (\x,2.5) to (\x+1,2.5);}
	\foreach \x in{-5,-1}{
		\draw[blue, line width=1.4pt,line cap=round] (\x,-0.5) to (\x+1,-0.5);}
	\draw[blue, line width=1.4pt,line cap=round] (0,4) to (1,4);
	\draw[blue, line width=1.4pt,line cap=round] (-5,1.5) to (-4,1.5);


\foreach \x in {-4,2} {
	\draw[-stealth, Green, line width=1.4pt] (\x,1.5) to (\x,-0.9);
}

\foreach \x in {-3,-2,-1,3,4} {
	\draw[-stealth, Green, line width=1.4pt] (\x,2.5) to (\x,-0.9);
}

\foreach \x in {0,1} {
    \draw[-stealth, Green, line width=1.4pt] (\x,4) to (\x,-0.9);
}

         
	\draw[Red, line width=1.4pt,line cap=round] (2,3) to (3,3);


\foreach \x in{2,3}{
\draw[-stealth, orange, line width=1.4pt] (\x,3) to (\x,4.4);
}


    \foreach \x in {-6,5,6} {
    \foreach \y in {-0.5,0,0.5,1,1.5,2,2.5,3,3.5,4} {
      \draw[fill,black] (\x,\y) circle (.05);
}}

    \foreach \x in {-5} {
    \foreach \y in {0,0.5,1,2,2.5,3,3.5,4} {
      \draw[fill,black] (\x,\y) circle (.05);
}}

    \foreach \x in {-4} {
    \foreach \y in {2,2.5,3,3.5,4} {
      \draw[fill,black] (\x,\y) circle (.05);
}}

    \foreach \x in {-3,-2,-1,4} {
    \foreach \y in {3,3.5,4} {
      \draw[fill,black] (\x,\y) circle (.05);
}}

      \draw[fill,black] (2,2) circle (.05);
      \draw[fill,black] (2,2.5) circle (.05);

\draw[fill,Red] (2,3) circle (.05);
\draw[fill,Red] (3,3) circle (.05);

\draw[fill,Orange] (2,3.5) circle (.05);
\draw[fill,Orange] (3,3.5) circle (.05);
\draw[fill,Orange] (2,4) circle (.05);
\draw[fill,Orange] (3,4) circle (.05);

    \foreach \y in {-0.5,1.5} {
      \draw[fill,blue] (-5,\y) circle (.05);
}

    \foreach \y in {-0.5,1,1.5} {
      \draw[fill,blue] (-4,\y) circle (.05);
}

    \foreach \y in {0.5,1,2.5} {
      \draw[fill,blue] (-3,\y) circle (.05);
}

    \foreach \y in {0.5,1,2.5} {
      \draw[fill,blue] (-2,\y) circle (.05);
}

    \foreach \y in {-0.5,1,2.5} {
      \draw[fill,blue] (-1,\y) circle (.05);
}

    \foreach \y in {-0.5,1,2.5,4} {
      \draw[fill,blue] (0,\y) circle (.05);
}

    \foreach \y in {0.5,1,1.5,4} {
      \draw[fill,blue] (1,\y) circle (.05);
}

    \foreach \y in {0.5,1,1.5} {
      \draw[fill,blue] (2,\y) circle (.05);
}

    \foreach \y in {0.5,1,2.5} {
      \draw[fill,blue] (3,\y) circle (.05);
}

    \foreach \y in {0.5,2.5} {
      \draw[fill,blue] (4,\y) circle (.05);
}

    \foreach \y in {0,0.5} {
      \draw[fill,Green] (-4,\y) circle (.05);
}

    \foreach \y in {-0.5,0,1.5,2} {
      \draw[fill,Green] (-3,\y) circle (.05);
}

    \foreach \y in {-0.5,0,1.5,2} {
      \draw[fill,Green] (-2,\y) circle (.05);
}

    \foreach \y in {0,0.5,1.5,2} {
      \draw[fill,Green] (-1,\y) circle (.05);
}

    \foreach \y in {0,0.5,1.5,2,3,3.5} {
      \draw[fill,Green] (0,\y) circle (.05);
}

    \foreach \y in {-0.5,0,2,2.5,3,3.5} {
      \draw[fill,Green] (1,\y) circle (.05);
}

    \foreach \y in {-0.5,0} {
      \draw[fill,Green] (2,\y) circle (.05);
}

    \foreach \y in {-0.5,0,1.5,2} {
      \draw[fill,Green] (3,\y) circle (.05);
}

    \foreach \y in {-0.5,0,1,1.5,2} {
      \draw[fill,Green] (4,\y) circle (.05);
}

\foreach \x in {0,1,2,3,4,5,6}{
    \node at (\x,-1.2) {\footnotesize{$S^{\x}$}};
     }
     
\foreach \x in {-6,-5,-4,-3,-2,-1}{
    \node at (\x,-1.2) {\footnotesize{$S^{\x}$}};
     }


    \draw[Gray] (0.5,0.7) node {\footnotesize{$Q_{10}$}};
    \draw[Gray] (2.5,0.7) node {\footnotesize{$Q_{30}$}};
    \draw[Gray] (-1.5,0.7) node {\footnotesize{$Q_{(-1)0}$}};
    \draw[Gray] (0.5,3.7) node {\footnotesize{$Q_{12}$}};
    \draw[Gray] (1.5,0.2) node {\footnotesize{$Q_{2(-1)}$}};
    \draw[Gray] (1.5,1.2) node {\footnotesize{$Q_{21}$}};
    \draw[Gray] (-0.5,2.2) node {\footnotesize{$Q_{01}$}};
    \draw[Gray] (-0.472,-0.8) node {\footnotesize{$Q_{0(-1)}$}};
    \draw[Gray] (3.5,2.2) node {\footnotesize{$Q_{41}$}};
    \draw[Gray] (3.5,0.2) node {\footnotesize{$Q_{4(-1)}$}};
    \draw[Red] (2.5,2.7) node {\footnotesize{$Q_{32}$}};

    \draw[blue] (-0.3,4) node {\footnotesize{$y^0_2$}};
    \draw[blue] (-0.3,1) node {\footnotesize{$y^0_0$}};
    \draw[blue] (0.3,2.5) node {\footnotesize{$y^0_1$}};
    \draw[blue] (0.4,-0.5) node {\footnotesize{$y^0_{-1}$}};

    \draw[red] (1.7,3) node {\footnotesize{$y^2_{2}$}};
    \draw[red] (3.3,3) node {\footnotesize{$y^3_{2}$}};

    \draw[orange] (3.5,4.2) node {\footnotesize{$y^3_{2}S^3$}};
    \draw[orange] (1.5,4.2) node {\footnotesize{$y^2_{2}S^2$}};

\end{tikzpicture}
			\vspace{-3em}
			\caption{Depicted in blue and green are the subgraphs $X_{32}$ and $Y_{32}$ that are used to choose $Q_{32}$. The paths $Q_{ij}$ that are chosen before $Q_{32}$ are shown in grey.}
			\label{fig:X32}
		\end{figure}
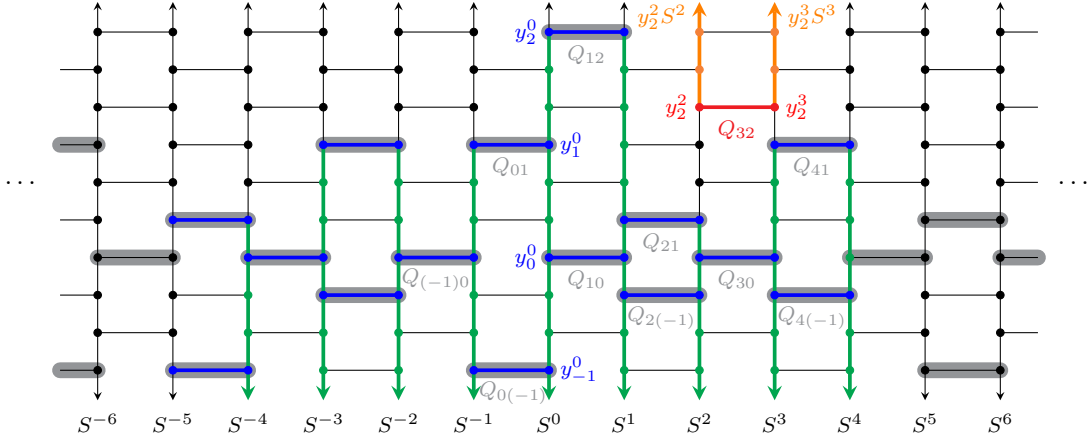
		
		So let $i \in 2\N+1$ be given, and assume that we have already chosen paths $Q_{kn}$ for all $k \in 2\N+1$ with ${k < i}$. We will now select a path $P_{ij}$ to be $Q_{in}$; again, the choice of the $Q_{(-i)n}$'s can be done analogously after the choice of the $Q_{in}$'s.
		
		It follows from \ref{itm:DefEscHG:S^i} and \ref{itm:DefEscHG:P_ij} of escaping subdivisions that the distance from a horizontal path $P_{k\ell}$ to some $S^i$ or $P_{ij}$ increases with the horizontal coordinate; more precisely we have $d_G(P_{ij}, P_{k\ell}), d_G(S^i, P_{k\ell}) \geq \max\{i, |k|\}$ for all $k \notin {\{i-1, i, i+1\}}$ and $j, \ell \in \Z$. Hence, no matter which $P_{ij}$ we choose to be $Q_{in}$, we will have that $d_G(Q_{in}, Q_{k\ell}) \geq \max\{i, n, |k|, |\ell|\}$ and $d_G(y^{i-1}_nS^{i-1}, Q_{k\ell}), d_G(y^i_nS^i, Q_{k\ell}) \geq \max\{i, n, |k|\}$ for all $k$ with $|k| \geq k_{in} := \max\{i+2, n\}$ and $|\ell| \leq n$. So when choosing~$Q_{in}$, we only need to consider those finitely many $Q_{k\ell}$ with $|k| < k_{in}$ and $|\ell| \leq n$. Since
		$X_{in} := \left(\bigcup_{|k| < k_{in}, |\ell| < n} Q_{k\ell}\right) \cup \left(\bigcup_{0< k < i} Q_{kn}\right)$ is finite and $G$ is locally finite, all but finitely many~$P_{ij}$ have the property that $x^{i-1}_jS^{i-1} \cup P_{ij} \cup x^{i}_jS^{i}$ has distance at least $\max\{i, n\}$ from $X_{in}$ (see \cref{fig:X32}). 
		
		Moreover, similar as above, it follows from properties \ref{itm:DefEscHG:S^i} and \ref{itm:DefEscHG:P_ij} of escaping subdivisions that\\ $d_G(P_{ij}, S^k), d_G(S^i, S^k)$ $\geq \max\{i, |k|\}$ for all $k \neq i$. Hence, no matter which $P_{ij}$ we choose to be $Q_{in}$, we will have that each of $d_G(Q_{in}, S^k)$, $d_G(y^i_nS^i, S^k)$, and $d_G(y^{i-1}_nS^{i-1}, S^k)$ is at least $\max\{i, n, |k|\}$ for all~$k$ with $|k| \geq \max\{i+1, n\}$ and $|\ell| \leq n$.
		Thus, we only need to choose~$Q_{in}$ in a way, so that it satisfies \ref{itm:DivHexGrid:2} and~\ref{itm:DivHexGrid:3} with respect to the finitely many $S^k$ with $|k| < \max\{i+1,n\}$. 
		By \cref{cor:FurtherPropertiesOfEscapingHGs}~\ref{itm:DefEscHG:S^i:2} and~\ref{itm:DefEscHG:P_ij:2} and because $S^0$ diverges, all but finitely many of the~$P_{ij}$ have the property that $x^{i-1}_nS^{i-1} \cup P_{ij} \cup x^i_{j}S^i$ has distance at least $\max\{i, n\}$ from $Y_{in} := \bigcup_{|k| < k_{in}} S^ky^k_{(n-1)}$.
		Hence, we can pick a path~$P_{ij}$ whose endvertices on~$S^{i-1}$ and~$S^i$ appear on $S^{i-1}$ and $S^{i}$ after the endvertices of $Q_{(i-1)(n-1)}$ and $Q_{(i+1)(n-1)}$, respectively, such that $x^{i-1}_jS^{i-1} \cup P_{ij} \cup x^{i}_jS^{i}$ has distance at least $\max\{i, n\}$ from $X_{in} \cup Y_{in}$, and we set $Q_{in} := P_{ij}$. Clearly,~$Q_{in}$ satisfies properties \ref{itm:DivHexGrid:1}--\ref{itm:DivHexGrid:3}.
	\end{proof}
	
	We are now ready to prove \cref{cor:EscHexGridYieldsAsympAndDivHexGrid}.
	
	\begin{proof}[Proof of \cref{cor:EscHexGridYieldsAsympAndDivHexGrid}]
		\ref{itm:EscHexGridYieldsDivHexGrid}: By \cref{lem:DivSubHexGridOfEscHexGrid} every escaping subdivision $H$ of the hexagonal half- or full-grid contains a diverging subdivision $H'$ of the hexagonal half- or full-grid, respectively, as a subgraph.
		
		\ref{itm:EscHexGridYieldsAsympHexGrid}: Assume first that $H$ is an escaping subdivision of the hexagonal half-grid, and let $H' \subseteq H$ be obtained from $H$ by applying \cref{lem:DivSubHexGridOfEscHexGrid}. Further, let $\widetilde{H}$ be obtained from $H'$ by deleting all horizontal paths that are not of the form $P'_{ij}$ for $i \in 2\Z$ and $j \in 3\Z$ or $i \in 2\Z+1$ and $j \in 3\Z+1$ (see \cref{fig:EscHexGridYieldsAsympAndDivHexGrid}). Then, for all $K\in\N$, the subgraph $\widetilde{H}_{K}$ of $\widetilde{H}$ consisting of all $S^i$ with $i \geq K$ and all $P'_{ij} \subseteq \widetilde{H}$ with $i > K$ is still a subdivision of the hexagonal half-grid, and we claim that it is $K$-fat. Indeed, to turn $\widetilde{H}_{K}$ into a $K$-fat model of the hexagonal half-grid we may choose the following sets $V_x$ as branch sets, for every $x \in V(\widetilde{H}_{K})$ of degree~$3$. Assume that $x \in V(S^i)$ for some $i\geq K$ and let $E_{e}$ and $E_{f}$ be the branch paths of $H'_{K} := \left(\bigcup_{i \geq K} S^i\right) \cup \left(\bigcup_{i > K, j \in \Z} P'_{ij}\right) \subseteq H'$ starting at~$x$ that are contained in $S^i$. We then set $V_x:=E_{e}\cup E_f$. By construction and since the images in $H'_{K}$ of any two non-incident edges of the hexagonal half-grid have distance at least $K$ in $G$, it follows that the model is $K$-fat.
		
		\begin{figure}[ht]
			\begin{tikzpicture}
    \tikzset{edge/.style = {->,> = stealth}}
    
    \foreach \x in {0,1,2,3,4,5} {
    \draw[stealth-stealth, line width=1.3] (\x,-0.4) to (\x,4.9);
}


\foreach \x in {1,3}{
    \foreach \y in {0.5,2.5,3.5} {
	\draw[Gray] (\x,\y) to (\x+1,\y);
}}

\foreach \x in {1,3}{
    \foreach \y in {1.5,4.5} {
	\draw[line width=1.3] (\x,\y) to (\x+1,\y);
}}

\foreach \x in {0,2,4}{
    \foreach \y in {1,2,4} {
	\draw[Gray] (\x,\y) to (\x+1,\y);
}}

\foreach \x in {0,2,4}{
    \foreach \y in {0,3} {
	\draw[line width=1.3] (\x,\y) to (\x+1,\y);
}}

         
\foreach \y in {0.5,2.5,3.5} {
	\draw[Gray] (5,\y) to (5.5,\y);
}

\foreach \y in {1.5,4.5} {
	\draw[line width=1.3] (5,\y) to (5.5,\y);
}


    \foreach \x in {0,1,2,3,4,5} {
    \foreach \y in {0.5,1,2,2.5,3.5,4} {
      \draw[fill,black] (\x,\y) circle (.05);
}}

    \foreach \x in {0,1,2,3,4,5} {
    \foreach \y in {0,1.5,3,4.5} {
      \draw[fill,black] (\x,\y) circle (.05);
}}


    \node at (6,3.5) {$\widetilde H$};

    \node at (6,2) {$\dots$};

    \node[Gray] at (6,0.5) {$H'$};

\end{tikzpicture}
			\vspace{-3em}
			\caption{The black subdivision $\widetilde{H}$ of the hexagonal half-grid is a subgraph of the grey subdivision $H'$ of the hexagonal half-grid.}
			\label{fig:EscHexGridYieldsAsympAndDivHexGrid}
		\end{figure}
		
		Second, assume that $H$ is an escaping subdivision of the hexagonal full-grid, let $H' \subseteq H$ be obtained from $H$ by applying \cref{lem:DivSubHexGridOfEscHexGrid}, and let $\widetilde{H} \subseteq H'$ be defined as above.
		Then the graph $\widetilde{H}_{K+1}$ defined as above is $K$-fat for every $K \in \N$ by the argument above. Similarly, it follows by (the symmetry of) the construction of $H'$ in the proof of \cref{lem:DivSubHexGridOfEscHexGrid} that also the subgraph $\widetilde{H}_{-K-1}$ consisting of all $S^i$ with $i \leq -K-1$ and all $P'_{ij} \subseteq \widetilde{H}$ with $i \leq -K-1$ is $K$-fat for all $K \in \N$. Since $G$ is locally finite, we then find infinitely many $S^{-K-1}$--$S^{K+1}$ paths $W_j$ in $\widetilde{H}$ which are pairwise at least $K$ apart. 
		By the assumptions on~$\widetilde{H}$ and since $\widetilde{H}$ is escaping, gluing the $W_i$'s with $\widetilde{H}_{K+1}$ and $\widetilde{H}_{-K-1}$ together yields (after possibly applying \cref{prop:HexGridAfterDeletingPaths}) a subdivision $H''$ of the hexagonal full-grid. By construction, $H''$ is $K$-fat. 
	\end{proof}

	\subsection{Proof of the main results given \texorpdfstring{\cref{thm:AsymptoticFullGrid,thm:HalfGrid:BoundedCycles}}{Theorems 3.3 and 3.4}} \label{subsec:FormalProofOfMainResults}
	
	In this section we derive \cref{main:HalfGrid:BoundedCycles,main:DivergingHalfGrid,main:AsymptoticFullGrid,main:DivergingFullGrid} from \cref{thm:AsymptoticFullGrid,thm:HalfGrid:BoundedCycles}; in fact, we show the following more detailed versions. 
	
	\begin{customthm}{Theorem~1$^\prime$} \label{thm:AsympFG}
		\emph{Let $G$ be a locally finite, \qt\ graph whose cycle space is generated by cycles of bounded length. 
			Then $FG \prec^\eps_\infty G$ for every thick end $\eps$ of $G$.}
	\end{customthm}
	
	\begin{proof}[Proof of \cref{main:AsymptoticFullGrid} and \cref{thm:AsympFG} given \cref{thm:AsymptoticFullGrid}]
		By \cref{obs:UFImpliesAsym} and \cref{cor:EscHexGridYieldsAsympAndDivHexGrid}~\ref{itm:EscHexGridYieldsAsympHexGrid}, \cref{thm:AsymptoticFullGrid} yields \cref{thm:AsympFG} and hence also \cref{main:AsymptoticFullGrid}, where we note that the subdivision obtained from \cref{cor:EscHexGridYieldsAsympAndDivHexGrid}~\ref{itm:EscHexGridYieldsAsympHexGrid} has all its rays in the same end as the full-grid obtained from \cref{thm:AsymptoticFullGrid}.
	\end{proof}

	\begin{customthm}{\cref*{main:DivergingFullGrid}$^\prime$} \label{thm:DivergingFG}
		\emph{Let $\eps$ be a thick end of a locally finite, \qt\ graph $G$ whose cycle space is generated by cycles of bounded length. Then $G$ contains a diverging subdivision of the hexagonal full-grid whose rays all lie in $\eps$.}
	\end{customthm}
	
	\begin{proof}[Proof of \cref{main:DivergingFullGrid} and \cref{thm:DivergingFG} given \cref{thm:AsymptoticFullGrid}]
		By \cref{obs:UFImpliesDiv} and \cref{cor:EscHexGridYieldsAsympAndDivHexGrid}~\ref{itm:EscHexGridYieldsDivHexGrid}, \cref{thm:AsymptoticFullGrid} yields \cref{thm:DivergingFG}, where we note that the subdivision obtained from \cref{cor:EscHexGridYieldsAsympAndDivHexGrid}~\ref{itm:EscHexGridYieldsDivHexGrid} has all its rays in the same end as the full-grid obtained from \cref{thm:AsymptoticFullGrid}.
		For \cref{main:DivergingFullGrid}, note that every diverging subdivision of the hexagonal full-grid yields a diverging model of the full-grid by choosing suitable branch sets and paths inside the subdivision.
	\end{proof}

	\begin{customthm}{Theorem~3$^\prime$} \label{thm:HG:BC}
		\emph{Let $G$ be graph of finite maximum degree whose cycle space is generated by cycles of bounded length. Then $HG \prec^\eps_\infty G$ for every thick end $\eps$ of~$G$.}
	\end{customthm}
	
	\begin{proof}[Proof of \cref{main:HalfGrid:BoundedCycles} and \cref{thm:HG:BC} given \cref{thm:HalfGrid:BoundedCycles}]
		By \cref{obs:UFImpliesAsym} and \cref{cor:EscHexGridYieldsAsympAndDivHexGrid}~\ref{itm:EscHexGridYieldsAsympHexGrid}, \cref{thm:HalfGrid:BoundedCycles} yields \cref{thm:HG:BC}, and hence also \cref{main:HalfGrid:BoundedCycles}, where we note that the subdivision obtained from \cref{cor:EscHexGridYieldsAsympAndDivHexGrid}~\ref{itm:EscHexGridYieldsAsympHexGrid} has all its rays in the same end as the half-grid obtained from \cref{thm:HalfGrid:BoundedCycles}.
	\end{proof}

	\begin{customthm}{Theorem~4$^\prime$} \label{thm:DivergingHG}
		\emph{Let $\eps$ be a thick end of a graph $G$ of finite maximum degree whose cycle space is generated by cycles of bounded length. Then $G$ contains a diverging subdivision of the hexagonal half-grid whose rays all lie in $\eps$.}
	\end{customthm}
	
	\begin{proof}[Proof of \cref{main:DivergingHalfGrid} and \cref{thm:DivergingHG} given \cref{thm:HalfGrid:BoundedCycles}]
		By \cref{obs:UFImpliesDiv} and \cref{cor:EscHexGridYieldsAsympAndDivHexGrid}~\ref{itm:EscHexGridYieldsDivHexGrid}, \cref{thm:HalfGrid:BoundedCycles} yields \cref{thm:DivergingHG}, where we note that the subdivision obtained from \cref{cor:EscHexGridYieldsAsympAndDivHexGrid}~\ref{itm:EscHexGridYieldsDivHexGrid} has all its rays in the same end as the half-grid obtained from \cref{thm:HalfGrid:BoundedCycles}.
		For \cref{thm:DivergingHG}, note that a diverging subdivision of the hexagonal half-grid yields a diverging model of the half-grid by choosing suitable branch sets and paths inside the subdivision.
	\end{proof}

	\section{Diverging double rays and quasi-geodesic 3-stars of rays in thick ends} \label{sec:DivAndGeodDR}
	
	In this section we prove two theorems about double rays and $3$-stars of rays in thick ends, which we need for the proofs of \cref{main:AsymptoticFullGrid,main:HalfGrid:BoundedCycles,main:DivergingFullGrid,main:DivergingHalfGrid}.

	\subsection{Diverging double rays} \label{subsec:DivergingDR}
	
	Georgakopoulos and Papasoglu \cite{GP2023+}*{Theorem~8.16} showed that every connected graph of finite maximum degree which has an infinite set of pairwise disjoint rays has a diverging double ray (whose tails may lie in two distinct ends).
	For the proofs of \cref{main:HalfGrid:BoundedCycles,main:DivergingHalfGrid} we will need the following variant of that theorem, which lets us find the diverging double ray in any thick end we like.
	
	\begin{thm} \label{thm:DivergingRays}
		Let $G$ be a graph of finite maximum degree, and let $\eps$ be a thick end of $G$. Then $G$ has a diverging double $\eps$-ray.
	\end{thm}
	
	The proof of \cref{thm:DivergingRays} uses the same idea as the one of \cite{GP2023+}*{Theorem~8.16} by Georgakopoulos and Papasoglu, in that \cref{cor:RaysOrBall} and \cref{lem:AgelosPanosLemmaForDivergingRays} below are variants of \cite{GP2023+}*{Corollary 8.15 and Lemma 8.17}. However, our proof is more involved, as we need to take care that the tails of the double ray lie in the prescribed end.
	
	Essentially, we will deduce \cref{thm:DivergingRays} from the following coarse Menger's theorem for two paths, which was proven independently by the first author, Huynh, Jacobs, Knappe and Wollan \cite{DistanceMengerForTwo}*{Theorem~1} and by Georgakopoulos and Papasoglu \cite{GP2023+}*{Theorem~8.1}; the version we state here is the one from the latter paper. 
	A \defn{metric graph} is a pair $(G, \ell)$ of a graph $G$ and an assignment of edge-lengths $\ell\colon E(G) \rightarrow \R_{> 0}$.
	
	\begin{thm} \label{thm:DistanceMenger2Paths:GP}
		Let $G$ be a metric graph, and let $X,Y \subseteq V(G)$. For every $K > 0$, there is either
		\begin{enumerate}[label=\rm{(\roman*)}]
			\item \label{itm:DistMenger2Paths:GP:ball} a set $B \subseteq V(G)$ of diameter $\leq K$ such that $G - B$ contains no path joining $X$ to $Y$, or 
			\item \label{itm:DistMenger2Paths:GP:paths} two $X$--$Y$ paths at distance at least $d := K/272$ from each other.
		\end{enumerate}
	\end{thm}
	
	In the proof of the next lemma, we use a compactness argument to show that we may replace in \cref{thm:DistanceMenger2Paths:GP} the set $Y$ by an end $\eps$ and the two paths in \ref{itm:DistMenger2Paths:GP:paths} by $\eps$-rays.
	
	\begin{lem} \label{cor:RaysOrBall}
		Let $(G, \ell)$ be a metric graph such that $B_{(G, \ell)}(v, n)$ is finite for all $v \in V(G)$ and $n \in \N$. Let~$A$ be a finite set of vertices in $G$, and let $\eps$ be an end of $G$. For every $K > 0$, there is either
		\begin{enumerate}[label=\rm{(\roman*)}]
			\item \label{itm:RaysOrBall:Ball} a set $B \subseteq V(G)$ of diameter $\leq K$ such that $G-B$ contains no $\eps$-ray starting at $A$, or 
			\item \label{itm:RaysOrBall:Rays} two $\eps$-rays starting at $A$ at distance at least $d := K/272$ from each other.
		\end{enumerate}
	\end{lem}
	
	We remark that the assumption that the balls $B_{(G,\ell)}(v,n)$ are finite is weaker than the assumption that the underlying graph is locally finite. We refer to this property as \defn{$(G,\ell)$ has finite balls}.
	
	\begin{proof}
		In the following, we abbreviate $d_{(G, \ell)}$ and $B_{(G, \ell)}$ with $d_G$ and $B_G$, respectively. In particular, all balls and distances in this proof are taken with respect to $\ell$.
		
		As $(G, \ell)$ has finite balls, and since~$A$ is finite, the balls $B_G(A, n)$, for $n \in \N$, are finite. Hence, there exists, for every $n \in \N$ a unique component~$C_n$ of $G - B_G(A,n)$ such that every $\eps$-ray has a tail in~$C_n$. Their neighbourhoods $N_G(C_n) \subseteq B_G(A,n)$ are finite; so we may set $k_n := |N_G(C_n)|$ and enumerate $N_G(C_n) =: \{v^n_1, \dots, v^n_{k_n}\}$.
		
		We apply \cref{thm:DistanceMenger2Paths:GP} to the sets $X := A$ and $Y := N_G(C_n)$ in~$G$. If, for some $n \in \N$, \cref{thm:DistanceMenger2Paths:GP} yields a set $B \subseteq V(G)$ of diameter $\leq K$ that separates $A$ and $N_G(C_n)$, then by the definition of $C_n$ this set $B$ satisfies~\ref{itm:DistMenger2Paths:GP:ball}. Hence, we may assume that for every $n \in \N$, we find two $A$--$N_G(C_n)$ paths $P_n, Q_n$ that are at least $d := K/272$ apart in $G$.
		
		For all $m \leq n \in \N$, we define a $k_m$-tuple 
		\[
		t^m_n = ((t^{nm}_1, \widetilde{t}^{nm}_1, s^{nm}_1), \dots, (t^{nm}_{k_m}, \widetilde{t}^{nm}_{k_m}, s^{nm}_{k_m})) \in T_m := (\{0, \dots, d\}^2 \times \{-4, \dots, 4\})^{k_m}
		\]
		of triples as follows. We let $t^{nm}_i$ be the distance $d_G(P_n, v_i^m)$ between $P_n$ and $v_i^m$ if it is less than~$d$; otherwise we set $t^{nm}_i := d$. Analogously, we let $\widetilde{t}^{nm}_i$ be the distance $d_G(Q_n, v_i^m)$ between $Q_n$ and $v_i^m$ if it is less than~$d$; otherwise we set $\widetilde{t}^{nm}_i := d$.
		Further, if $t^{nm}_i \neq 0$ and $\widetilde{t}^{nm}_i \neq 0$, then we set $s^{nm}_i := 0$. Otherwise, it follows that precisely one of $P_n$ and $Q_n$ meets $v^m_i$, and we then let $s^{nm}_i$ encode whether $v^m_i$ meets $P_n$ ($s^{nm}_i \in \{1,2,3,4\}$) or $Q_n$ ($s^{nm}_i \in \{-4,-3,-2,-1\}$) and whether its predecessor and successor on~$P_n$ or $Q_n$ both lie in $C_m$ ($|s^{nm}_i| = 1$), both lie in $G-C_m$ ($|s^{nm}_i| = 2$), or its predecessor lies in~$C_m$ and its successor lies in $G-C_m$  ($|s^{nm}_i| = 3$) or vice versa ($|s^{nm}_i| = 4$). 
		
		Since all $T_m$ are finite, there exist infinite index sets $\N \supseteq I_0 \supseteq I_1 \supseteq \ldots$ such that, for all $m \in \N$, all $t_n^m$ with $n \in I_m$ are equal. We pick, for every $m \in \N$, some $i_m \in I_m$.
		Now set 
		\[
		\widetilde{P}_{i_m} := P_{i_m} \cap (G[C_{m-1}, 1] - C_{m})\; \text{ and }\; \widetilde{Q}_{i_m} :=  Q_{i_m} \cap (G[C_{m-1}, 1] - C_{m}),
		\]
		and let $P := \bigcup_{m \in \N} \widetilde{P}_{i_m}$ and $Q := \bigcup_{m \in \N} \widetilde{Q}_{i_m}$.
		We claim that $P$ and $Q$ are at least $d$ apart in~$G$ and that they both contain an $\eps$-ray that starts in~$A$. It then follows that these rays are as in \ref{itm:RaysOrBall:Rays}.
		
		First, we show that $P$ and $Q$ are at least $d$ apart in~$G$. For this, recall that $d_G(P_{n}, Q_{n}) \geq d$ for all $n \in \N$ by the choice of $P_{n}, Q_{n}$. Now let $m \leq n \in \N$ be given. We show that $d_G(\widetilde{P}_{i_m}, \widetilde{Q}_{i_n}) \geq d$; the other case is symmetric. Clearly, if $m = n$, then $d_G(\widetilde{P}_{i_m}, \widetilde{Q}_{i_n}) \geq d$ holds by the choice of $P_{i_m}, Q_{i_n}$, so we may assume that $m < n$. Set $\ell := d_G(\widetilde{P}_{i_m}, \widetilde{Q}_{i_n})$, and let $W = w_0 \dots w_\ell$ be a shortest $\widetilde{P}_{i_m}$--$\widetilde{Q}_{i_n}$ path. Then $W$ meets $N_G(C_m)$ in a vertex $v_j^m$ because $\widetilde{Q}_{i_n} \subseteq G[C_{n-1},1] \subseteq G[C_m,1]$ and $\widetilde{P}_{i_m} \subseteq G-C_m$ as $m < n$. It follows that
		\begin{align*}
			\ell &= d_G(w_0, v_j^m) + d_G(v^m_j, w_\ell) \geq d_G(\widetilde{P}_{i_m}, v_j^m) + d_G(v^m_j, \widetilde{Q}_{i_n}) \geq d_G(P_{i_m}, v^m_j) + d_G(v^m_j, Q_{i_n}) \geq t^{i_mm}_j + \widetilde{t}^{i_nm}_j\\
			&= t^{i_nm}_j + \widetilde{t}^{i_nm}_j = \min\{d_G(P_{i_n}, v^m_j), d\} + \min\{d_G(v^m_j, Q_{i_n}), d\} \geq \min\{d_G(P_{i_n}, Q_{i_n}), d\} = d,
		\end{align*}
		where we used $t^{i_mm}_j = t^{i_nm}_j$ since $i_m, i_n \in I_m$. Hence, $d_G(P, Q) \geq d$ as desired.
		
		So to conclude the proof, it remains to show that $P$ and~$Q$ both contain an $\eps$-ray that starts at~$A$. We show the claim for~$P$; the other case is symmetric. By definition, it is clear that $P$ meets $A$ in a unique vertex~$a$, which is the endvertex of $P_{i_1}$ in~$A$; in particular, $a$ has degree~$1$ in~$P$. Hence, it suffices to show that all other vertices in $P$ have degree~$2$ in $P$, as then $P$ contains a ray that starts in $a$, and which then has to lie in $\eps$ since $P - C_m = \bigcup_{n \leq m} \widetilde{P}_{i_n}$ is finite for all $m \in \N$. 
		By definition of~$P$, every vertex of~$P$ that is not contained in some $N_G(C_m)$ is contained in precisely one $\widetilde{P}_{i_m}$, and has thus degree $2$ in~$P$.
		So let some $v^m_j \in V(P) \cap N_G(C_m)$ be given. Then $s^{i_mm}_j \in \{1, 2, 3, 4\}$, and it follows that $v^m_j$ has degree $2$ in~$G$ because $s^{i_mm}_j = s^{i_{m+1}m}_j$ by the choice of $i_m, i_{m+1}$.
	\end{proof}
	
	The remainder of the proof of \cref{thm:DivergingRays} is now analogous to the one of \cite{GP2023+}*{Theorem~8.16}. More precisely, we have the following auxiliary lemma, which is variant of \cite{GP2023+}*{Lemma 8.17}.
	
	\begin{lem} \label{lem:AgelosPanosLemmaForDivergingRays}
		Let $G$ be a graph of finite maximum degree, and let $\eps$ be a thick end of $G$. Then there is a finite set $A$ of vertices in $G$ and an assignment of edge-lengths $\ell\colon E(G) \rightarrow \R_{>0}$ with the following properties:
		\begin{enumerate}[label=\rm{(\roman*)}]
			\item no ball of radius $1$ in the corresponding metric $d_\ell$ separates $A$ from $\eps$,
			\item $\lim_{e\in E(G)} \ell(e) = 0$,\footnote{This means that for every $\eps > 0$ all but finitely many $e \in E(G)$ satisfy $\ell(e) < \eps$.} and 
			\item \label{itm:AgelosPanosLemma:3} every ball of finite radius in $d_\ell$ is finite.
		\end{enumerate} 
	\end{lem}
	
	\begin{proof}
		This follows immediately from the proof of \cite{GP2023+}*{Lemma~8.17} if we choose the sequence $(R_n)_{n \in \N}$ of pairwise disjoint rays so that every $R_n$ is an $\eps$-ray, which is possible because $\eps$ is thick. 
		
		Note that \ref{itm:AgelosPanosLemma:3} follows easily from the proof, since the $S^n$'s are `thickened rings' $B_G(o, r_n) \setminus B_G(o, r_{n-1})$ around a vertex~$o \in V(G)$ and because $\sum_{i \in \N} 1/n$ is infinite. 
	\end{proof}
	
	\begin{proof}[Proof of \cref{thm:DivergingRays}]
		The proof is analogous to \cite{GP2023+}*{Theorem~8.16} with just one exception: instead of \cite{GP2023+}*{Corollary~8.15 \& Lemma~8.17} we apply \cref{cor:RaysOrBall,lem:AgelosPanosLemmaForDivergingRays}.
	\end{proof}

	\subsection{Quasi-geodesic 3-stars of rays} \label{subsec:QuasiGeod3Star}
	
	By \cref{thm:DivergingRays} every graph $G$ of finite maximum degree contains for every thick end $\eps$ a diverging double $\eps$-ray. For the proofs of \cref{main:AsymptoticFullGrid,main:DivergingFullGrid}, we need the following result, which strengthens \cref{thm:DivergingRays} in the special case where $G$ is \qt\ and accessible. 
	
	\begin{thm} \label{lem:QuasiGeodesic3StarOfRays}
		Let $\eps$ be a thick end of a locally finite, accessible, \qt\ graph $G$.
		Then there exists $c \in \N_{\geq 1}$ and $\eps$-rays $R_1, R_2, R_3$ in~$G$ such that $R_1 \cap R_2 = R_1 \cap R_3 = R_2 \cap R_3 = \{v\}$ for some $v \in V(G)$ and such that $R_1 \cup R_2 \cup R_3$ is $c$-quasi-geodesic in~$G$.
	\end{thm}
	
	For the proof of \cref{lem:QuasiGeodesic3StarOfRays}, we first need the following auxiliary lemma.
	
	\begin{lem}\label{cor:1EndQTGeodSubgraph}
		Let $G$ be a \lf, accessible, \qt\ graph that contains a thick end~$\eps$. Then there exists a connected, one-ended, quasi-geodesic subgraph $H$ of~$G$ such that every ray in~$H$ is an $\eps$-ray in~$G$ and such that the stabilizer of~$H$ acts \qt ly on~$H$.    
	\end{lem}
	
	\begin{proof}
		By a result of Diestel, Jacobs, Knappe and Kurkofka \cite{DJKK}*{Lemma 7.12} and in particular its proof \cite{DJKKlong}*{Appendix A}, there exists a connected, induced, one-ended subgraph $H$ of~$G$ whose rays all lie in~$\eps$ such that every component of $G-H$ has finite neighbourhood in~$H$, such that there are only finitely many orbits of such components under the stabilizer $\Gamma$ of~$H$ in the automorphism group of~$G$ and such that $\Gamma$ acts \qt ly on~$H$.
		It remains to prove that $H$ is quasi-geodesic.
		Since there are only finitely many orbits of components of $G-H$ under~$\Gamma$, since each such component has finite neighbourhood in~$H$ and since $H$ is connected, there exists $c\in\N_{\geq 1}$ such that for every component $C$ of $G-H$ every two vertices in $N_G(C)$ have distance at most $c$ in~$H$.
		
		We claim that~$H$ is $c$-quasi-geodesic. Indeed, let $x, y \in V(H)$ be given, and let~$P$ be a shortest $x$--$y$ path in~$G$. 
		Further, let $Q_0, \dots, Q_m$ be the maximal non-trivial subpaths of~$P$ that are internally disjoint from~$H$. Then every~$Q_i$ is internally contained in some component~$C$ of $G-H$ and starts and ends in~$N_G(C)$.
		By the choice of~$c$, there exists a path~$Q'_i$ in~$H$ of length at most~$c$ which has the same endvertices as~$Q_i$. 
		It follows that the union~$W$ over $P \cap H$ and the~$Q'_m$ is connected, contained in~$H$, contains $x, y$, and hence contains an $x$--$y$ path.
		Since all~$Q_i$ are non-trivial and have thus length at least~$1$, it follows that $d_H(x, y) \leq |E(W)| \leq |E(P \cap H)| + c(m+1) \leq c\cdot |E(P)| = c\cdot d_G(x,y)$ as desired.
	\end{proof}
	
	We also need the following result of the first author, Diestel, Elm, Fluck, Jacobs, Knappe and Wollan~\cite{radialpathwidth}.
	
	\begin{lem}[\cite{radialpathwidth}*{Lemma 4.3}] \label{lem:CombiningGeodesicGraphs}
		Let $X$ be a $c$-quasi-geodesic subgraph of some graph~$G$ for some $c \in \N_{\geq 1}$.
		If~$P$ is a shortest $v$--$X$ path in~$G$ for some vertex~$v \in V(G)$, then $X \cup P$ is $(2c+1)$-quasi-geodesic in~$G$.
	\end{lem}
	
	We can now prove \cref{lem:QuasiGeodesic3StarOfRays}.
	
	\begin{proof}[Proof of \cref{lem:QuasiGeodesic3StarOfRays}]
		Let us first assume that $G$ is one-ended. In this case, we apply two compactness arguments.
		First, a standard compactness argument (see e.g.\ \cite{TW1993}*{Proposition~5.2}) implies the existence of a geodesic double ray~$R = \dots r_{-1}r_0r_1\dots$ in the \lf\ and \qt\ graph~$G$.
		
		For the second compactness argument, we first show that $G[R, K] \neq G$ for all $K \in \N$. Let $K \in \N$ be given. Since~$G$ is \lf, the set $B_G(r_1Rr_{2K+1}, K)$ is finite. As $R$ is geodesic, the sets $B_G(R_{\leq 0}, K)$ and $B_G(R_{\geq 2K+2}, K)$ are disjoint and not joined by an edge. Hence, every $B_G(R_{\leq 0}, K)$--$B_G(R_{\geq 2K+2}, K)$ path meets either $B_G(r_1Rr_{2K+1}, K)$ or $G - G[R, K]$. But since both $R_{\leq 0}$ and $R_{\geq 2K+1}$ lie in the unique end of $G$, there are infinitely many disjoint such paths, of which at most finitely many can meet the finite set $B_G(r_1Rr_{2K+1}, K)$. Hence, $G-G[R,K]$ is non-empty.
		
		Thus, there exist vertices in~$G$ of arbitrary distance from~$R$. 
		Let $x_i$ be a vertex at distance~$i$ from~$R$, let~$r_{j_i}$ be a vertex of~$R$ with $d_G(x_i,r_{j_i})=d_G(x_i,R)$, and let $P_i = p_0^i \dots p^i_i$ be a shortest $x_i$--$r_{j_i}$ path.
		Then $R\cup P_i$ is $3$-quasi-geodesic by \cref{lem:CombiningGeodesicGraphs}.
		Since $G$ is \qt, there is an infinite index set $I \subseteq \N$ such that all $r_{j_i}$ lie in the same orbit. Let $s \in V(G)$ be another vertex in that orbit. 
		For all $i \in I$, let~$\phi_i$ be an automorphism of~$G$ that maps $r_{j_i}$ to~$s$.
		Then, since $G$ is locally finite, there exists an infinite index set $I_1 \subseteq I$ such that $\phi_i(r_{j_i-1}Rr_{j_i+1} \cup p_0^iPp_1^i)$ coincides for all $i \in I_1$, amongst which we again find an infinite index set $I_2 \subseteq I_1$ such that $\phi_i(r_{j_i-2}Rr_{j_i+2} \cup p_0^iPp_2^i)$ coincides for all $i \in I_2$ and so on.
		This results in three internally disjoint, geodesic rays starting in $s$ whose union is $3$-quasi-geodesic. Obviously, all three rays must lie in the unique end~$\eps$ of~$G$, so they are as desired.
		
		Let us now assume that $G$ has more than one end.
		Since $G$ is accessible, there exists by \cref{cor:1EndQTGeodSubgraph} a connected, one-ended, $c$-quasi-geodesic, \qt\ subgraph $H$ of~$G$ for some $c\in \N_{\geq 1}$ such that every ray in~$H$ is an $\eps$-ray in~$G$.
		By the first case, we find the desired three rays $R_1, R_2, R_3$ in~$H$ whose union is $3$-quasi-geodesic.
		Since $H$ is a $c$-quasi-geodesic subgraph of~$G$, the union of the rays $R_1, R_2, R_3$ forms a $3c$-quasi-geodesic subgraph of~$G$.
	\end{proof}

	\section{Half-grid minors}\label{sec:HG:BC}
	
	In this section we prove \cref{thm:HalfGrid:BoundedCycles}; in fact, we show a more detailed version, which we need in the next section for the proof of \cref{thm:AsymptoticFullGrid}.
	
	Let $R = \dots r_{-1}r_0r_1\dots$ be a double ray in a graph $G$, and let $K \in \N$. A component $C$ of $G-B_G(R,K)$ is \defn{long} if $C$ has a neighbour in $B_G(R_{\geq i},K)$ and in $B_G(R_{\leq -i}, K)$ for all $i \in \N$. Further,~$C$ is \defn{thick} if, for every $L \geq K$, some long component of $G-B_G(R,L)$ is contained in~$C$.
	
	\begin{customthm}{\cref*{thm:HalfGrid:BoundedCycles}$^\prime$} \label{thm:HalfGrid:BoundedCycles:copy}
		\emph{Let $R$ be a diverging double ray in a thick end $\eps$ of a locally finite graph~$G$ whose cycle space is generated by cycles of bounded length. 
			Then either $K_{\aleph_0} \prec_{UF}^\eps G$ or $G$ contains an escaping subdivision $H$ of the hexagonal half-grid whose first vertical ray is $R$.}
		
		\emph{In particular, if $K_{\aleph_0} \not\prec^\eps_{UF} G$ and $C$ is a thick component of $G-B_G(R, L)$ for some $L \in \N$, then we may choose the vertical double rays $S^i$ of $H$ so that $S^i \subseteq C$ for all $i \geq 1$.}
	\end{customthm}
	
	\begin{proof}[Proof of \cref{thm:HalfGrid:BoundedCycles} given \cref{thm:HalfGrid:BoundedCycles:copy}]
		By \cref{thm:DivergingRays}, there exists a diverging double $\eps$-ray $R$ in $G$. Apply \cref{thm:HalfGrid:BoundedCycles} to $R$.
	\end{proof}
	
	In the remainder of this section we prove \cref{thm:HalfGrid:BoundedCycles:copy}; see \cref{subsec:ProofSketch34} for a sketch of the proof.
	
	\begin{lem} \label{lem:kappa/2NhoodIsConnected}
		Let $G$ be a graph whose cycle space is generated by cycles of length at most $\kappa \in \N$, and let~$Y$ be a connected subgraph of $G$. Then for every component $C$ of $G-Y$ that attaches to $Y$, the graph $C[\partial_G C, \lfloor \frac{\kappa-2}{2}\rfloor]$ is connected.
	\end{lem}
	
	Note that if $G$ is connected, then every component of $G-Y$ attaches to $Y$. We remark that $C[\partial_G C, \lfloor \frac{\kappa-2}{2}\rfloor] = C[B_G(\partial_G C, \lfloor \frac{\kappa-2}{2}\rfloor) \cap V(C)]$.
	
	\begin{proof}
		Clearly, it suffices to show for every two vertices $v_0, v_1 \in \partial_G C$ that there exists a $v_0$--$v_1$ path in~$C[\partial_G C, \lfloor\frac{\kappa-2}{2}\rfloor]$.
		So let $v_0,v_1 \in \partial_G C$ be given, and let~$u_0$ and~$u_1$ be vertices of~$Y$ which are adjacent to~$v_0$ and~$v_1$, respectively. Since~$Y$ is connected, there exists a $u_1$--$u_0$ path $Q$ in~$Y$. Let $P$ be a $v_0$--$v_1$ path in~$C$. Then $D := v_0Pv_1u_1Qu_0v_0$ is a cycle in~$G$.
		By the assumption that the cycle space of~$G$ is generated by cycles of bounded length, we can write~$D$ as a finite sum of cycles $D_1,\ldots, D_n$ in $G$ of length at most~$\kappa$, i.e.\ 
		\[
		D=\sum_{D_i \in \cD} D_i
		\]
		where $\cD := \{D_1, \dots, D_n\}$. 
		Let $\cD' \sub \{D_1,\ldots,D_n\}$ consist of those $D_i$ that do not lie completely in~$C$, i.e.\ that contain a vertex of $G - C$. Note that $D_i \cap C \subseteq C[\partial_G C, \lfloor\frac{\kappa-2}{2}\rfloor]$ for all $D_i \in \cD'$ since $D_i$ has length at most~$\kappa$ and meets $G-C$.
		Let 
		\[
		H := \left(\bigcup_{D_i \in \cD'} D_i\right) \cap C \subseteq C\left[\partial_G C, \left\lfloor\frac{\kappa-2}{2}\right\rfloor\right] \subseteq C
		\]
		be the subgraph of $C$ consisting of all vertices and edges in~$C$ that lie on cycles from~$\cD'$. 
		Note that $v_0, v_1 \in V(H)$ since $v_0u_0, v_1u_1 \in E(D)$. We claim that $v_0$ and $v_1$ lie in the same component of $H$, which clearly yields that there is a $v_0$--$v_1$ path in~$C[\partial_G C, \lfloor\frac{\kappa-2}{2}\rfloor]$, and hence concludes the proof.
		So suppose for a contradiction that~$v_0$ and~$v_1$ lie in distinct components~$H_0$, $H_1$ of~$H$.
		Then the set~$F$ of edges in $G$ between $H_0$ and $G-H_0$ is a cut in $G$  that separates $H_0$ and $H_1$; in particular,~$F$ is finite since $G$ is locally finite and because $H_0 \subseteq H$ is finite as $\bigcup_{D_i \in \cD'} D_i \supseteq H_0$ is a finite union of finite cycles.
		Since $P$ connects a vertex from $H_0$ with a vertex outside of~$H_0$, the cut $F$ must contain an edge $f$ from $P \subseteq C$. Then $f$ cannot lie in $\sum_{D_i \in \cD'} D_i \subseteq \bigcup_{D_i \in \cD'} D_i$, since $f \in E(C)$ but $f \notin E(H)$.
		Hence, as $f \in E(P) \subseteq E(D)$, it lies in 
		\begin{equation} \label{eq:NEW}
			H' := D+\sum_{D_i \in \cD'} D_i = \sum_{D_i \in \cD} D_i + \sum_{D_i \in \cD'} D_i = \sum_{D_i \in \cD\setminus \cD'} D_i \subseteq C,
		\end{equation}
		where for the last inclusion we used that $D_i \subseteq C$ for all $D_i \in \cD\setminus \cD'$ by the choice of $\cD'$. In particular, the same argument also yields that $E(P) \cap F \subseteq E(H')$.
		
		As $H'$ is a finite sum of cycles in $G$, it is an element of the cycle space of $G$. Thus, $H'$ meets the finite cut $F$ in an even number of edges. As $P$ is a finite path from $v_0 \in V(H_0)$ to $v_1 \in V(H_1) \subseteq V(G-H_0)$, it meets the finite cut $F$ in an odd number of edges. Combining these two facts with $E(P) \cap F \subseteq E(H')$ yields that $H'$ contains an edge $f' \neq f$ from $F$ which does not lie on~$P$.
		Since $H' \subseteq C$, the edge $f'$ must lie in~$C$. But since $f'$ is not an edge of $P = D \cap C$, it is not an edge of $D$ either. 
		Hence, as $f' \in E(H')$ and by \eqref{eq:NEW}, $f'$ is an edge of $\sum_{D_i \in \cD'} D_i$, and thus an edge of $\bigcup_{D_i \in \cD'} D_i$. Since $f'$ is also an edge of $C$, it lies in~$H$, which is a contradiction to the choice of~$F$.
	\end{proof}

	\begin{lem} \label{lem:DivDoubleRayDoesNotCoverEverything}
		Let $R = \dots r_{-1}r_0r_1 \dots$ be a diverging double ray in an end $\eps$ of a locally finite graph~$G$. Then for every $K, n \in \N$ some component of $G - B_G(R, K)$ attaches to $B_G(R_{\leq -n}, K)$ and $B_G(R_{\geq n}, K)$.
	\end{lem}
	
	\begin{proof}
		Since~$R$ diverges, there exists some $m \in \N$ such that $R_{\leq -m}$ and $R_{\geq m}$ are at least $2K+2$ apart in~$G$. Set $N := \max\{n, m\}$.
		As~$G$ is locally finite, the set $B_G(r_{-N}Rr_N, K)$ is finite. Hence, as~$R_{\leq -N}$ and~$R_{\geq N}$ are both $\eps$-rays and thus equivalent, there exists an $R_{\leq -N}$--$R_{\geq N}$ path $P = p_0 \dots p_\ell$ in~$G$ that avoids $B_G(r_{-N}Rr_N,K)$. 
		Since $p_0 \in V(R_{\leq -N})$ and $p_\ell \in V(R_{\geq N})$, there is a first vertex~$p_i$ of~$P$ that is contained in $B_G(R_{\geq N}, K)$, and a last vertex~$p_j$ with $j \leq i$ that is still contained in $B_G(R_{\leq -N}, K)$.
		
		We claim that $i \geq j+2$, which then implies that $P' := p_{j+1}Pp_{i-1}$ is non-empty. As~$P$ avoids $B_G(r_{-N}Rr_N, K)$ and by the choice of~$p_i$ and~$p_j$, it then follows that~$P'$ is contained in a component of $G-B_G(R,K)$, which then attaches to $B_G(R_{\leq -N}, K)$ and $B_G(R_{\geq N}, K)$ via $p_{j}p_{j+1}$ and $p_{i-1}p_i$, respectively, and which is thus as desired.
		
		So suppose for a contradiction that $i - j \leq 1$. Then $d_G(p_j, p_i) \leq i - j \leq 1$, and thus
		\[
		d_G(R_{\leq -N}, R_{\geq N}) \leq d_G(R_{\leq -N}, p_{j}) + d_G(p_j, p_i) + d_G(p_i, R_{\geq N}) \leq K + 1 + K = 2K+1,
		\]
		which is a contradiction since $d_G(R_{\leq -N}, R_{\geq N}) \geq 2K+2$ by the choice of $N$.
	\end{proof}

	\begin{lem} \label{lem:LongCompsExist}
		Let $R$ be a diverging double ray in a thick end of a locally finite graph $G$ whose cycle space is generated by cycles of bounded length. Then for every $K \in \N$ some component of $G-B_G(R,K)$ is long.
	\end{lem}
	
	\begin{proof}
		Let $\kappa \in \N$ be such that the cycle space of $G$ is generated by cycles of length at most $\kappa$.
		Suppose for a contradiction that no component of $G-B_G(R, K)$ is long.
		Since $R$ diverges, there exists some $N \in \N$ such that $d_G(R_{\leq -N}, R_{\geq N}) \geq 2K + \kappa + 2$.
		As $G$ is locally finite, $B_G(r_{-N}Rr_N, K)$ is finite. Hence, the set $\cC$ of components of $G-B_G(R, K)$ which attach to $B_G(r_{-N}Rr_N, K)$ is finite. 
		Since no $C \in \cC$ is long by assumption and because $\cC$ is finite, there exists some $m \in \N$ such that $N_G(C) \subseteq B_G(R_{\geq -m}, K)$ or $N_G(C) \subseteq B_G(R_{\leq m}, K)$ for all $C \in \cC$.
		
		By \cref{lem:DivDoubleRayDoesNotCoverEverything}, some component $C$ of $G - B_G(R, K)$ attaches to $B_G(R_{\leq -m-1}, K)$ and $B_G(R_{\geq m+1}, K)$; in particular, $C \notin \cC$ by the choice of $m$.
		Let $U^-, U^+ \subseteq \partial_G C$ be the set of vertices in $C$ that send an edge to $B_G(R_{\leq -N},K)$ or to $B_G(R_{\geq N}, K)$, respectively. Then $U^{-} \cup U^+ = \partial_G C$ because $C \notin \cC$. Since $C[\partial_G C, \lfloor\frac{\kappa-2}{2}\rfloor]$ is connected by \cref{lem:kappa/2NhoodIsConnected}, this implies that $B_C(U^-, \lfloor\frac{\kappa-2}{2}\rfloor)$ and $B_C(U^+, \lfloor\frac{\kappa-2}{2}\rfloor)$ either intersect non-emptily or there is an edge between them. Hence, there are vertices $u^- \in U^-$ and $u^+ \in U^+$ of distance at most $\lfloor\frac{\kappa-2}{2}\rfloor + 1 + \lfloor\frac{\kappa-2}{2}\rfloor$ from each other. Thus,
		\begin{equation*}
			\begin{aligned}
				d_G(R_{-N}, R_N) &\leq d_G(R_{-N}, u^-) + d_G(u^-, u^+) + d_G(u^+, R_N)\\
				&\leq (K+1) + \left(\left\lfloor\frac{\kappa-2}{2}\right\rfloor + 1 + \left\lfloor\frac{\kappa-2}{2}\right\rfloor\right) + (K+1)\\
				&\leq 2K + \kappa + 1
			\end{aligned}
		\end{equation*}
		which is a contradiction since $d_G(R_{-N}, R_N) \geq 2K + \kappa + 2$ by the choice of $N$.
	\end{proof}

	\begin{lem} \label{lem:HalfGrid}
		Let $\eps$ be an end of a locally finite graph~$G$. Suppose there are $M_0 < M_1 < \ldots \in \N$ and double $\eps$-rays $S^0, S^1, \dots$ such that $S^0$ diverges, such that $S^i \subseteq G[S^0, M_i] - B_G(S^0, M_{i-1})$ for all $i \in \N_{\geq 1}$ and such that there are infinitely many disjoint $S^0_{\geq 0}$--$S^i_{\geq 0}$ paths and infinitely many disjoint $S^0_{\leq 0}$--$S^i_{\leq 0}$ paths in $G[S^0, M_i]$.
		Then either $K_{\aleph_0} \prec_{UF}^{\eps} G$, or there are $0 = i_0 < i_1 < \ldots \in \N$ and an escaping subdivision $H$ of the hexagonal half-grid whose vertical double rays are the~$S^{i_j}$.
	\end{lem}
	
	\begin{proof}
		By passing to a subsequence of the $S^i$'s if necessary, we may assume that $M_i > M_{i-1}+2i$ and that
		\begin{enumerate}[label=(\alph*)]
			\item \label{itm:HG:a} $S^i \subseteq G[S^0, M_i] - B_G(S^0, M_{i-1} + 2i)$ for all $i \in \N$.
		\end{enumerate}
		
		Set $T'_0 := S^0_{\geq 0}$ and $T''_0 := S^0_{\leq 0}$. By assumption and \cref{lem:FurtherPropertiesOfEscapingHGs}, every~$S^i$ has disjoint tails~$T'_i$ and~$T''_i$ that are contained in $G[T'_0, M_i]$ and in $G[T''_0, M_i]$, respectively.
		For each vertex $t$ in $T'_i$ we choose a shortest $t$--$T'_0$ path in $G$, which then has length $\leq M_i$ and lies in $G[T'_0, M_i]$. Then infinitely many of these paths are $T'_i$--$T'_0$ paths (i.e.\ they only have their first vertex on $T'_i$), of which infinitely many are pairwise disjoint since they have length $\leq M_i$ and because~$G$ is locally finite; let us denote these paths by~$Q_{ij}$.
		
		For every $Q_{ij}$ let $k_{ij} \neq i$ be maximal such that $d_G(Q_{ij}, T'_{k_{ij}}) < k_{ij}$; if no such $k_{ij}$ exists, we set $k_{ij} := 0$. Note that $k_{ij} < i$ since $d_G(Q_{ij}, T'_k) \geq d_G(Q_{ij}, S^k) \geq k$ for all $k > i$ by \ref{itm:HG:a} and because $Q_{ij} \subseteq G[T'_0, M_i]$ for all $i, j \in \N$.
		We now obtain $T'_i$--$T'_{k_{ij}}$ paths~$Q'_{ij}$ by concatenating a suitable (initial) subpath of~$Q_{ij}$ with a shortest $Q_{ij}$--$T'_{k_{ij}}$ path. In particular, since $G$ is locally finite, we may assume that the $Q'_{ij}$'s for every (arbitrary but fixed) $i \in \N_{\geq 1}$ are pairwise disjoint, by passing to a subfamily of the $Q'_{ij}$'s. 
		By the definition of the~$k_{ij}$ and by \ref{itm:HG:a}, it follows that
		\begin{enumerate}[label=(\alph*)]
			\setcounter{enumi}{1}
			\item \label{itm:HG:c} $d_G(Q'_{ij}, T'_k) \geq k$ for all $k, j \in \N$ and $i \in \N_{\geq 1}$ with $k \notin \{i, k_{ij}\}$.
		\end{enumerate}
		Moreover, since $Q_{ij}$ is a shortest path between its first vertex and $T'_0$, it follows that once $Q_{ij}$ meets $G[T'_0, M_{k_{ij}-1}+i]$ it will stay in there. By the definition of $Q'_{ij}$ and $k_{ij}$ and by \ref{itm:HG:a}, this implies that
		\begin{enumerate}[label=(\alph*)]
			\setcounter{enumi}{2}
			\item \label{itm:HG:b} $Q'_{ij} \subseteq G[T'_0, M_i] - B_G(T'_0, M_{k_{ij}-1}+ k_{ij})$ for all $i \in \N_{\geq 1}$ and $j \in \N$.
		\end{enumerate}
		
		Let $X$ be the auxiliary graph on the vertex set $\{T'_i \mid i \in \N\}$ where $T'_i$ and $T'_{i'}$ are connected by an edge in~$X$ for $i' < i$ if and only if infinitely many of the $Q'_{ij}$'s have one endvertex on $T'_{i'}$. Clearly, every $T'_i$ is adjacent to at least one $T'_{i'}$ with $i' < i$, and hence $X$ is connected. Thus, since $X$ is infinite, it either has a vertex of infinite degree or it contains a ray by \cref{lem:StarComb}.
		
		Let us first assume that there is some $\ell \in \N$ and an infinite subset $I = \{i_0, i_1, \dots\} \subseteq \N$ with $i_0 < i_1 < \ldots$ such that $T'_\ell$ is adjacent in $X$ to all $T'_i$ with $i \in I$. Then the $T'_i$'s for $i \in I$ form the branch sets $V_{n} := V(T'_{i_n})$ of an ultra fat model of $K_{\aleph_0}$ in~$G$. Indeed, we have $d_G(V_n, V_m) \geq \min\{n,m\}$ by \ref{itm:HG:a}, so it remains to find suitable branch paths. 
		Given any enumeration of $\N^2$, we may choose the branch paths~$P_{nm}$ between~$V_n$ and~$V_m$ recursively. Since $G$ is locally finite, and because at step $(n,m)$ we have only chosen finitely many branch paths~$P_{n'm'}$, there exist paths $Q'_{i_nj}$, $Q'_{i_mj'}$ that both end in $T'_{\ell}$ such that the path~$P_{nm}$ consisting of $Q'_{i_nj}, Q'_{i_mj'}$ and a suitable subpath of~$T'_\ell$ is at least $\min\{n,m\}$ apart from all earlier chosen branch paths~$P_{n'm'}$. Then by construction and \ref{itm:HG:a} and~\ref{itm:HG:c} it follows that also $d_G(P_{nm}, V_k) \geq k$ for all $k \notin \{n,m\}$, and hence the model of $K_{\aleph_0}$ is ultra fat. Moreover, since its branch sets are the vertex sets of the $\eps$-rays $T'_{i_n}$, we find $K_{\aleph_0} \prec_{UF}^\eps G$.
		
		Hence, we may assume that there are $0 = i_0 < i_1 < \ldots \in \N$ such that every $T'_{i_n}$ is adjacent in $X$ to~$T'_{i_{n-1}}$. 
		Then there are, for every $n \in \N$, infinitely many $Q'_{i_nj}$ that end in~$T'_{i_{n-1}}$. 
		We reindex these $Q'_{i_nj}$'s by $\N_{\geq 1} \times \N$, and the $S^{i_n}$'s by~$\N$.
		
		We now apply the same argument to the tails $T''_i$ of the (reindexed) $S^i$. This either yields $K_{\aleph_0} \prec_{UF}^\eps G$, or we find indices $0 = i_0 < i_1 < \ldots \in \N$ such that, for every $n \in \N$, there are pairwise disjoint $T''_{i_{n-1}}$--$T''_{i_n}$ paths $Q''_{i_nj}$ that satisfy \ref{itm:HG:c} and \ref{itm:HG:b} with $T''_k$ instead of~$T'_k$.
		Now the $S^{i_n}$'s form the vertical double rays of an escaping subdivision of the hexagonal half-grid.
		Indeed, we can choose for every $n \in \N$ infinitely many $T'_{i_{n-1}}$--$T'_{i_n}$ paths~$P'_{i_nj}$ in $\bigcup_{i_{n-1} < k \leq i_n} (T'_k \cup \bigcup_{j \in \N} Q'_{kj})$. 
		We also set $P'_{i_n(-j)} := Q''_{i_nj}$ for all $j \in \N$.
		Then combining \ref{itm:HG:b} of the $Q'_{i_nj}$'s and the $Q''_{i_nj}$'s with \cref{lem:FurtherPropertiesOfEscapingHGs} yields that
		\begin{enumerate}[label=\rm{(\alph*$^\prime$)}]
			\setcounter{enumi}{2}
			\item \label{itm:HG:c2} $P'_{i_nj} \subseteq G[S^0, M_{i_n}] - B_G(S^0, M_{i_n-2}+n)$
		\end{enumerate}
		for every $n \in \N_{\geq 1}$ and all but finitely $j \in \N$.
		Since $G$ is locally finite, we now obtain a subdivision of the hexagonal half-grid with vertical double rays the $S^{i_n}$'s by recursively selecting paths $P_{nj}$ amongst the $P'_{i_nk}$'s to represent the horizontal edges $e_{nj}$ as in the proof of \cref{prop:HexGridAfterDeletingPaths}. 
	\end{proof}
	
	We are now ready to prove \cref{thm:HalfGrid:BoundedCycles:copy}.
	
	\begin{proof}[Proof of \cref{thm:HalfGrid:BoundedCycles:copy}]
		Let $\kappa \in \N$ such that the cycle space of $G$ is generated by cycles of length at most~$\kappa$. Let $N_n$, for $n \in \N$, be such that $d_G(R_{\geq N_n}, R_{\leq -N_n}) > n$, which exists since $R$ diverges.
		
		By \cref{lem:HalfGrid}, it suffices to show that there are $M_0 < M_1 < \ldots \in \N$ and double $\eps$-rays $R := S^0, S^1, \dots$ such that $S^0$ is diverging, such that $S^i \subseteq G[S^0, M_i] - B_G(S^0, M_{i-1})$ for all $i \in \N_{\geq 1}$ and such that there are infinitely many disjoint $S^0_{\geq 0}$--$S^i_{\geq 0}$ paths and infinitely many disjoint $S^0_{\leq 0}$--$S^i_{\leq 0}$ paths in $G[S^0, M_i]$. 
		We will prove the assertion with $M_0 := 0$ and $M_i := M_0 + \lfloor \kappa/2 \rfloor$ for all $i > 0$. 
		
		Let $i > 0$ be given.
		By \cref{lem:LongCompsExist}, there exists a long component~$C_i$ of $G-B_G(R, M_{i-1})$. For the `in particular' part we note that if we are given some $L \in \N$ and a thick component $C$ of $G-B_G(R, L)$, then we may set $M_0 := L$ instead of $M_0 := 0$ and choose as $C_i$ always a long component of $G-B_G(R, M_{i-1})$ which is contained in~$C$.
		
		Set $U^+ := \partial_G C_i \cap B_G(R_{\geq 0}, M_{i-1}+1)$ and $U^- := \partial_G C_i \cap B_G(R_{\leq 0}, M_{i-1}+1)$. Since~$G$ is locally finite and~$C_i$ is long,~$U^+$ and~$U^-$ are infinite. As $C_i[\partial_G C_i, \lfloor \frac{\kappa-2}{2}\rfloor]$ is connected by \cref{lem:kappa/2NhoodIsConnected}, applying the Star-Comb Lemma (cf.\ \cref{lem:StarComb}) in $C_i[\partial_G C_i, \lfloor \frac{\kappa-2}{2}\rfloor]$ to $U^+$ and $U^-$, respectively, yields two combs $D^+$ and $D^-$. 
		By \cref{lem:FurtherPropertiesOfEscapingHGs}, their spines $S^+$ and $S^-$ are eventually contained in $G[R_{\geq N_{M_i}}, M_i]$ and $G[R_{\leq -N_{M_i}}, M_i]$, respectively, i.e.\ they have tails $T^+, T^-$ such that $T^+ \subseteq G[R_{\geq N_{M_i}}, M_i]$ and $T^- \subseteq G[R_{\leq -N_{M_i}}, M_i]$. In particular, $T^+, T^-$ are disjoint by the choice of $N_{M_i}$, so we can link them by a path in the connected $C_i[\partial_G C_i, \lfloor \frac{\kappa-2}{2}\rfloor]$ to obtain a double ray $S^i \subseteq C_i[\partial_G C_i, \lfloor \frac{\kappa-2}{2}\rfloor]$. Clearly, $S^i$ is as desired. Indeed, the infinitely many $S^0_{\geq 0}$--$S^i_{\geq 0}$ paths can be obtained by extending the paths in $D^+$ from $T^+$ to its teeth by shortest paths to $S^0$, and analogously for~$T^-$.
	\end{proof}

	\section{Full-grid minors}\label{sec:FG}
	
	In this section we prove \cref{thm:AsymptoticFullGrid}, which we restate here for convenience. 
	
	\begin{customthm}{\cref{thm:AsymptoticFullGrid}} \label{thm:AsymptoticFullGrid:copy}
		\emph{Let $\eps$ be a thick end of a \lf, \qt\ graph $G$ whose cycle space is generated by cycles of bounded length. Then either $K_{\aleph_0} \prec_{UF}^\eps G$ or $G$ contains an escaping subdivision of the hexagonal full-grid whose rays all lie in $\eps$.}
	\end{customthm}
	
	In fact, we will prove the following variant of \cref{thm:AsymptoticFullGrid:copy}, which implies \cref{thm:AsymptoticFullGrid:copy}:
	
	\begin{thm} \label{thm:AsymptoticFullGridWeaker}
		Let $G$ be a locally finite, \qt\ graph whose cycle space is generated by cycles of bounded length. 
		If $G$ has a thick end, then either $K_{\aleph_0} \prec_{UF} G$ or $G$ contains an escaping subdivision of the hexagonal full-grid. 
	\end{thm}
	
	We remark that instead of adding \cref{thm:AsymptoticFullGridWeaker} as an intermediate step in the proof of \cref{thm:AsymptoticFullGrid:copy} we could have also formulated the three lemmas below which we use to construct either an ultra fat model of $K_{\aleph_0}$ or an escaping subdivision of the hexagonal full-grid (\cref{lem:TwoThickComponents,lem:ThickAndHalfThickComponent,lem:TwoThickCompsYieldFullGrid}) so that we may choose a thick end $\eps$, and the lemma then returns the desired structure `in' $\eps$. However, while this would have been possible in \cref{lem:TwoThickCompsYieldFullGrid,lem:ThickAndHalfThickComponent} without changing their proofs, this is not true for \cref{lem:TwoThickComponents}. There, we would then have to use the fact that $G$ is accessible, to reduce the problem to one-ended graphs. 
	
	\begin{proof}[Proof of \cref{thm:AsymptoticFullGrid} given \cref{thm:AsymptoticFullGridWeaker}]
		By \cref{thm:CycleSpaceAccessible}, $G$ is accessible. Hence, by \cref{cor:1EndQTGeodSubgraph}, there exists a connected, one-ended, $c$-quasi-geodesic, \qt\ subgraph $X$ of $G$ for some $c \in \N$ such that every ray in~$X$ is an $\eps$-ray in~$G$. 
		Applying \cref{thm:AsymptoticFullGridWeaker} to $X$ yields either an ultra fat model $((V_i)_{i \in \N}, (E_{ij})_{i \neq j \in \N})$ of $K_{\aleph_0}$ in $X$ or an escaping subdivision $H$ of the hexagonal full-grid in $X$. Since $X$ is a $c$-quasi-geodesic subgraph of~$G$, we find in the former case that $((V_i)_{i \in c\N}, (E_{ij})_{i \neq j \in c\N})$ is an ultra fat model of $K_{\aleph_0}$ in $G$. In fact, since every ray in $X$ is an $\eps$-ray in $G$, we have $K_{\aleph_0} \prec_{UF}^\eps G$. Similarly, in the latter case, $H$ contains a subdivision of the hexagonal full-grid which is escaping in $G$. Indeed, we may choose vertical double rays $\dots, S^{i_{-1}}, S^{i_0}, S^{i_1}, \dots$ of $H$ such that $i_0 = 0$ and $M_{i_j-1} + 2i_j \geq c(M_{i_{j-1}} + 2j)$ for $j > 0$, and similarly $M_{i_j+1} + 2|i_j| \geq c(M_{i_{j+1}} + 2|j|)$ for $j < 0$. Then adding suitable $S^{i_j}$--$S^{i_{j+1}}$ paths in~$H$ yields a subdivision $H' \subseteq H$ of the hexagonal full-grid which is escaping in $G$ since $X$ is a $c$-quasi-geodesic subgraph of $G$. Moreover, all rays in $H'$ are $\eps$-rays.
	\end{proof}
	
	In the remainder of this section we prove \cref{thm:AsymptoticFullGridWeaker}. The formal proof of \cref{thm:AsymptoticFullGridWeaker}, which collects the tools from this whole section, can be found at the end of the last subsection, \cref{subsec:TwoThickComps}.
	
	We first give a brief overview of this section; a more detailed sketch of the proof of \cref{thm:AsymptoticFullGridWeaker} can be found in \cref{subsec:ProofSketch12}.
	Let $G$ be a locally finite, \qt\ graph whose cycle space is generated by cycles of bounded length and which has a thick end. 
	Further, let $R = \dots r_{-1}r_0r_1\dots$ be a double ray, and let $K \in \N$. Recall that a component $C$ of $G-B_G(R,K)$ is \defn{long} if $C$ has a neighbour in $B_G(R_{\geq i},K)$ and in $B_G(R_{\leq -i}, K)$ for all $i \in \N$. 
	Further,~$C$ is \defn{thick} if, for every $L \geq K$, some long component of $G-B_G(R,L)$ is contained in~$C$.
	
	In \cref{lem:TwoThickCompsYieldFullGrid} below, we show that if $G$ contains a diverging double ray $R$ such that, for some $L \in \N$, $G-B_G(R, L)$ has at least two thick components, then either $K_{\aleph_0} \prec_{UF} G$ or $G$ contains an escaping subdivision of the hexagonal full-grid.
	Our remaining task then is to prove that $G$ indeed contains such a double ray $R$. Showing this will be the main effort of this proof, and it will be done in \cref{subsec:TwoThickComps} (see \cref{lem:TwoThickComponents,lem:ThickAndHalfThickComponent}). For this, in \cref{subsec:HGWithCrosses}, we provide with \cref{lem:KFatHexGridWithJumpingPaths} a sufficient condition for $G$ to contain $K_{\aleph_0}$ as an ultra fat minor, which enables us to find an ultra fat $K_{\aleph_0}$ minor in $G$ if we cannot find such a double ray~$R$. 
	
	\begin{lem} \label{lem:TwoThickCompsYieldFullGrid}
		Let $R$ be a diverging double ray in a locally finite graph $G$ whose cycle space is generated by cycles of bounded length. Suppose that for some $L \in \N$, there are at least two thick components of $G-B_G(R, L)$. Then either $K_{\aleph_0} \prec_{UF} G$, or $G$ contains an escaping subdivision of the hexagonal full-grid.
	\end{lem}
	
	\begin{proof}
		Let $C \neq D$ be two distinct thick components of $G-B_G(R, L)$. 
		Since we are done if $K_{\aleph_0} \prec_{UF} G$, we may assume that applying \cref{thm:HalfGrid:BoundedCycles:copy} to~$R$, $K$ and $C$ or $D$, respectively, yields escaping subdivisions~$H_C$ and~$H_D$, respectively, of the hexagonal half-grid.
		Let $M^C_0 < M_1^C < \ldots$ and $M^D_0 < M^D_1 < \ldots$ witness that~$H^C$ and~$H^D$ are escaping. Further, let $S^C_i, S^D_i$ and $P^C_{ij}, P^D_{ij}$ be the vertical double rays and horizontal paths of~$H_C$ and~$H_D$, respectively. 
		By the `in particular' part of \cref{thm:HalfGrid:BoundedCycles:copy}, it follows that the $S^C_i$'s and the $S^D_i$'s are contained in $C$ and $D$, respectively; in particular, they are disjoint. Moreover, by property~\ref{itm:DefEscHG:P_ij} of escaping subdivisions, we have that for $M := \max\{M^C_1+2, M^D_1+2\}$ the paths $P^C_{ij}$ and $P^D_{ij}$ with $i \geq M$ are contained in $C$ and $D$, respectively, and are hence disjoint from each other. Let $H'_C \subseteq H_C$ be a subdivision of the hexagonal half-grid with vertical double rays $S^C_0$ and $S^C_i$ for $i \geq M$, which we may obtain by choosing as the new branch paths for the horizontal edges $e_{1j}$ infinitely many disjoint $S^C_0$--$S^C_M$ paths $Q^C_{j}$ in~$H'_C$. 
		Let $H'_D$ be chosen analogously.
		Clearly, $H'_C$ and $H'_D$ are still escaping.
		
		Now since $S^C_0 = R = S^D_0$, gluing $H'_C$ and $H'_D$ together along $R$ yields a graph $H'$ which is nearly as desired except that the paths $Q^C_{j}$ and $Q^D_{\ell}$ may intersect. But since $G$ is locally finite, we can delete some of the $Q^C_{j}$'s and $Q^D_{\ell}$'s and apply \cref{prop:HexGridAfterDeletingPaths} to obtain a subdivision $H \subseteq H'$ of the hexagonal full-grid. By construction, $H$ is escaping. 
	\end{proof}

	\subsection{Half-grids with crosses} \label{subsec:HGWithCrosses}
	
	In this section we establish a sufficient condition which ensures that a graph $G$ contains $K_{\aleph_0}$ as an ultra fat minor. This condition essentially requires an escaping subdivision $H \subseteq G$ of the hexagonal half-grid, and infinitely many $H$-paths in~$G$ that `jump over' the vertical double rays in~$H$.
	For this, we first need the following two auxiliary statements about~$K_{\aleph_0}$ minors in half-grids with certain additional edges.
	
	\begin{lem} \label{lem:HalfGridWithCrosses}
		Let $G$ be obtained from the half-grid by adding all edges of the form $(i,0)(i+1, 1)$ and $(i,1)(i+1,0)$ for $i \in \N$.
		Then $G$ contains a model $(\cV, \cE)$ of $K_{\aleph_0}$ such that $V_i, V(E_{ij}) \subseteq \N_{\geq i} \times \Z$ for all $i < j \in \N$ and such that the $E_{ij}$'s are pairwise disjoint.
	\end{lem}
	
	\begin{proof}
		One can easily construct the $K_{\aleph_0}$ minor recursively starting from $K_1$ with branch set $V_1 := \{(0,0)\}$. For this, assume that we have already defined a model of $K_n$ with branch sets $V_1, \dots, V_n$, branch paths~$E_{ij}$ and integers $j_n, \ell_n \in \N$ such that $V_i, E_{ij} \subseteq \{i, \dots, \ell_n\} \times \{-j_n, \dots, 0, \dots, j_n\}$ and such that either for each $i\in\N$, the branch set $V_i$ meet $\{i, \ell_n\} \times \{-j_n\}$ or for each $i\in\N$, the branch set $V_i$ meet $\{i, \ell_n\} \times \{j_n\}$ and such that the $E_{ij}$'s are pairwise disjoint. We can then extend the branch sets $V_1, \dots, V_n$ and add a new branch set~$V_{n+1}$ as well as new branch paths $E_{(n+1)j}$ as depicted in \cref{fig:KAleph0MinorInHGWithCrosses} such that $V_1, \dots, V_{n+1}$ are the branch sets and the $E_{ij}$'s are the branch paths of a $K_{n+1}$ minor.
	\end{proof}
	
	\begin{figure}[ht]
		\centering
		\scalebox{1.2}{
\begin{tikzpicture}
    \foreach \y in {0.5,1,2.5,3,3.5,4}
    \draw[-stealth] (0,\y) -- (7,\y);
    \foreach \y in {1.5,2}
    \draw[-stealth] (1.5,\y) -- (7,\y);

    \foreach \y in {0,1.5,2,2.5,3,3.5,4,4.5,5,5.5,6,6.5}
    \draw[stealth-stealth] (\y,0) -- (\y,4.5);
    \foreach \y in {0.5,1}
    \draw[stealth-] (\y,0) -- (\y,1);
    \foreach \y in {0.5,1}
    \draw[-stealth] (\y,2.5) -- (\y,4.5);

	\draw[Black] (1.5,1.5) -- (1.675,1.675);
	\draw[Black] (1.825,1.825) arc (45:225:0.10606601717);
	\draw[Black] (1.825,1.825) -- (2,2);
	\draw[Black] (2,1.5) -- (1.5,2);

	\draw[Black] (2,1.5) -- (2.175,1.675);
	\draw[Black] (2.325,1.825) arc (45:225:0.10606601717);
	\draw[Black] (2.325,1.825) -- (2.5,2);
	\draw[Black] (2.5,1.5) -- (2,2);

	\draw[Black] (2.5,1.5) -- (2.675,1.675);
	\draw[Black] (2.825,1.825) arc (45:225:0.10606601717);
	\draw[Black] (2.825,1.825) -- (3,2);
	\draw[Black] (3,1.5) -- (2.5,2);

	\draw[Black] (3,1.5) -- (3.175,1.675);
	\draw[Black] (3.325,1.825) arc (45:225:0.10606601717);
	\draw[Black] (3.325,1.825) -- (3.5,2);
	\draw[Black] (3.5,1.5) -- (3,2);

	\draw[Black] (3.5,1.5) -- (3.675,1.675);
	\draw[Black] (3.825,1.825) arc (45:225:0.10606601717);
	\draw[Black] (3.825,1.825) -- (4,2);
	\draw[Black] (4,1.5) -- (3.5,2);

	\draw[Black] (4,1.5) -- (4.175,1.675);
	\draw[Black] (4.325,1.825) arc (45:225:0.10606601717);
	\draw[Black] (4.325,1.825) -- (4.5,2);
	\draw[Black] (4.5,1.5) -- (4,2);

	\draw[Black] (4.5,1.5) -- (4.675,1.675);
	\draw[Black] (4.825,1.825) arc (45:225:0.10606601717);
	\draw[Black] (4.825,1.825) -- (5,2);
	\draw[Black] (5,1.5) -- (4.5,2);

	\draw[Black] (5,1.5) -- (5.175,1.675);
	\draw[Black] (5.325,1.825) arc (45:225:0.10606601717);
	\draw[Black] (5.325,1.825) -- (5.5,2);
	\draw[Black] (5.5,1.5) -- (5,2);

 	\draw[Black] (5.5,1.5) -- (5.675,1.675);
	\draw[Black] (5.825,1.825) arc (45:225:0.10606601717);
	\draw[Black] (5.825,1.825) -- (6,2);
	\draw[Black] (6,1.5) -- (5.5,2);

 	\draw[Black] (6,1.5) -- (6.175,1.675);
	\draw[Black] (6.325,1.825) arc (45:225:0.10606601717);
	\draw[Black] (6.325,1.825) -- (6.5,2);
	\draw[Black] (6.5,1.5) -- (6,2);

  	\draw[Black] (6.5,1.5) -- (6.675,1.675);
	\draw[Black] (6.825,1.825) arc (45:225:0.10606601717);
	\draw[Black] (6.825,1.825) -- (6.9,1.9);
	\draw[Black] (6.9,1.6) -- (6.5,2);

	\draw[pattern=north east lines, pattern color=black] (0,1) rectangle (1.5,2.5);
	\fill[white] (0.5,1.5) rectangle (1,2);
	\draw[black] (0.75,1.75) node {$K_4$};

    \fill[Mulberry] (0,2.5) circle (2pt);
	\draw[Mulberry, line width=1.5,line cap=round] (0,2.5) -- (0,4);
	\draw[Mulberry, line width=1.5,line cap=round] (0,4) -- (5,4);
	\draw[Mulberry, line width=1.5,line cap=round] (5,4) -- (5,2);
	\draw[Mulberry, line width=1.5,line cap=round] (5,2) -- (5.5,1.5);
	\draw[Mulberry, line width=1.5,line cap=round] (5.5,1.5) -- (5.5,0.5);
    \fill[Mulberry] (5.5,0.5) circle (2pt);
	\draw[Mulberry] (0.25,2.75) node {\footnotesize{$V_1$}};

    \fill[blue] (0.5,2.5) circle (2pt);
	\draw[blue, line width=1.5,line cap=round] (0.5,2.5) -- (0.5,3.5);
	\draw[blue, line width=1.5,line cap=round] (0.5,3.5) -- (4,3.5);
	\draw[blue, line width=1.5,line cap=round] (4,3.5) -- (4,2);
	\draw[blue, line width=1.5,line cap=round] (4,2) -- (4.5,1.5);
	\draw[blue, line width=1.5,line cap=round] (4.5,1.5) -- (4.5,0.5);
    \fill[blue] (4.5,0.5) circle (2pt);
	\draw[blue] (0.75,2.75) node {\footnotesize{$V_2$}};

    \fill[BurntOrange] (1,2.5) circle (2pt);
	\draw[BurntOrange, line width=1.5,line cap=round] (1,2.5) -- (1,3);
	\draw[BurntOrange, line width=1.5,line cap=round] (1,3) -- (3,3);
	\draw[BurntOrange, line width=1.5,line cap=round] (3,3) -- (3,2);
	\draw[BurntOrange, line width=1.5,line cap=round] (3,2) -- (3.5,1.5);
	\draw[BurntOrange, line width=1.5,line cap=round] (3.5,1.5) -- (3.5,0.5);
    \fill[BurntOrange] (3.5,0.5) circle (2pt);
	\draw[BurntOrange] (1.25,2.75) node {\footnotesize{$V_3$}};

    \fill[Green] (1.5,2.5) circle (2pt);
	\draw[Green, line width=1.5,line cap=round] (1.5,2.5) -- (2,2.5);
	\draw[Green, line width=1.5,line cap=round] (2,2.5) -- (2,2);
	\draw[Green, line width=1.5,line cap=round] (2,2) -- (2.5,1.5);
	\draw[Green, line width=1.5,line cap=round] (2.5,1.5) -- (2.5,0.5);
    \fill[Green] (2.5,0.5) circle (2pt);
	\draw[Green] (1.75,2.75) node {\footnotesize{$V_4$}};

    \fill[Red] (2,0.5) circle (2pt);
    \draw[Red, line width=1.5,line cap=round] (2,0.5) -- (2,1.5);
	\draw[Red, line width=1.5,line cap=round] (2,1.5) -- (2.175,1.675);
	\draw[Red, line width=1.5,line cap=round] (2.325,1.825) arc (45:225:0.10606601717);
	\draw[Red, line width=1.5,line cap=round] (2.325,1.825) -- (2.5,2);
	\draw[Red, line width=1.5,line cap=round] (2.5,2) -- (3,1.5);
	\draw[Red, line width=1.5,line cap=round] (3,1.5) -- (3.175,1.675);
	\draw[Red, line width=1.5,line cap=round] (3.325,1.825) arc (45:225:0.10606601717);
	\draw[Red, line width=1.5,line cap=round] (3.325,1.825) -- (3.5,2);
	\draw[Red, line width=1.5,line cap=round] (3.5,2) -- (4,1.5);
	\draw[Red, line width=1.5,line cap=round] (4,1.5) -- (4.175,1.675);
	\draw[Red, line width=1.5,line cap=round] (4.325,1.825) arc (45:225:0.10606601717);
	\draw[Red, line width=1.5,line cap=round] (4.325,1.825) -- (4.5,2);
	\draw[Red, line width=1.5,line cap=round] (4.5,2) -- (5,1.5);
	\draw[Red, line width=1.5,line cap=round] (5,1.5) -- (5.175,1.675);
	\draw[Red, line width=1.5,line cap=round] (5.325,1.825) arc (45:225:0.10606601717);
	\draw[Red, line width=1.5,line cap=round] (5.325,1.825) -- (5.5,2);
	\draw[Red, line width=1.5,line cap=round] (5.5,2) -- (6,1.5);
    \draw[Red] (2.25,0.75) node {\footnotesize{$V_5$}};
    \fill[Red] (6,1.5) circle (2pt);

	\draw[Black, line width=1.2,line cap=round] (2,1.5) -- (2.5,1.5);
	\draw[Black, line width=1.2,line cap=round] (3,1.5) -- (3.5,1.5);
	\draw[Black, line width=1.2,line cap=round] (4,1.5) -- (4.5,1.5);
	\draw[Black, line width=1.2,line cap=round] (5,1.5) -- (5.5,1.5);

	\draw[black,rounded corners,dashed] (-0.25,0.25) rectangle (6.25,4.25);
	\draw[black] (-0.75,1.75) node {$K_{5}$};

\end{tikzpicture}}
		\vspace{-3em}
		\caption{Sketch of a $K_5$ minor in a half-grid with all $(i,0)(i+1,1)$ and $(i,1)(i+1,0)$ edges. The branch paths $E_{ij}$ between $V_5$ and the $V_i$'s are the thickened $(i,0)(i+1,0)$ edges.}
		\label{fig:KAleph0MinorInHGWithCrosses}
	\end{figure}
	
	In particular, we have the following corollary:
	
	\begin{cor} \label{cor:HalfGridWithJumpingEdges:NEW}
		For every $n\in\N$, let $(i_n, k_n),(j_n, \ell_n)\in\N\times\N$ such that $\max\{i_n,j_n\}<\min\{i_{n'},j_{n'}\}$ and $|i_n - j_n| \geq 2$ for all $n < n'$.
		Let $X$ be obtained from the half-grid by adding a new edge $f_n=(i_n, k_n)(j_n, \ell_n)$ for every $n \in \N$.
		Then $X$ contains a model $(V, \cE)$ of $K_{\aleph_0}$ such that $V_i, V(E_{ij}) \subseteq \N_{\geq i} \times \Z$ for all $i < j \in \N$ and such that the $E_{ij}$'s are pairwise disjoint.
	\end{cor}
	
	\begin{proof}
		It is straight forward to check that the lines $\{i_n\} \times \Z$ in $X$ form the vertical double rays of a subdivision $X'$ of the graph $G$ in the premise of \cref{lem:HalfGridWithCrosses} such that every branch path in $X'$ corresponding to an edge of $G$ between $\{n\} \times \Z$ and $\{n+1\} \times \Z$ is contained in $X[\{i_n, \dots, i_{n+1}\} \times \Z]$. 
		Hence, \cref{lem:HalfGridWithCrosses} immediately yields the assertion.
	\end{proof}
	
	Before we can state the main lemma of this subsection, we first need the following definition.
	Let $H$ be an escaping subdivision of the hexagonal half-grid in a graph $G$ with vertical double rays $S^i$ and horizontal paths $P_{ij}$. An $H$-path $Q$ in $G$ with endvertices on $S^i$ and $S^j$ for some $i, j \in \N$ is \defn{$K$-fat} for some $K \in \N$ if 
	\begin{itemize}
		\item $d_G(Q, S^k) \geq \min\{k,K\}$ for all $k \neq i,j \in \N$, and
		\item $d_G(Q, P_{k\ell}) \geq \min\{k,K\}$ for all $k \in \N$ and $\ell \in \Z$.
	\end{itemize}
	
	\begin{lem} \label{lem:KFatHexGridWithJumpingPaths}
		Let $H$ be an escaping subdivision of the hexagonal half-grid in a locally finite graph~$G$ with vertical double rays~$S^i$. Suppose there are infinitely many pairwise disjoint $H$-paths~$Q_m$ such that for each $m\in\N$, the path~$Q_m$ has endvertices on $S^{i_m}$ and~$S^{j_m}$ for some $i_m, j_m \in \N$, such that~$Q_m$ is $m$-fat and such that $\max\{i_m, j_m\} < \min\{i_{m'}, j_{m'}\}$ and $|i_m - j_m| \geq 2$ for all $m < m'$.
		Then $H$ contains $K_{\aleph_0}$ as an ultra fat minor.
	\end{lem}
	
	\begin{proof}
		Since $G$ is locally finite and the $Q_m$'s are finite, we may assume that $d_G(Q_m, Q_{m'}) \geq m$ for all $m < m' \in \N$, by possibly deleting some of the $Q_m$'s. 
		Further, we may assume, for every $Q_m =: q_0^m \dots q_{n_m}^m$, that $Q_m \cap G[S^{i_m}, i'_m] = q_0^m \dots q_{i'_m}^m$ and $Q_m \cap G[S^{j_m}, j'_m] = q^m_{n_m-j'_m} \dots q^m_{n_m}$ where $i'_m := \min\{\lfloor i_m/3\rfloor, m\}$ and $j'_m := \min\{\lfloor j_m/3 \rfloor, m\}$. In particular, if $Q_m$ has distance less than $\lfloor i'_m/3 \rfloor$ to vertices $u, v \in V(S^{i_m})$, then $d_G(u, v) < i'_m$, and similarly for $j_m$.
		Indeed, let $Q'_m$ be a subpath of $Q_m$ which is a $B_G(S^{i_m}, i'_m)$--$B_G(S^{j_m}, j'_m)$ path. Then we can replace $Q_m$ by a path that consists of $Q'_m$ and a shortest $S^{i_m}$--$Q'_m$ path and a shortest $Q'_m$--$S^{j_m}$ path. 
		Note that to regain that $d_G(Q_m, P_{k\ell}) \geq \min\{k,m\}$, we might have to delete some of the $P_{k\ell}$'s. But since the new $Q_m$ can only be too close to paths $P_{k\ell}$ with $k \in \{i_m, i_m+1, j_m, j_m+1\}$ and because $G$ is locally finite, we only need to delete at most finitely many $P_{k\ell}$ for every $k$, so by \cref{prop:HexGridAfterDeletingPaths} this yields a subdivision $H' \subseteq H$ of the hexagonal half-grid with the same vertical double rays as $H'$.
		Note that for the new $Q_m$, we in particular have that if $Q_m$ has distance less than $i'_m$ to vertices $u, v \in V(S^{i_m})$, then $d_G(u, v) < i_m$, and similarly for $j_m$.
		
		Then the graph $\widetilde{H}$ obtained from the union of $H'$ and the $Q_m$'s is a subdivision of a graph $X$ as in the premise of \cref{cor:HalfGridWithJumpingEdges:NEW}.
		Hence,~$G$ contains a model $(\cV, \cE)$ of $K_{\aleph_0}$ such that $V_i, V(E_{ij}) \subseteq \widetilde{H}_{i} := H'_{\geq i} \cup \bigcup_{m \geq i} Q_m$ and such that the $E_{ij}$'s are pairwise disjoint. 
		By \cref{lem:DivSubHexGridOfEscHexGrid} and the assumptions on the $Q_m$'s, we may assume for every $K \in \N$ that the images of every two non-incident edges of $X$ have distance at least $\lfloor K/3 \rfloor$ in $G$ if their images in $\widetilde{H}$ are contained in $\widetilde{H}_{K}$. 
		Since every vertex of $K_{\aleph_0}$ has degree at least~$2$, we may assume that if some $V_i \in \cV$ meets some branch paths of $\widetilde{H}$ in an inner vertex, then it in fact contains it. Similarly, every $E_{ij} \in \cE$ contains any branch path of $\widetilde{H}$ as soon as it meets an inner vertex of it.
		Hence, since the $V_i$'s, $E_{ij}$'s are pairwise disjoint (except for incident branch set - path pairs) and contained in $\widetilde{H}_{i}$, they have distance at least $\lfloor i/3\rfloor$ to all $V_k, E_{k\ell}$ with $k \geq i$. It follows that $((V_i)_{i \in 3\N}, (E_{ij})_{i \neq j \in 3\N})$ is ultra fat.
	\end{proof}

	\subsection{Finding two thick components} \label{subsec:TwoThickComps}
	
	Let $G$ be a locally finite, \qt\ graph whose cycle space is generated by short cycles. In this subsection we show that if $G$ has a thick end, then either $K_{\aleph_0} \prec_{UF} G$, or $G$ contains a quasi-geodesic double ray such that, for some $L \in \N$, there are at least two thick components of $G-B_G(R, L)$. Together with \cref{lem:TwoThickCompsYieldFullGrid}, this then concludes the proof of \cref{thm:AsymptoticFullGridWeaker}.
	The proof of this assertion is mainly divided into two lemmas, \cref{lem:TwoThickComponents,lem:ThickAndHalfThickComponent} below; see \cref{subsec:ProofSketch12} for a sketch of the proof. 
	
	To carry out the proofs of \cref{lem:TwoThickComponents,lem:ThickAndHalfThickComponent}, we need the following two auxiliary statements.
	
	\begin{lem} \label{lem:ShortPathsInComp:NEW}
		Let $R=\dots r_{-1}r_0r_1\dots$ be a quasi-geodesic double ray in a graph $G$ of finite maximum degree whose cycle space is generated by cycles of bounded length. For every $K, N \in \N$ there exist $d = d(K, N), \ell = \ell(K) \in \N$ such that the following holds: if a component $C$ of $G-B_G(R, K)$ has neighbours in $B_G(R_{\leq i}, K)$ and $B_G(R_{\geq i + N}, K)$ for some $i \in \Z$, then there exist vertices $x \in \partial_G C \cap B_G(r_{i-\ell}Rr_i, K+1)$ and $y \in \partial_G C \cap B_G(r_{i+N}Rr_{i+N+\ell}, K+1)$ such that $d_C(x, y) \leq d$.
	\end{lem}
	
	\begin{proof}
		Let $c, \kappa \in \N$ such that $R$ is $c$-quasi-geodesic and such that the cycle space of $G$ is generated by cycles of length at most $\kappa$.
		By assumption, the maximum degree $\Delta(G)$ of~$G$ is finite.
		We prove the assertion with $d := N\cdot \Delta(G)^{K+\lfloor\frac{\kappa}{2}\rfloor+1} + \kappa$ and $\ell := c(2K + \kappa + 1)$.
		
		Let $U^+, U^*, U^- \subseteq B_C(\partial_G C, \lfloor \frac{\kappa-2}{2}\rfloor)$ be the set of vertices in $C$ that have distance at most $K + \lfloor\frac{\kappa}{2}\rfloor$ in~$G$ from $R_{\geq i + N}$, $r_iRr_{i+N}$ and $R_{\leq i}$, respectively; and note that by the assumptions on~$C$, the sets~$U^+$ and~$U^-$ are non-empty.
		By \cref{lem:kappa/2NhoodIsConnected}, $C[\partial_G C, \lfloor\frac{\kappa-2}{2}\rfloor]$ is connected. 
		Hence, since $U^+, U^- \subseteq B_C(\partial_G C, \lfloor\frac{\kappa-2}{2}\rfloor)$, there is a $U^+$--$U^-$ path $P = p_0 \dots p_n$ in~$C[\partial_G C, \lfloor\frac{\kappa-2}{2}\rfloor]$. By the definition of a $U^+$--$U^-$ path, $P$ has only its first vertex in $U^+$ and only its last vertex in $U^-$. Thus, since $B_C(\partial_G C, \lfloor\frac{\kappa-2}{2}\rfloor) = U^+ \cup U^* \cup U^-$, it follows that $\mathring{P} = p_1 \dots p_{n-1}$ is contained in~$U^*$, and hence~$P$ has length at most $|U^*| + 1 \leq N\cdot\Delta(G)^{K+\lfloor\frac{\kappa}{2}\rfloor+1}$. 
		Let $j \geq i+N$ and $j' \leq i$ such that $P$ starts in $V(C) \cap B_G(r_j, K+\lfloor\frac{\kappa}{2}\rfloor)$ and ends in $V(C) \cap B_G(r_{j'}, K+\lfloor\frac{\kappa}{2}\rfloor)$, and let $Q$ be a shortest $B_G(r_j, K+1)$--$p_0$ path and $Q'$ a shortest $B_G(r_{j'}, K+1)$--$p_n$ path. Then the concatenation of $Q, P$ and $Q'$ yields a $B_G(r_j, K+1)$--$B_G(r_{j'}, K+1)$ path $P'$ in $G$ which starts in a vertex $x \in \partial_G C \cap B_G(r_j, K+1)$ and ends in some $y \in \partial_G C \cap B_G(r_{j'}, K+1)$. In particular,~$Q$ and~$Q'$ have length at most $\lfloor\frac{\kappa-2}{2}\rfloor$, and hence $P'$ has length at most $d$, which implies $d_C(x,y) \leq d$. Moreover, $j \geq i+N$ and $j' \leq i$. It remains to show that $j \leq i+N+\ell$ and $j' \geq i - \ell$.  
		Since $\mathring{P}$ lies in~$U^*$, there exists a $p_1$--$r_iRr_{i+N}$ path~$Q''$ of length at most $K + \lfloor \frac{\kappa}{2}\rfloor$. Then $Qp_0p_1Q''$ is an $r_j$--$r_k$ path for some $k \leq i+N$. As $Qp_0p_1Q''$ has length at most $(K + \lfloor\frac{\kappa}{2}\rfloor) + 1 + (K + \lfloor\frac{\kappa}{2}\rfloor) \leq 2K + \kappa + 1$ and $R$ is $c$-quasi-geodesic, it follows that $|j - k| = d_R(r_j, r_k) \leq c(2K + \kappa + 1) = \ell$, and hence $j \leq k + \ell \leq i + N + \ell$. The case $j' \geq i-\ell$ is analogous.
	\end{proof}
	
	\begin{cor} \label{cor:OnlyFinitelyManyLongComps}
		Let $L \in \N$, and let $R$ be a quasi-geodesic double ray in some locally finite graph $G$ whose cycle space is generated by cycles of bounded length. Then $G-B_G(R,L)$ has at most finitely many long components.
		
		In particular, if, for every $K \geq L$, $G-B_G(R,K)$ has a long component $C_K$, then $G-B_G(R,L)$ has a thick component which contains infinitely many $C_K$.
	\end{cor}
	
	\begin{proof}
		Applying \cref{lem:ShortPathsInComp:NEW} to $R$ and $K := L$ and $N := 0$ yields some $\ell \in \N$ such that every long component of $G - B_G(R, L)$ has a neighbour in $B_G(r_0Rr_{\ell}, L)$. Hence, since $B_G(r_0Rr_{\ell}, L)$ is finite as $r_0Rr_\ell$ is finite and $G$ is locally finite, there are at most finitely many long components.
		
		For the `in particular'-part note that since $G$ is locally finite, every component of $G-B_G(R,L)$ which contains a long component of $G-B_G(R, K)$ for some $K \geq L$ is long.
		Hence, the assertion follows from the first part by the pigeonhole principle.
	\end{proof}
	
	We can now prove the two main lemmas of this section.
	Given a double ray $R = \dots r_{-1}r_0r_1 \dots$ in a graph $G$, a component $C$ of $G-B_G(R,L)$ is \defn{half-long} if~$C$ has neighbours in $B_G(R_{\geq i}, L)$ \emph{or} in $B_G(R_{\leq -i}, K)$ for all $i \in \N$, and $C$ is \defn{half-thick} if, for every $M \geq L$, some half-long component of $G-B_G(R,M)$ is contained in~$C$.
	
	\begin{lem} \label{lem:TwoThickComponents}
		Let $c, L \in \N$ and let $R$ be a $c$-quasi-geodesic double ray in a locally finite, quasi-transitive graph $G$ whose cycle space is generated by cycles of bounded length. Assume that $G-B_G(R, L)$ has distinct components $C \neq D$ such that $C$ is thick and $D$ is half-thick. 
		Then $G$ contains a $c$-quasi-geodesic double ray $S$ such that $G-B_G(S,L)$ has two thick components.
	\end{lem}
	
	\begin{proof}
		Let $\kappa\in\N$ such that the cycle space of~$G$ is generated by cycles of length at most~$\kappa$.
		As we are done if $D$ is thick, we may assume that we can enumerate $R =: \dots r_{-1}r_0r_1\dots$ so that $N_G(D) \subseteq B_G(R_{\geq 0}, L)$.
		Since $G$ is quasi-transitive, there is an infinite index set $I_0 \subseteq \N$ such that all~$r_i$ with $i \in I_0$ lie in the same $\Aut(G)$-orbit.
		Let $v$ be another vertex in that orbit.
		Then there exists a sequence $(\varphi_i)_{i \in I_0}$ of automorphisms of~$G$ such that $\varphi_i(r_i) = v$.
		Since $G$ is locally finite, there is an infinite index set $I_1 \subseteq I_0$ such that $\varphi_i(r_{i-1}r_{i}r_{i+1})$ coincides for all $i \in I_1$ amongst which we again find some infinite set $I_2 \subseteq I_1$ such that $\varphi_i(r_{i-2}\ldots r_{i+2})$ coincides for all $i \in I_2$ and so on.
		Now pick for every $n \in \N$ some $i_n \in I_n$, and let $I$ consist of these $i_n$'s.
		This leads to a $c$-quasi-geodesic double ray $S$ that contains~$v$ and such that every subpath of~$S$ of length $2\ell \in \N$ that contains~$v$ as central vertex is the image of $r_{i_n-\ell}Rr_{i_n+\ell}$ under~$\phi_{i_n}$ for all $i_n \in I$ with $n \geq \ell$. We enumerate $S =: \dots s_{-1}s_0s_1\dots$ where $s_0 := v$ and $s_1 = \phi_{i_1}(r_{i_1+1})$.
		
		We claim that $S$ is as desired. For this, we show that, for every $K \geq L$, $G-B_G(S,K)$ has long components $C'_K, D'_K$ such that every $C'_K$--$D'_K$ path in $G$ meets $B_G(S, L)$. We first show that it will imply the desired result. Indeed, by \cref{cor:OnlyFinitelyManyLongComps}, $G-B_G(S, L)$ has a thick component $E$ which contains infinitely many of the $C'_K$'s, $D'_K$'s. If infinitely many of the $C'_K$'s, $D'_K$'s do not lie in $E$, then applying \cref{cor:OnlyFinitelyManyLongComps} again yields a second thick component $E' \neq E$. Otherwise, at most finitely many of the $C'_K$'s, $D'_K$'s are not contained in~$E$, which implies that~$E$ contains both~$C'_K$ and~$D'_K$ for some $K \geq L$. But since~$E$ is connected and avoids $B_G(S,L)$, this contradicts that every $C'_K$--$D'_K$ path meets $B_G(S,L)$.
		
		So let $K \geq L$ be given. It remains to show that $G-B_G(S,K)$ has long components $C'_K, D'_K$ such that every $C'_K$--$D'_K$ path in $G$ meets $B_G(S,L)$.
		Since $C$ is thick and $D$ is half-thick there exist components $C_K \subseteq C$ and $D_K \subseteq D$ of $G-B_G(R,K)$ such that $C_K$ is long and $D_K$ is half-long. 
		
		\begin{claim} \label{claim:TwoThickComps}
			There are long components $C'_K, D'_K$ of $G-B_G(S, K)$ and vertices $x \in \partial_G C'_K$ and $y \in \partial_G D'_K$ such that $\phi^{-1}_i(x) \in \partial_G C_K$ and $\phi^{-1}_i(y) \in \partial_G D_K$ for infinitely many $i \in I$.
		\end{claim}
		
		\begin{claimproof}
			Let us first find a component $D'_K$.
			Since $N_G(D) \subseteq B_G(R_{\geq 0}, L)$ and $K \geq L$, we also have $N_G(D_K) \subseteq B_G(R_{\geq 0}, K)$. 
			Let $m \in \N$ such that $N_G(D_K) \cap B_G(r_m, K) \neq \emptyset$. For every $n \in \N$, let $\ell, d_n$ be given by applying \cref{lem:ShortPathsInComp:NEW} to $R$, $K$ and $N := 2n$ (note that $\ell$ does not depend on~$N$). 
			
			Then there exists, for every $n \in \N$ and $i \geq m+n$, a $B_G(r_{i-n-\ell}Rr_{i-n}, K+1)$--$B_G(r_{i+n}Rr_{i+n+\ell}, K+1)$ path $P_{in}$ in $D_K$ of length at most $d_n$ with endvertices $p_{in} \in \partial_G D_K \cap B_G(r_{i-n-\ell}Rr_{i-n}, K+1)$ and $q_{in} \in \partial_G D_K \cap B_G(r_{i+n}Rr_{i+n+\ell}, K+1)$.
			
			By the choice of $I$, we have $\phi_i(r_{i-\ell}Rr_{i+\ell}) = s_{-\ell}Ss_{\ell}$ for all large enough $i \in I$, and hence these automorphisms $\phi_i$ map $B_G(r_{i-\ell}Rr_{i+\ell}, K+1)$ to $B_G(s_{-\ell}Ss_{\ell}, K+1)$. As $G$ is locally finite, the set $B_G(s_{-\ell}Ss_{\ell}, K+1)$ is finite. Combining these facts with $p_{i0}, q_{i0} \in B_G(r_{i-\ell}Rr_{i+\ell}, K+1)$ yields that there are $p_0, q_0 \in V(G)$ and an infinite index set $J_0 \subseteq I$ such that $\phi_i(p_{i0}) = p_0$ and $\phi_i(q_{i0}) = q_0$ for all $i \in J_0$.
			Using again that $G$ is locally finite and $P_{i0}$ has length at most $d_0$ for all $i \in J_0$, we find a $p_0$--$q_0$ path $P_0 \subseteq G$ and an infinite index set $J'_0 \subseteq J_0$ such that $\phi_i(P_{i0}) = P_0$ for all $i \in J'_0$.
			By the same argument, we find subsets $J'_0 \supseteq J'_1 \supseteq \dots$ such that, for all $n \in \N$ and $i \in J'_n$, $\phi_i(P_{in}) = P_n$ for some path $P_n \subseteq G$ with endvertices $p_n, q_n \in V(G)$.
			Pick for every $n \in \N$ some $i \in J'_n$ and let $J$ consist of these $i$'s.
			
			We claim that there is a component $D'_K$ of $G-B_G(S,K)$ which contains infinitely many of the $P_n$'s, and which is hence long. For this, we first show that no $P_n$ meets $B_G(S, K)$.
			So let $n \in \N$ be given, and set $d'_n := n+\ell + c(d_n+2K+2)$. By the choice of $P_n$, we have $P_n = \phi_i(P_{in})$ for all large enough $i \in J$. Since also $\phi_i(B_G(r_{i-d'_n}Rr_{i+d'_n}, K)) = B_G(s_{-d'_n}Ss_{d'_n}, K)$ for all large enough $i \in J$ and since $P_{in}$ avoids $B_G(R, K)$, it follows that $P_n$ avoids $B_G(s_{-d'_n}Ss_{d'_n},K)$. Moreover, $P_n$ is also disjoint from $B_G(S_{< -d'_n}, K)$ and $B_G(S_{> d'_n}, K)$ since $d'_n = n+\ell+c(d_n+2K+2)$ and because $S$ is $c$-quasi-geodesic and~$P_n$ has length at most $d_n$ and starts in $B_G(s_{-n-\ell}Ss_{-n}, K+1)$ and ends in $B_G(s_{n}Ss_{n+\ell}, K+1)$. 
			
			So every $P_n$ avoids $B_G(S, K)$ and is hence contained in a component of $G-B_G(S, K)$.
			By the choice of~$\ell$ via \cref{lem:ShortPathsInComp:NEW} and because $P_n$ starts in $B_G(S_{\geq n}, K+1)$ and ends in $B_G(S_{\leq -n}, K+1)$, every component that contains some $P_n$ with $n \geq \ell$ attaches to $B_G(s_0Ss_{\ell}, K)$. (This follows by setting $i=N=0$ in \cref{lem:ShortPathsInComp:NEW}, where then $y$ is a witness.) Since this set is finite as $G$ is locally finite, infinitely many $P_n$ lie in the same component $D'_K$ of $G-B_G(S, K)$, which then needs to be long.
			Now since $\phi_i^{-1}(P_n) = P_{in} \subseteq D_K$ for all $i\in I'$, where $I'$ is an infinite subset of~$I$, we may choose $y \in \partial_G D'_K$ as one endvertex of some $P_n$ which is contained in~$D'_K$.
			
			The proof for $C_K'$ is now completely analogous except that we have to use $I'$ instead of~$I$ in order to ensure that $\phi^{-1}_i(x) \in \partial_G C_K$ and $\phi^{-1}_i(y) \in \partial_G D_K$ for (the same) infinitely many $i \in I$.
		\end{claimproof}
		
		Let $C'_K, D'_K$ and $x \in \partial_G C'_K$, $y \in \partial_G D'_K$ be given by \cref{claim:TwoThickComps}.
		To finish the proof, we are left to show that every $C'_K$--$D'_K$ path in $G$ meets $B_G(S, L)$.
		For this, suppose for a contradiction that there is a $C'_K$--$D'_K$ path that avoids $B_G(S,L)$. Then, since $C'_K \ni x$ and $D'_K \ni y$ are connected, there also exists an $x$--$y$ path~$Q$ that avoids $B_G(S,L)$. Denote by $m$ the length of $Q$, let $\ell \in \N$ such that $x, y \in B_G(s_{-\ell}Ss_\ell, K+1)$, and set $m' := c(K+m+L+2)+\ell$. 
		By the choice of $S$ there is some $i$ among the infinitely many $i \in I$ that satisfies $\phi_i^{-1}(x) \in \partial_G C_K$ and $\phi^{-1}_i(y) \in \partial_G D_K$ such that $\phi_i(r_{i-m'}Rr_{i+m'}) = s_{-m'}Ss_{m'}$. Since $Q$ avoids $B_G(S,L)$, it follows that $\phi^{-1}_i(Q)$ avoids $B_G(r_{i-m'}Rr_{i+m'}, L)$. But $\phi^{-1}_i(Q)$ also avoids $B_G(R_{< i-m'}, L) \cup B_G(R_{> i+m'}, L)$: otherwise, since $\phi^{-1}_i(Q)$ has length $m$ and starts in $B_G(r_{i-\ell}Rr_{i+\ell}, K+1)$, there would be a path of length at most $L + m + (K+1)$ that joins a vertex from $R_{< i-m'} \cup R_{> i+m'}$ to a vertex from $r_{i-\ell}Rr_{i+\ell}$. Since $m'-\ell = c(L+m+K+2)$, this would contradict that $R$ is $c$-quasi-geodesic.
		Hence, $\phi_i^{-1}(Q)$ is an $\phi^{-1}_i(x)$--$\phi_i^{-1}(y)$ path which avoids $B_G(R, L)$. But since $\phi^{-1}_i(x) \in \partial_G C_K \subseteq C_L$ and $\phi_i^{-1}(y) \in \partial_G D_K \subseteq D_L$, this contradicts that $C_L$ and $D_L$ are distinct components of $G-B_G(R, L)$.
	\end{proof}

	\begin{lem} \label{lem:ThickAndHalfThickComponent}
		Let $G$ be a locally finite, quasi-transitive graph with a thick end whose cycle space is generated by cycles of bounded length. If $G$ does not contain $K_{\aleph_0}$ as an ultra fat minor, then there exists $L \in \N$ and a quasi-geodesic double ray $R$ in~$G$ such that $G - B_G(R, L)$ has distinct components $C \neq D$ such that~$C$ is thick and $D$ is half-thick.
	\end{lem}
	
	\begin{proof}
		Let $R_1, R_2, R_3$ be given by applying \cref{lem:QuasiGeodesic3StarOfRays} to some thick end of the \lf, \qt\ graph $G$, which is accessible by \cref{thm:CycleSpaceAccessible}. Then applying \cref{thm:HalfGrid:BoundedCycles} to the quasi-geodesic, and hence diverging, double ray $R_1 \cup R_2$ yields an escaping subdivision~$H$ of the hexagonal half-grid with vertical double rays $S^i$ and horizontal paths $P_{ij}$ such that $S^0 = R_1 \cup R_2$. Since $H$ is escaping, there exist $M_0 < M_1 < \ldots \in \N$ such that $M_i > M_{i-1} + 2i$ for all $i \geq 1$ and
		\begin{enumerate}[label=(\roman*)]
			\item \label{itm:ThickHalfThickComp:Si} $S^i \subseteq G[S^0, M_i] - B_G(S^0, M_{i-1}+2i)$ for all $i \in \N_{\geq 1}$, and
			\item \label{itm:ThickHalfThickComp:Pij} $P_{1j} \subseteq G[S^0, M_1]$ and $P_{ij} \subseteq G[S^0, M_{i}]-B_G(S^0, M_{i-2}+i)$ for all $i \in \N_{\geq 2}$ and $j \in \Z$.
		\end{enumerate}
		Let us also note that, since $R_1 \cup R_2 \cup R_3$ is quasi-geodesic, we also have that
		\begin{enumerate}[label=(\roman*)]
			\setcounter{enumi}{2}
			\item \label{itm:ThickHalfThickComp:R3} the set $V(R_3) \cap B_G(S^0, L)$ is finite for all $L \in \N$.
		\end{enumerate}
		Hence, by \ref{itm:ThickHalfThickComp:Pij} and \ref{itm:ThickHalfThickComp:R3} and because $G$ is locally finite, we may assume, by deleting at most finitely many~$P_{ij}$ for every $i \in \N$ and applying \cref{prop:HexGridAfterDeletingPaths}, that
		\begin{enumerate}[label=(\roman*)]
			\setcounter{enumi}{3}
			\item \label{itm:ThickHalfThickComp:R3Pij} $d_G(R_3, P_{ij}) > 2i$ for all $i \in \N$ and $j \in \Z$.
		\end{enumerate}
		
		Let $\defnm{H_{\geq n}} \subseteq H$, for $n \in \N$, be the subgraph consisting of all $S^i, P_{kj}$ with $i \geq n$ and $k > n$. 
		By \ref{itm:ThickHalfThickComp:Si} and \ref{itm:ThickHalfThickComp:Pij}, we have, for every $L \in \N_{\geq 1}$, that $d_G(H_{\geq L}, S^0) > L$; let~\defn{$C_L$} be the component of $G-B_G(S^0, L)$ containing $H_{\geq L}$. Clearly, the $C_L$'s are long and hence, since $C_1 \supseteq C_2 \supseteq \dots$, the $C_L$'s are thick.
		
		If there is some $L \in \N_{\geq 1}$ such that $R_3 \cap C_L = \emptyset$, then we are done. 
		Indeed, by \ref{itm:ThickHalfThickComp:R3}, there is, for every $L' \geq L$, a component $D_{L'}$ of $G-B_G(S^0, L')$ that contains a tail of $R_3$. Since $R_3$ and $R_1$ belong to the same end, the components $D_{L'}$ are half-long. As clearly $D_{L'} \subseteq D_L$ for all $L' \geq L$, we find that $D_L$ is half-thick. Since also $D_L \neq C_L$ by assumption, $R := R_1 \cup R_2$, $L$, $C := C_L$ and $D:=D_L$ are as desired.
		\medskip
		
		Thus, we may assume that $R_3 \cap C_L \neq \emptyset$ for all $L \in \N_{\geq 1}$. We distinguish two cases. 
		
		\paragraph{\textbf{Case 1:}} \emph{For all $K \in \N$ there is some $N_K \in \N$ such that $d_G(R_3, S^i) > K$ for all $i \geq N_K$.} (See \cref{fig:ThickHalfThickComp:1}.)
		
		We will use~$R_3$ to find fat $H$-paths as in \cref{lem:KFatHexGridWithJumpingPaths}, which then implies that $G$ contains $K_{\aleph_0}$ as an ultra fat minor, concluding the first case of the proof.
		
		Without loss of generality let $N_K \geq K$ for all $K \in \N$.
		By \ref{itm:ThickHalfThickComp:Si} and \ref{itm:ThickHalfThickComp:R3}, for every $K \in \N$, the ray $R_3$ has a tail $T_K$ which avoids $B_G(S^0, M_{N_K-1}+K)$, and which thus satisfies $d_G(T_K, S^i) > K$ for all $i \in \N$ by \ref{itm:ThickHalfThickComp:Si}.
		By \ref{itm:ThickHalfThickComp:Pij} and \ref{itm:ThickHalfThickComp:R3Pij}, we also find $d_G(P_{ij}, T_K) > K$ for all $i \in \N$ and $j \in \Z$, and hence $d_G(H, T_K) > K$. 
		
		\begin{figure}[ht]
			\centering
			\begin{tikzpicture}

    \foreach \x in {0,1,2,3,4,5,6,7,8}
    \draw[stealth-stealth, line width=0.6] (\x,-0.5) -- (\x,4.5);


    \foreach \x in {0,2,4,6}
	\foreach \y in {0,1,2,3,4}
    \draw[line width=0.6] (\x,\y) -- (\x+1,\y);


    \foreach \x in {1,3,5,7}
    \foreach \y in {0.5,1.5,2.5,3.5}
    \draw[line width=0.6] (\x,\y) -- (\x+1,{\y});


	\foreach \y in {0,1,2,3,4}
    \draw[line width=0.6] (8,\y) -- (8.5,\y);


	\draw[blue, line width=1.2,line cap=round] (0,1.25) -- (1.25,1.25);

    \begin{scope}[opacity=0.5]
        \draw[Gray, line width=6pt,line cap=round] (3.5,3) to[out angle=270, in angle=150, curve through = {(3.5,2.5)}] (4,1.75);
        \draw[Gray, line width=6pt,line cap=round] (3.5,3) to[out angle=20, in angle=187, curve through = {(4.5,3.28)}] (6.5,3.55);	   \draw[Gray, line width=6pt,line cap=round] (6.5,3.55) to[out angle=290, in angle=30, curve through = {(6.75,1.5)}] (6,0.25);
    \end{scope}
	\draw[blue, line width=1.2,line cap=round] (1.25,1.25) to[out angle=70, in angle=200, curve through = {(2.5,2.5)}] (3.5,3);
	\draw[blue, line width=1.2,line cap=round] (3.5,3) to[out angle=20, in angle=187, curve through = {(4.5,3.28)}] (6.5,3.55); 
	\draw[blue, thick,-stealth, line width=1.2,line cap=round] (6.5,3.55) to[out angle=7, in angle=182, curve through = {(7.5,3.625)}] (9,3.7);
    \draw[Red, line width=1.2,line cap=round] (3.5,3) to[out angle=270, in angle=150, curve through = {(3.5,2.5)}] (4,1.75);
	\draw[Red, line width=1.2,line cap=round] (6.5,3.55) to[out angle=290, in angle=30, curve through = {(6.75,1.5)}] (6,0.25);
	\draw[Red,line width=1.2,line cap=round] (2.5,2.5) to[out angle=270, in angle=30, curve through = {(2.5,2)}] (2,0.75);
	\draw[Red,line width=1.2,line cap=round] (4.5,3.28) to[out angle=270, in angle=150, curve through = {(4.5,2.75)}] (5,1.25);
	\draw[Red,line width=1.2,line cap=round] (7.5,3.625) to[out angle=270, in angle=150, curve through = {(7.5,2.5)}] (8,1.75);


	\draw (-0.5,2.75) node {$R_1$};
	\draw (-0.5,0.5) node {$R_2$};
	\draw[blue] (9,3.3) node {$R_3$};
	\draw (0,-0.8) node {\scalebox{0.85}{$S^0=R_1\cup R_2$}};
	\draw (4.2,-0.8) node {\scalebox{0.85}{$S^{i_{3m}}$}};
	\draw (5.35,-0.8) node {\scalebox{0.85}{$S^{i_{3m+1}}$}};
	\draw (6.38,-0.8) node {\scalebox{0.85}{$S^{i_{3m+2}}$}};
	\draw[Red] (4.45,1.35) node {\scalebox{0.7}{$Q'_{3m+1}$}};
	\draw[Gray] (6.35,1.5) node {\scalebox{0.85}{$W_m$}}; 
	\draw (8.75,0.5) node {$H$};

\end{tikzpicture}
			\vspace{-3em}
			\caption{Sketch of Case 1 in the proof of \cref{lem:ThickAndHalfThickComponent}: the hexagonal half-grid $H$ and the ray $R_3$, which has large distance from $H$ but is connected to $H$ by paths $Q'_m$. The paths $W_m$ consisting of $Q'_{3m}, Q'_{3m+2}$ and a suitable subpath of $R_3$ are $m$-fat $H$-paths.}
			\label{fig:ThickHalfThickComp:1}
		\end{figure}
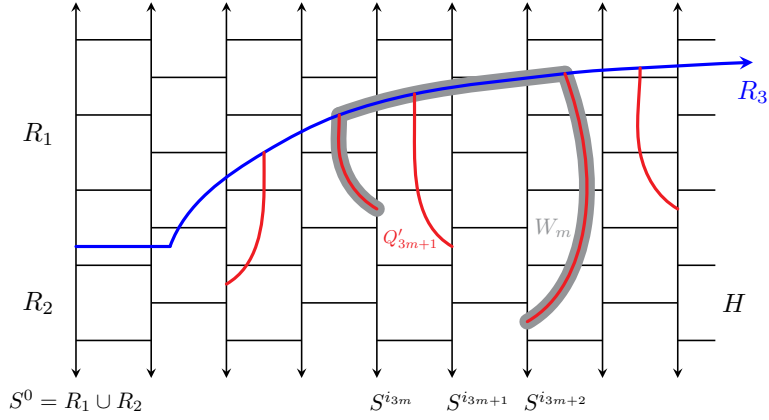
		
		By \ref{itm:ThickHalfThickComp:Si}, and because $R_3 \cap C_L \neq \emptyset$ as well as $H_{\geq L} \subseteq C_L$ for all $L \in \N_{\geq 1}$, there exists, for every $m \in \N$, a $T_m$--$\bigcup_{i \geq N_m} B_G(S^i, m)$ path $Q_m$ that ends in some $B_G(S^{i_m}, m)$ for $i_m \geq m$ and avoids $B_G(S^0, M_{N_m}+m)$. In particular, by definition of an $A$--$B$ path, $Q_m$ has only its last vertex in $\bigcup_{i \geq N_m} B_G(S^i, m)$.
		We extend each $Q_m$ to a $T_m$--$S^{i_m}$ path $Q'_m$ by adding a shortest $S^{i_m}$--$Q_m$ path. 
		Since $d_G(Q_m, S^i) > m$ for all $i \neq i_m$ as mentioned above and again by~\ref{itm:ThickHalfThickComp:Si}, we find $d_G(Q'_m, S^i) > m$ for all $i \neq i_m$. 
		Moreover, by~\ref{itm:ThickHalfThickComp:Si}, the paths $Q'_m$ still avoid $B_G(S^0, M_{N_m}+m)$.
		Hence, by~\ref{itm:ThickHalfThickComp:Pij}, we have that, for every $i \in \N$, the paths $Q'_m$ with $m \geq i$ have distance at least $m$ from the paths $P_{ij}$. Since $G$ is locally finite, this implies that we may assume, by deleting for every $i \in \N$ at most finitely many $P_{ij}$ and applying \cref{prop:HexGridAfterDeletingPaths}, that $d_G(P_{ij}, Q'_m) > m$ for all $m, i \in \N$ and $j \in \Z$.
		All in all, we find $d_G(H-S^{i_m}, Q'_m) > m$ for all $m \in \N$.
		
		Since $Q'_m$ avoids $B_G(S^0, M_{N_m}+m)$ but is itself eventually contained in some $B_G(S^0, m')$ for $m' > m$ as~$Q'_m$ is finite, we may assume, by passing to a subsequence of the $Q'_m$'s if necessary, that the $Q'_m$'s are pairwise disjoint.
		Moreover, by \ref{itm:ThickHalfThickComp:Si} and again since $Q'_m$ avoids $B_G(S^0, M_{N_m})$, we may assume that $i_1 < i_2 < \dots$, by once again passing to a subsequence of the $Q'_m$'s if necessary.
		
		Now since $d_G(H-S^{i_m}, Q'_m) > m$ and $d_G(H, T_m) > m$, the paths $W_m$ that consist of $Q'_{3m}$, $Q'_{3m+2}$ and a suitable subpath of $T_{3m}$ are $m$-fat $H$-paths with endvertices on $S^{i_{3m}}$ and $S^{i_{3m+2}}$ (see \cref{fig:ThickHalfThickComp:1}). 
		In particular, since the $Q'_m$'s are pairwise disjoint, we may assume, by passing to a subsequence of the $W_m$'s, that also the $W_m$'s are pairwise disjoint.
		Hence, the $W_m$'s are $m$-fat $H$-paths as in \cref{lem:KFatHexGridWithJumpingPaths}, which implies that $G$ contains $K_{\aleph_0}$ as an ultra fat minor. This concludes the first case of the proof.
		\medskip
		
		\paragraph{\textbf{Case 2:}} \emph{There exists some $K$ such that $d_G(R_3, S^i) < K$ for infinitely many $i \in \N$.} (See \cref{fig:ThickHalfThickComp:2}.) 
		
		We first show that we may assume, by passing to a subgraph of $H$ if necessary, that
		\begin{enumerate}[label=(\alph*)]
			\item \label{itm:ThickHalfThickComp:2:S^i} $d_G(R_3, S^{i}) < i$ for all $i \geq K$, and
			\item \label{itm:ThickHalfThickComp:2:S^ileq} for all $K \leq j < i$, if $d_G(r, S^i) < i$ for some $r \in V(R_3)$, then $d_G(rR_3, S^{j}) \geq j$.
		\end{enumerate}
		Indeed, set $i_j := j$ for all $j <K$, and let $i_K \geq K$ be minimal such that $d_G(R_3, S^{i_K}) < K$. By \ref{itm:ThickHalfThickComp:Si} and~\ref{itm:ThickHalfThickComp:R3},~$R_3$ has a tail $T$ that is disjoint from $B_G(S^0, M_{i_K}+i_K)$. 
		Since $R_3-T$ is finite and because of~\ref{itm:ThickHalfThickComp:Si}, we have $d_G(R_3-T, S^i) > i$ for all large enough $i \in \N$, and hence, as $d_G(R_3, S^i) < K$ for infinitely many $i \in \N$, there is some $i_{K+1} > i_K$ such that $d_G(T, S^{i_{K+1}}) < i_{K+1}$ and $d_G(R_3-T, S^{i}) \geq i$ for all $i \geq i_{K+1}$. By continuing in this way, we obtain a sequence $0 := i_0 < i_1 < \dots \in \N$ such that the rays $S^{i_j}$ satisfy~\ref{itm:ThickHalfThickComp:2:S^i} and~\ref{itm:ThickHalfThickComp:2:S^ileq}.
		Now pick a subdivision $H' \subseteq H$ of the hexagonal half-grid whose vertical double rays are precisely the~$S^{i_j}$. It is easy to see that $H'$ still satisfies \ref{itm:ThickHalfThickComp:Si} to \ref{itm:ThickHalfThickComp:R3}. By deleting from $H'$ at most finitely many $P'_{ij}$ for every $i \in \N$ and applying \cref{prop:HexGridAfterDeletingPaths}, we also regain property \ref{itm:ThickHalfThickComp:R3Pij} for~$H'$. Hence, $H'$ is the desired subgraph of~$H$, and we denote $H'$ again by~$H$.
		
		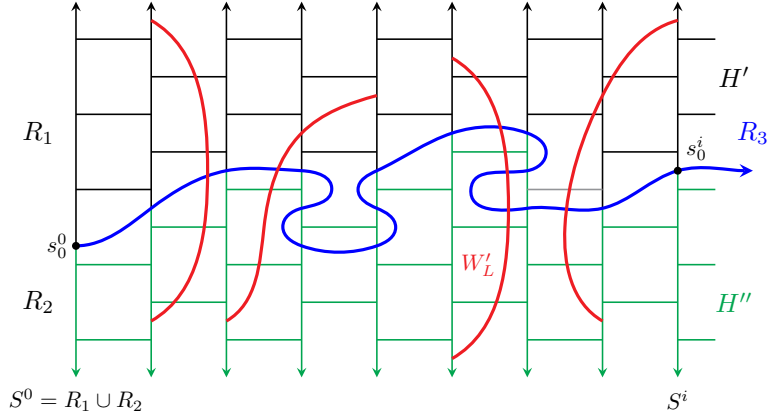
\begin{figure}[ht]
			\centering
			\begin{tikzpicture}

    \draw[-stealth, line width=0.6] (0,1.25) -- (0,4.5);
    \draw[-stealth, line width=0.6] (1,1.75) -- (1,4.5);
    \draw[-stealth, line width=0.6] (2,2.25) -- (2,4.5);
    \draw[-stealth, line width=0.6] (3,2.25) -- (3,4.5);
    \draw[-stealth, line width=0.6] (4,2.25) -- (4,4.5);
    \draw[-stealth, line width=0.6] (5,2.75) -- (5,4.5);
    \draw[-stealth, line width=0.6] (6,2.75) -- (6,4.5);
    \draw[-stealth, line width=0.6] (7,1.75) -- (7,4.5);
    \draw[-stealth, line width=0.6] (8,2.25) -- (8,4.5);

    \draw[Green, -stealth, line width=0.6] (0,1.25) -- (0,-0.5);
    \draw[Green, -stealth, line width=0.6] (1,1.75) -- (1,-0.5);
    \draw[Green, -stealth, line width=0.6] (2,2.25) -- (2,-0.5);
    \draw[Green, -stealth, line width=0.6] (3,2.25) -- (3,-0.5);
    \draw[Green, -stealth, line width=0.6] (4,2.25) -- (4,-0.5);
    \draw[Green, -stealth, line width=0.6] (5,2.75) -- (5,-0.5);
    \draw[Green, -stealth, line width=0.6] (6,2.75) -- (6,-0.5);
    \draw[Green, -stealth, line width=0.6] (7,1.75) -- (7,-0.5);
    \draw[Green, -stealth, line width=0.6] (8,2.25) -- (8,-0.5);


	\foreach \y in {2,3,4}
    \draw[line width=0.6] (0,\y) -- (1,\y);
	\foreach \y in {0,1}
    \draw[Green, line width=0.6] (0,\y) -- (1,\y);


    \foreach \y in {2.5,3.5}
    \draw[line width=0.6] (1,\y) -- (2,\y);
    \foreach \y in {0.5,1.5}
    \draw[Green, line width=0.6] (1,\y) -- (2,\y);


	\foreach \y in {3,4}
    \draw[line width=0.6] (2,\y) -- (3,\y);
	\foreach \y in {0,1,2}
    \draw[Green, line width=0.6] (2,\y) -- (3,\y);


    \foreach \y in {2.5,3.5}
    \draw[line width=0.6] (3,\y) -- (4,\y);
    \foreach \y in {0.5,1.5}
    \draw[Green, line width=0.6] (3,\y) -- (4,\y);


	\foreach \y in {3,4}
    \draw[line width=0.6] (4,\y) -- (5,\y);
	\foreach \y in {0,1,2}
    \draw[Green, line width=0.6] (4,\y) -- (5,\y);


    \foreach \y in {3.5}
    \draw[line width=0.6] (5,\y) -- (6,\y);
    \foreach \y in {0.5,1.5,2.5}
    \draw[Green, line width=0.6] (5,\y) -- (6,\y);


	\foreach \y in {3,4}
    \draw[line width=0.6] (6,\y) -- (7,\y);
    \draw[Gray, line width=0.6] (6, 2) -- (7, 2);
	\foreach \y in {0,1}
    \draw[Green, line width=0.6] (6,\y) -- (7,\y);


    \foreach \y in {2.5,3.5}
    \draw[line width=0.6] (7,\y) -- (8,\y);
    \foreach \y in {0.5,1.5}
    \draw[Green, line width=0.6] (7,\y) -- (8,\y);


	\foreach \y in {0,1,2}
    \draw[Green, line width=0.6] (8,\y) -- (8.5,\y);

	\foreach \y in {3,4}
    \draw[line width=0.6] (8,\y) -- (8.5,\y);


	\draw[blue,-stealth, line width=1.2] (0,1.25) to[out angle=0, in angle=180, curve through = {(1,1.75) (2,2.25) (3,2.25) (3.4,2) (3,1.75) (2.75,1.5) (3,1.25) (4,1.25) (4.25,1.5) (4,1.75) (3.75,2) (4,2.25) (5,2.75) (6,2.75) (6.25,2.5) (6,2.25) (5.25,2) (6,1.75) (7,1.75) (8,2.25)}] (9,2.25);


	\draw[Red, line width=1.2] (1,4.25) to[out angle=330, in angle=30, curve through = {(1.75,2.25)}] (1,0.25);
	\draw[Red, line width=1.2] (4,3.25) to[out angle=195, in angle=30, curve through = {(3,2.75) (2.5,1.5)}] (2,0.25);
	\draw[Red, line width=1.2] (5,3.75) to[out angle=330, in angle=30, curve through = {(5.75,1.75)}] (5,-0.25);
	\draw[Red, line width=1.2] (8,4.25) to[out angle=195, in angle=150, curve through = {(7,3.25) (6.5,1.5)}] (7,0.25);


	\draw (-0.5,2.75) node {$R_1$};
	\draw (-0.5,0.5) node {$R_2$};
	\draw[blue] (9,2.75) node {$R_3$};
	\draw[Green] (8.75,0.5) node {$H''$};
	\draw (8.75,3.5) node {$H'$};
	\draw (0,-0.8) node {\scalebox{0.85}{$S^0=R_1\cup R_2$}};
	\draw (8,-0.8) node {\scalebox{0.85}{$S^i$}};
	\draw[Red] (5.35,1) node {\scalebox{0.85}{$W'_L$}};


    \fill (0,1.25) circle (1.5pt);
	\draw (-0.25,1.25) node {\scalebox{0.85}{$s_0^0$}};
    

    \fill (8,2.25) circle (1.5pt);
	\draw (8.25,2.55) node {\scalebox{0.85}{$s_0^i$}};

\end{tikzpicture}
			\vspace{-3em}
			\caption{Sketch of Case 2 in the proof of \cref{lem:ThickAndHalfThickComponent}: the hexagonal half-grid $H$ and the ray $R_3$, which has distance $<i$ to every vertical double ray $S^i$ of $H$ by \ref{itm:ThickHalfThickComp:2:S^i}. By \ref{itm:ThickHalfThickComp:2:S^ileq}, once $R_3$ comes close to some $S^i$, it will never come close to some $S^j$ with $j < i$ again.
				The paths $W'_L$ have distance at least $L$ from $R_3$ and `jump over' $R_3$.}
			\label{fig:ThickHalfThickComp:2}
		\end{figure}
		
		Let $S^i =: \dots s^i_{-1} s^i_0 s^i_1 \dots$ such that $d_G(R_3,s^i_0) < i$ and such that $S^i_{\geq 0}$ is the `upper half' and $S^i_{\leq 0}$ is the `lower half' of $S^i$ (see \cref{fig:ThickHalfThickComp:2}).
		Further, let $H'$ be the `upper half' of $H$ with respect to this enumeration, that is, let $H'$ be the subgraph of $H$ that consists of the $S^i_{\geq 0}$'s and all paths $P_{ij}$ whose endvertices lie on the $S^i_{\geq 0}$'s. Let $H''$ be the `lower half' of~$H$ defined analogously.
		
		Now choose for all $i, L \in \N$ maximal tails \defn{$T'_{iL}, T''_{iL}$} of $S^i_{\geq 0}$ and $S^i_{\leq 0}$, respectively, such that 
		\begin{enumerate}[label=(\alph*)]
			\setcounter{enumi}{2}
			\item \label{itm:ThickHalfThickComp:2:T_iL} $d_G(R_3, T'_{iL}) \geq 2L$ and $d_G(R_3, T''_{iL}) \geq 2L$.
		\end{enumerate}
		Note that by \ref{itm:ThickHalfThickComp:Si} and \ref{itm:ThickHalfThickComp:R3}, the $T'_{iL}$'s, $T''_{iL}$'s are non-empty. Note further that, by \ref{itm:ThickHalfThickComp:Pij} and \ref{itm:ThickHalfThickComp:R3}, for all $i, L \in \N$ all but finitely many of the $P_{ij}$'s avoid $B_G(R_3, L)$. 
		Hence, there are, for every $L \in \N$, escaping subdivisions $H'_L, H''_L \subseteq H_{\geq L}$ of the hexagonal `quarter-grid' whose vertical rays are the $T'_{iL}$'s or $T''_{iL}$, respectively, for $i \geq L$ such that $H'_1 \supseteq H'_2 \supseteq \dots$ as well as $H''_1 \supseteq H''_2 \supseteq \dots$, and such that $H'_L, H''_L$ avoid $B_G(R_1 \cup R_3, L)$.
		
		Since $H'_L, H''_L$ avoid $B_G(R_1 \cup R_3, L)$, they are contained in components $C_L, D_L$ of $G-B_G(R_1 \cup R_3, L)$, respectively. Clearly, $C_L$ is long and $D_L$ is half-long. As $C_1 \supseteq C_2 \supseteq \dots$ and $D_1 \supseteq D_2 \supseteq \dots$, this implies that $C_L$ is thick and $D_L$ is half-thick.
		If there is some $L \in \N$ such that $C_L \neq D_L$, then we are done as then $L$, $R := R_1 \cup R_3$, $C := C_L$ and $D := D_L$ are as desired. 
		
		Hence, we may assume that $C_L = D_L$ for all $L \in \N$.
		We will now once again construct fat $H$-paths as in the premise of \cref{lem:KFatHexGridWithJumpingPaths}, which then yields that $G$ contains $K_{\aleph_0}$ as an ultra fat minor, and which thus concludes the proof.
		Since $C_L = D_L$ and $H'_L \subseteq C_L$, $H''_{L} \subseteq D_L$ for all $L \geq K$, there are $\left(\bigcup_{i \in \N} B_G(T'_{iL}, \min\{i,L\})\right)$--$\left(\bigcup_{i \in \N} B_G(T''_{iL}, \min\{i,L\})\right)$ paths $W_L$ that avoid $B_G(R_1 \cup R_3, 5L)$. 
		We now modify $W_L$ as follows. 
		Let $i'_L$ be such that $W_L$ starts in $B_G(T'_{i'_LL}, \min\{i'_L,L\})$.
		If $W_L$ meets $B_G(P_{ij}, \min\{i,L\})$ for some $P_{ij} \subseteq H'_{4L}$ with $i \notin \{i'_L, i'_L+1\}$, then we let $P_{ij}$ be the last such path, and we shorten $W_L$ so that it meets $B_G(P_{ij}, \min\{i,L\})$ precisely in its first vertex. Then we extend $W_L$ by a shortest $W_L$--$P_{ij}$ path and a suitable subpath of~$P_{ij}$ so that $W_L$ ends in $T'_{iL}$ (or in $T'_{(i-1)L}$ if the shortest $W_L$--$P_{ij}$ path has distance $< \min\{i,L\}$ from $T'_{(i-1)L}$). Otherwise, we extend $W_L$ by a shortest $W_L$--$T'_{i'_LL}$ path. Analogously, we modify the end of~$W_L$.
		Let \defn{$W'_L$} be the path which we obtain in this way from $W_L$ and let $i_L, j_L \in \N$ be such that $W'_L$ starts in $T'_{i_LL}$ and ends in $T''_{j_LL}$. Further, let \defn{$w^0_L, w^1_L$} be the endvertices of $W'_L$ on $T'_{i_LL}$ and $T''_{j_LL}$, respectively. Then, for all $L \in \N$,
		\begin{enumerate}[label=(\greek*)]
			\item \label{itm:ThickHalfThickComp:3:Ends} $W'_L$ starts at $w^0_L \in T'_{i_LL}$ and ends at $w^1_L \in T''_{j_LL}$, 
			\item \label{itm:ThickHalfThickComp:3:R_3} $W'_L$ avoids $B_G(R_1 \cup R_3, 2L)$, and
			\item \label{itm:ThickHalfThickComp:3:T_iK} $d_G(W'_L, T'_{kL} \cup T''_{kL}) \geq \min\{k,L\}$ for all $k \neq i_L, j_L \in \N$.
		\end{enumerate}
		
		We now update the paths $W'_L$, for $L \in \N$, so that they additionally satisfy
		\begin{enumerate}[label=(\greek*)]
			\setcounter{enumi}{3}
			\item \label{itm:ThickHalfThickComp:NEW} $d_G(S^0, W'_L) > L'$ or $j_L \leq L'$ for all $L' \leq L$.
		\end{enumerate}
		For this, let $L \in \N$ be given. If $W'_L$ avoids $B_G(S^0, L)$, then it satisfies the first option of \ref{itm:ThickHalfThickComp:NEW} for all $L' \leq L$, and we can leave $W'_L$ unchanged. Therefore, assume that $W'_L$ meets $B_G(S^0, L)$, and let $w$ be the first vertex on $W'_L$ when going along $W'_L$ from $w^0_L$ to $w^1_L$ that is contained in $B_G(S^0,L)$. Let $P$ be a shortest $w$--$S^0$ path in $G$ (of length at most $L$), and let $p$ be the first vertex of $P$ in $\bigcup_{i \leq L} B_G(T''_{iL}, i)$. (Note that $P$ ends in $T''_{0L}$, and hence meets this union.) We update $j_L$ to be the (unique) $i \leq L$ such that $p \in B_G(T''_{iL},i)$, we let $W'_L$ be the concatenation of $w^0_LW'_Lw$ and $wPp$ and a $p$--$T''_{iL}$ path of length $i$, and again denote by $w^1_L$ the endvertex of (this new) $W'_L$ in $T''_{j_LL}$. Then $W'_L$ clearly still satisfies \ref{itm:ThickHalfThickComp:3:Ends} and \ref{itm:ThickHalfThickComp:3:T_iK}. Moreover, it also satisfies \ref{itm:ThickHalfThickComp:3:R_3} as the old $W'_L$ had actually satisfied \ref{itm:ThickHalfThickComp:3:R_3} with $4L$ instead of $2L$. Additionally, $W'_L$ now satisfies \ref{itm:ThickHalfThickComp:NEW}: for every $L' \geq j_L$ it satisfies the second option, and for every $L' < j_L$ it satisfies the first option by construction and by \ref{itm:DefEscHG:S^i} of escaping subdivisions (and because $M_{j_L} \geq 2j_L$). 
		
		Since $W'_L$ avoids $B_G(R_1 \cup R_3, 2L)$ but is eventually contained in $B_G(R_1 \cup R_3, L')$ for some $L' > 2L$ because~$W'_L$ is finite, we may assume, by passing to a subsequence if necessary, that the $W'_L$'s are pairwise disjoint. 
		Moreover, by \cref{cor:FurtherPropertiesOfEscapingHGs}~\ref{itm:DefEscHG:S^i:2}, every $S^i_{\geq 0}$ is contained in $B_G(R_1, L)$ for some large enough $L \in \N$ and hence, by \ref{itm:ThickHalfThickComp:3:R_3} and by once again passing to a subsequence, we may assume that 
		\begin{enumerate}[label=(\greek*)]
			\setcounter{enumi}{4}
			\item \label{itm:ThickHalfThickComp:3:i_Lleq} $i_1 < i_2 < \dots$.
		\end{enumerate} 
		Finally, by construction, $W'_L$ can have distance $< \min\{i,L\}$ to some $P_{ij}$ only if $i \in \{i_L, i_L+1, j_L, j_L+1\}$. So since~$G$ is locally finite and because of \ref{itm:ThickHalfThickComp:3:Ends} and \ref{itm:ThickHalfThickComp:3:i_Lleq}, we may assume, by deleting at most finitely many~$P_{ij}$ for every $i \in \N$ and applying \cref{prop:HexGridAfterDeletingPaths}, that 
		\begin{enumerate}[label=(\greek*)]
			\setcounter{enumi}{5}
			\item \label{itm:ThickHalfThickComp:3:P_ij} $d_G(W'_L, P_{ij}) > \min\{i,L\}$ for all $P_{ij}$ with $i \notin \{j_L, j_L+1\}$.
		\end{enumerate}
		
		\begin{figure}
			\centering
			\begin{subfigure}[b]{0.4\linewidth}
				\centering
				\begin{tikzpicture}

	\foreach \x in {0,1,2,3,4}
    \draw[line width=1.2] (\x,-0.5) -- (\x,4.5);


	\draw[blue, line width=1.7pt] (-0,3.25) to[out angle=0, in angle=180, curve through = {(1,3.5) (1.2,3.25) (1,3) (0.6,1.75) (1,1.5) (1.625,2.25) (1.625,2.5) (2,3.25) (2.175,2.45) (2.125,2.15) (1.8,1) (2.55,1.375) (2.75,1.375) (3.25,1.85) (2.75,3) (3.25,2.5)}] (4,2.25);


\begin{scope}[opacity=0.5]

	\draw[Gray, line width=6pt,line cap=round] (2,4.5) to[out angle=270, in angle=75, curve through = {(2,4) (1.5,3.25)}] (1,2);
	\draw[Gray, line width=6pt,line cap=round] (1,2) to[out angle=40, in angle=90, curve through = {(2.3,2.15) (2.85,0.75) (3,0)}] (3,-0.5);
	\foreach \x in {0,4}
    \draw[Gray, line width=6pt,line cap=round] (\x,-0.5) -- (\x,4.5);

\end{scope}


	\draw[Red, line width=1.5,line cap=round] (2,4) to[out angle=270, in angle=75, curve through = {(1.5,3.25)}] (1,2);
	\draw[Red, line width=1.5,line cap=round] (1,2) to[out angle=40, in angle=90, curve through = {(2.3,2.15) (2.85,0.75)}] (3,0);


	\draw (0,-0.8) node {\scalebox{0.85}{$S^{i_{2L-1}}$}};
	\draw (2,-0.8) node {\scalebox{0.85}{$S^{i_{2L}}$}};
	\draw (3,-0.8) node {\scalebox{0.85}{$S^{j_{2L}}$}};
    \draw (4,-0.8) node {\scalebox{0.85}{$S^{i_{2L+1}}$}};
	\draw[Gray] (2.5,4.2) node {\scalebox{0.85}{$\tilde{S}^{i_{2L}}$}};
	\draw[Red] (1.45,4) node {\scalebox{0.85}{$W'_{2L}$}};
    \draw[blue] (3.6,2.75) node {$Q_{2L}$};

\end{tikzpicture}
				\vspace{-3em}
				\subcaption{Case 2a}
				\label{fig:ThickHalfThickComp:2a}
			\end{subfigure}
			\hspace{3em}
			\begin{subfigure}[b]{0.4\linewidth}
				\centering
				\begin{tikzpicture}

	\foreach \x in {0,1,2,3,4}
    \draw[line width=1.5] (\x,-0.5) -- (\x,4.5);


    \draw[blue, line width=1.7] (-0.5,2.75) to[out angle=20, in angle=160, curve through = {(0,3) (0.3,2.9) (0,2.3) (-0.3,1.5) (0,1.4) (0.65,2) (1,2.5) (1.4,2.7) (0.8,1.3) (1,0.5) (1.5,1.1) (2,1.8) (2.3,2.2) (1.8,3.4) (2,3.5) (2.5,3.4) (3,3.5) (3.3,3.55) (2.7,2.6) (3,2.4) (4,0.5) (4.4, 1) (3.75, 1.5) (4, 1.9)}] (4.5,2);


\begin{scope}[opacity=0.5]

	\draw[Gray, line width=6pt,line cap=round] (1,4.5) -- (1,2.5);
	\draw[Gray, line width=6pt,line cap=round] (1,2.5) to[out angle=50, in angle=170, curve through = {(1.4,2.7) (0.8,1.3)}] (1,0.5);
	\draw[Gray, line width=6pt,line cap=round] (1,0.5) -- (1,-0.5);

	\draw[Gray, line width=6pt,line cap=round] (2,4.5) -- (2,3.5);
	\draw[Gray, line width=6pt,line cap=round] (2,3.5) to[out angle=185, in angle=45, curve through = {(1.8,3.4) (2.3,2.2)}] (2,1.8);
	\draw[Gray, line width=6pt,line cap=round] (2,1.8) -- (2,-0.5);

	\draw[Gray, line width=6pt,line cap=round] (3,4.5) -- (3,3.5);
	\draw[Gray, line width=6pt,line cap=round] (3,3.5) to[out angle=25, in angle=180, curve through = {(3.3,3.55) (2.7,2.6)}] (3,2.4);
	\draw[Gray, line width=6pt,line cap=round] (3,2.4) -- (3,-0.5);

	\draw[Gray, line width=6pt,line cap=round] (4,4.5) -- (4,1.9);
	\draw[Gray, line width=6pt,line cap=round] (4,1.9) to[out angle=205, in angle=0, curve through = {(3.75, 1.5) (4.4, 1)}] (4,0.5);
	\draw[Gray, line width=6pt,line cap=round] (4,0.5) -- (4,-0.5);

\end{scope}


	\draw[Red, line width=1.5,line cap=round] (2,4) to[out angle=220, in angle=35, curve through = {(1.5,3.4) (1.33,2.8) (1.3,2.4)}] (1,2); 
	\draw[Red, line width=1.5,line cap=round] (1,2) to[out angle=140, in angle=40, curve through = {(0.6,2.2)}] (0,1.8);
	\draw[Red, line width=1.5,line cap=round] (0,1.8) to[out angle=180, in angle=150, curve through = {(-0.45,1.65)}] (0,0.5);


	\draw[Red, line width=1.5,line cap=round] (0,0) to[out angle=30, in angle=150, curve through = {(0.3,0.5) (0.75,1)}] (1,0.9);
	\draw[Red, line width=1.5,line cap=round] (1,0.9) to[out angle=10, in angle=160, curve through = {(1.55,1.4) (2.3,1.2) (3.2,2.2)}] (4,2.5);

\begin{scope}[opacity=0.4]
    
	\draw[Orange, line width=5pt,line cap=round] (2,4) to[out angle=220, in angle=35, curve through = {(1.5,3.4) (1.33,2.8) (1.3,2.4)}] (1,2);
	\draw[Orange, line width=5pt,line cap=round] (1,2) to[out angle=140, in angle=40, curve through = {(0.6,2.2)}] (0,1.8);
	\draw[Orange, line width=5pt,line cap=round] (0,1.8) to[out angle=180, in angle=150, curve through = {(-0.45,1.65)}] (0,0.5);

	\draw[Orange, line width=5pt,line cap=round] (0,0.5) -- (0,0);

    \draw[Orange, line width=5pt,line cap=round] (0,0) to[out angle=30, in angle=150, curve through = {(0.3,0.5) (0.75,1)}] (1,0.9);
	\draw[Orange, line width=5pt,line cap=round] (1,0.9) to[out angle=10, in angle=160, curve through = {(1.55,1.4) (2.3,1.2) (3.2,2.2)}] (4,2.5); 
    
\end{scope}


    \draw[blue] (4.4,2.3) node {$R_3$};
	\draw (0,-0.8) node {\scalebox{0.85}{$S^{N'}$}};
	\draw (2,-0.8) node {\scalebox{0.85}{$S^{i_{2L}}$}};
    \draw (4,-0.8) node {\scalebox{0.85}{$S^{i_{2L+2}}$}};
	\draw[Gray] (2.45,4.2) node {\scalebox{0.85}{$\tilde{S}^{i_{2L}}$}};
    \draw[Gray] (4.55,4.2) node {\scalebox{0.85}{$\tilde{S}^{i_{2L+2}}$}};
	\draw[Red] (1.45,4) node {\scalebox{0.85}{$W'_{2L}$}};
    \draw[Red] (2.5, 0.9) node {\scalebox{0.85}{$W'_{2L+2}$}};
	\draw[Orange] (-0.5,0.25) node {\scalebox{0.85}{$W''_m$}};
    
\end{tikzpicture}
				\vspace{-3em}
				\subcaption{Case 2b}
				\label{fig:ThickHalfThickComp:2b}
			\end{subfigure}
			\caption{The new double rays $\widetilde{S}^{i_L}$ in Case 2a and 2b (indicated in grey), which again form the vertical double rays of an escaping subdivision $\widetilde{H}$ of the hexagonal half-grid. In Case 2a, subpaths $Q_{2L}$ of $R_3$ are fat $\widetilde{H}$-paths, while in Case 2b, the paths $W''_m$ are fat $\widetilde{H}$-paths.} 
			\label{fig:ThickHalfThickComp:2ab}
		\end{figure}
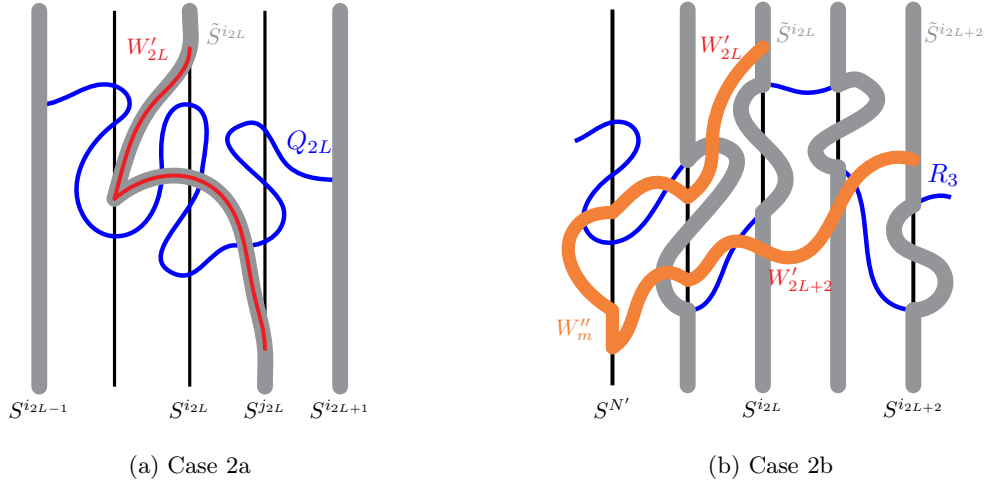
		
		Note that we might not be able to obtain \ref{itm:ThickHalfThickComp:3:P_ij} for all $i \in \N$ if infinitely many $j_L$ are the same. We now distinguish two cases based on this observation.
		
		\paragraph{\textbf{Case 2a:}}
		\emph{There is a sequence $L_1 < L_2 < \ldots \in \N$ such that $j_{L_1} < j_{L_2} < \dots$.}
		
		Since also $i_{1} < i_2 < \dots$ by \ref{itm:ThickHalfThickComp:3:i_Lleq}, we may assume, by passing to a subsequence of the $W'_L$'s, that $i_L, j_L < i_{L'}, j_{L'}$ for all $L < L' \in \N$. 
		Then we obtain, similar to~\ref{itm:ThickHalfThickComp:3:P_ij}, that
		\begin{enumerate}[label=(\greek*)]
			\setcounter{enumi}{6}
			\item \label{itm:ThickHalfThickComp:3:P_ij:2} $d_G(W'_L, P_{ij}) > \min\{i,L\}$ also for all $P_{ij}$ with $i \in \{j_L, j_L+1\}$.
		\end{enumerate}
		Moreover, by \ref{itm:ThickHalfThickComp:NEW}, and again passing to a subsequence of the $W'_L$'s, we may assume that 
		\begin{enumerate}[label=(\greek*$^\prime$)]
			\setcounter{enumi}{3}
			\item \label{itm:ThickHalfThickComp:NEW1} $d_G(S^0, W'_L) > L$ for all $L \in \N$.
		\end{enumerate}
		In particular, since the $W'_L$'s are finite, and by once again passing to a subsequence of the $W'_L$'s, we may assume that $W'_L \subseteq G[S^0, M_{i_{L+1}}-2i_{L+1}]-B_G(S^0, M_{i_{L-1}}+2i_{L-1})$.
		
		For every $2L \in \N$, we set $\widetilde{S}^{i_{2L}} := S^{j_{2L}}w^1_{2L}W'_{2L}w^0_{2L}S^{i_{2L}}$ (see \cref{fig:ThickHalfThickComp:2ab}\,\subref{fig:ThickHalfThickComp:2a}). 
		Since every $\widetilde{S}^{i_{2L}}$ has a tail in $S^{i_{2L}}_{\geq 0}$ and in $S^{j_{2L}}_{\leq 0}$ and because of \ref{itm:ThickHalfThickComp:3:P_ij} and \ref{itm:ThickHalfThickComp:3:P_ij:2}, we can find in $H$ infinitely many $S^{i_{2L-1}}$--$\widetilde{S}^{i_{2L}}$ paths and infinitely many $\widetilde{S}^{i_{2L}}$--$S^{i_{2L+1}}$ paths that make the $\widetilde{S}^{i_{2L}}$'s and the $S^{i_{2L+1}}$'s into an escaping subdivision~$\widetilde{H}$ of the hexagonal half-grid.
		By \ref{itm:ThickHalfThickComp:2:S^i} and \ref{itm:ThickHalfThickComp:2:S^ileq}, $R_3$ contains for every $2L$ a subpath $Q_{2L}$ that starts in $B_G(S^{i_{2L-1}}, i_{2L-1})$, ends in $B_G(S^{i_{2L+1}}, i_{2L+1})$ and is otherwise disjoint from all $B_G(S^{i_{L'}}, i_{L'})$ and $B_G(S^{j_{L'}}, j_{L'})$ with $L' \neq 2L \in \N$. 
		Since $d_G(R_3, W'_{2L'}) \geq 2L'$ by \ref{itm:ThickHalfThickComp:3:R_3}, it follows that also $d_G(Q_{2L}, \widetilde{S}^{i_{2L'}}) \geq 2L'$ for all $L' \neq L \in \N$. Moreover, by the definition of $\widetilde{S}^{i_{2L}}$ and because of \ref{itm:ThickHalfThickComp:3:R_3} and \ref{itm:ThickHalfThickComp:2:T_iL}, we also have $d_G(Q_{2L}, \widetilde{S}^{i_{2L}}) \geq 2L$.
		Hence, extending the $Q_{2L}$'s by shortest paths to $S^{i_{2L-1}}$--$S^{i_{2L+1}}$ paths yields fat $\widetilde{H}$-paths as in \cref{lem:KFatHexGridWithJumpingPaths}. This implies that $G$ contains $K_{\aleph_0}$ as an ultra fat minor, and hence concludes this case of the proof.
		
		\paragraph{\textbf{Case 2b:}} 
		\emph{There is some $N \in \N$ such that $j_L \leq N$ for infinitely many $L \in \N$.}
		
		By passing to a subsequence of the $W'_L$'s, we may assume that $j_L = N'$ for some $N' \in \N$ and all $L \in \N$.
		We again modify $H$ as follows.
		For every $i \in \N$, let $\widetilde{S}^{i}$ be obtained from $S^{i}$ by replacing the subpath $\widetilde{T}_i := S^i \setminus (T'_{ii} \cup T''_{ii})$ of~$S^i$ by a path $Q_i$ that consists of a suitable subpath of $R_3$ together with shortest $T'_{ii}$--$R_3$ and $R_3$--$T''_{ii}$ paths (see \cref{fig:ThickHalfThickComp:2ab}\,\subref{fig:ThickHalfThickComp:2b}). Then the $\widetilde{S}^i$'s are again double rays. 
		By \ref{itm:ThickHalfThickComp:Si} and \ref{itm:ThickHalfThickComp:2:S^ileq} and because $G$ is locally finite and the $Q_i$'s are finite and pairwise disjoint, there are $N' < i_{L_1} < i_{L_2} < \ldots \in \N$ such that the $\widetilde{S}^{i_{L_j}}$'s again satisfy \ref{itm:ThickHalfThickComp:Si} with $M_{i_{Lj}}$ updated to $M_{i_{L(j+1)}-1}$. Since every $\widetilde{S}^{i}$ still has a tail in~$S^{i}_{\geq 0}$ and in~$S^{i}_{\leq 0}$ and because of \ref{itm:ThickHalfThickComp:R3Pij}, we can find in $H$, for every $j \in \N$, infinitely many $\widetilde{S}^{i_{L_j}}$--$\widetilde{S}^{i_{L_{j+1}}}$ paths that make the $\widetilde{S}^{i_{L_j}}$'s into an escaping subdivision $\widetilde{H}$ of the hexagonal half-grid. 
		Since all $W'_{L_j}$ end in~$S^{N'}$, we can connect the $W'_{L_{2j}}$'s pairwise by suitable subpaths of~$S^{N'}$, to obtain infinitely many pairwise disjoint $\widetilde{H}$-paths $W''_m$. Note that a path $W''_m$ obtained from $W'_{L_i}$ and $W'_{L_j}$ for $i, j \in 2\N$ is an $\widetilde{H}$-path that starts in~$\widetilde{S}^{i_{L_i}}$ and ends in~$\widetilde{S}^{i_{L_j}}$. Moreover, by \ref{itm:ThickHalfThickComp:3:R_3}, \ref{itm:ThickHalfThickComp:3:T_iK} and \ref{itm:ThickHalfThickComp:3:P_ij} and because $\widetilde{H}$ is escaping, $W''_m$ is $\min\{L_i, L_j\}$-fat. Hence, by \ref{itm:ThickHalfThickComp:3:i_Lleq} and because we only used the paths $W'_{L_i}$ with $i \in 2\N$ to construct the $W''_m$'s, infinitely many of the $W''_m$'s yield fat $\widetilde{H}$-paths as in \cref{lem:KFatHexGridWithJumpingPaths}, which implies that $G$ contains $K_{\aleph_0}$ as an ultra fat minor, and which thus concludes the proof.
	\end{proof}
	
	We are now in a position to prove \cref{thm:AsymptoticFullGridWeaker}.
	
	\begin{proof}[Proof of \cref{thm:AsymptoticFullGridWeaker}]
		The assertion follows immediately by first applying \cref{lem:ThickAndHalfThickComponent} and \cref{lem:TwoThickComponents} and then \cref{lem:TwoThickCompsYieldFullGrid}.
	\end{proof}

	\section{Further comments}\label{sec:problems}
	
	Before we discuss other related topics to this paper, let us note that our main results give only partial answers to the problems that we mentioned in the introduction, that is to \cite{GP2023+}*{Problem~7.3} and \cite{GH24+}*{Problems 4.1 and 4.2}.
	Hence, these problems are still open for arbitrary finitely generated groups that need not be finitely presented and, more generally, for \lf, \qt\ graphs without any restrictions on their cycle spaces.
	Moreover, we propose the following strengthening of those problems that aligns well with \cref{thm:AsymptoticFullGrid}.
	
	\begin{conj}\label{prob:strengthening:GP7.3+GH4.1+4.2}
		Let $G$ be a \lf, \qt\ graph that has a thick end. Then $G$ contains either $K_{\aleph_0}$ as an ultra-fat minor or an escaping subdivision of the hexagonal full-grid.
	\end{conj}

	\subsection{Coarse embeddings}
	
	\cref{thm:DivergingFG} asserts that we can find a diverging subdivision of the hexagonal full-grid in every \lf, \qt\ graph whose cycle space is generated by cycles of bounded length and that has a thick end. The advantage of diverging subdivisions over arbitrary subdivisions is that they preserve some of the geometry of the original graph.
	One might wish to strengthen \cref{thm:DivergingFG}, by asking for a subdivision of the hexagonal full-grid whose geometry is even closer related to the geometry of~$G$. 
	
	For two graphs $G$ and~$H$, a map $f\colon V(H) \to V(G)$ is a \defn{coarse embedding} if there exist functions $\rho^-\colon [0, \infty) \to [0, \infty)$ and $\rho^+\colon [0, \infty) \to [0, \infty)$ such that $\rho^-(a) \to \infty$ for $a \to\infty$
	and
	\[
	\rho^-(d_H(u, v)) \leq d_G(f(u), f(v)) \leq \rho^+(d_H(u, v))
	\] 
	for all $u, v \in V(H)$.
	It is easy to check that a coarse embedding of the hexagonal full-grid always yields a diverging subdivision; however, conversely, a diverging subdivision is in general much weaker than a coarse embedding. One may thus ask whether we can always find a coarse embedding of the hexagonal full-grid in a \lf, \qt\ graph with a thick end.
	
	However, it was already discussed in~\cite{GH24+} that for arbitrary \lf, \qt\ graphs (without any condition on their cycle spaces) we cannot ask for coarse embeddings of the hexagonal full-grid.
	Indeed, coarse embeddings preserve the asymptotic dimension\footnote{See e.g.\ \cite{asymptoticdimminorclosed} for a definition of the asymptotic dimension.}, that is, if $H$ has asymptotic dimension at least~$n$ and $H$ is coarsely embeddable into~$G$, then the asymptotic dimension of~$G$ is at least~$n$, too.
	Since every \lf\ Cayley graph of the lamplighter group has asymptotic dimension~$1$, see Gentimis~\cite{G08}, and has a thick end, but the full-grid has asymptotic dimension~$2$, we cannot ask for coarse embeddings of the hexagonal full-grid into all \lf, \qt\ graphs with thick ends.
	
	However, a special case of a theorem by Fujiwara and Whyte~\cite{FW07} states that every \lf, \qt\ graph with a thick end whose cycle space is generated by cycles of bounded length has asymptotic dimension at least~$2$.
	Thus, the asymptotic dimension of the full-grid does not prevent it from being coarsely embeddable into such graphs. This motivates the following problem.
	
	\begin{conj}
		Let $G$ be a \lf, \qt\ graph whose cycle space is generated by cycles of bounded length and that has a thick end.
		Is the hexagonal full-grid coarsely embeddable into~$G$?
	\end{conj}
	
	\noindent Note that a positive answer to this question would also yield a positive answer to \cite{GH24+}*{Problem 4.5}.
	
	\subsection{Quasi-isometries to trees} \label{subsec:QuasiIsomToTrees}
	
	For two graphs $G$ and~$H$, a map $f\colon V(H) \to V(G)$ is a \defn{quasi-isometry} if there exist $c\geq 1$ and $d\geq 0$ such that
	\[
	\frac{1}{c}(d_H(u, v))-d \leq d_G(f(u), f(v)) \leq c(d_H(u, v))+d
	\]
	for all $u, v \in V(H)$ and
	\[
	d_H(f(V(G)),w)\leq d
	\]
	for all $w\in V(H)$.
	Two graphs are \defn{\qi} if there exists a \qiy\ between them.
	
	As we have discussed in the introduction, a result by Kr\"on and M\"oller \cite{KM08}*{Theorem~5.5} asserts that a \lf, \qt, connected graph is \qi\ to a tree if and only if it has no thick end.
	Hence, we obtain the following corollary from \cref{main:AsymptoticFullGrid,main:DivergingFullGrid}, which yields two new characterisations of \qt, \lf, connected graphs that are \qi\ to trees for the special case that the cycle space is generated by cycles of bounded length.
	
	\begin{cor} \label{cor:AsymptoticFullGrid+DivergingFullGrid}
		Let $G$ be a locally finite, \qt, connected graph whose cycle space is generated by cycles of bounded length.
		Then the following are equivalent:
		\begin{enumerate}[label=\rm{(\roman*)}]
			\item $G$ has a thick end.
			\item $G$ contains the full-grid as an asymptotic minor.
			\item $G$ contains the full-grid as a diverging minor.
			\item $G$ is not \qi\ to a tree. \qed
		\end{enumerate}
	\end{cor}
	
	For further characterisations of \qt, \lf, connected graphs that are \qi\ to trees, we refer the reader to \cites{A11,HLMR,KM08,W89}.

	\subsection{Quasi-isometries to planar graphs}
	
	Finally, we would like to draw the reader's attention to another related problem, which is still open. As we discussed in the previous subsection, the (global) geometry of \lf, \qt\ graphs without a thick end is well understood as they are quasi-isometric to forests. In that sense, our results, \cref{main:AsymptoticFullGrid,main:DivergingFullGrid}, can be seen as a step towards understanding the (global) geometry of the remaining \lf, \qt\ graphs -- those with a thick end. In the case where the cycle space of a \lf, \qt\ graph $G$ is generated by cycles of bounded length, we showed that $G$ contains the full-grid as an asymptotic minor. Since asymptotic minors cannot hide in a ball of small radius, they will appear in the global structure of~$G$.
	However, even if $G$ is one-ended, this does not mean that the geometry of~$G$ resembles that of the full-grid or, more generally, of a one-ended, planar graph. The reason for this is simple: the global structure of~$G$ may be far more involved, and may contain the full-grid only as a substructure.
	Indeed, even in our proof, we might have found the asymptotic full-grid inside an asymptotic minor of the infinite complete graph. However, Georgakopoulos and Papasoglu conjectured that this is in fact the only thing that can happen.
	
	\begin{conj}\cite{GP2023+}*{Conjecture 9.3}\label{prob:GP:9.4}
		Let $G$ be a \lf, transitive graph.
		Then $G$ either is \qi\ to a planar graph or contains every finite graph as an asymptotic minor.
	\end{conj}
	
	\noindent Note that this can be seen as a coarse version of Thomassen's~\cite{T1992} result that every \lf, one-ended, transitive graph is either planar or can be contracted into the infinite complete graph.
	\medskip
	
	MacManus~\cite{M24+} proved \cref{prob:GP:9.4} in the special case where $G$ is a locally finite Cayley graph of a finitely presented group.
	Recall that every Cayley graph of a finitely presented group is transitive and has a cycle space which is generated by cycles of bounded length.
	So the assumption that the cycle space is generated by cycles of bounded length, which was already crucial for our proofs of \cref{main:AsymptoticFullGrid,main:DivergingFullGrid,main:HalfGrid:BoundedCycles,main:DivergingHalfGrid}, reappears here.
	
	We remark that MacManus's proof uses deep group-theoretic results that have no counterpart for \qt\ graphs. Hence, \cref{prob:GP:9.4} is still open for arbitrary \lf, \qt\ graphs whose cycle space is generated by cycles of bounded length.
	
	In line with our results, we may ask for a strengthening of Georgakopoulos and Papasoglu's conjecture, and thus for a more direct coarse analogue of Thomassen's theorem, as follows.
	
	\begin{conj}\label{conj:Thomassen:ultraFat}
		Let $G$ be a \lf, \qt\ graph.
		Then $G$ is either \qi\ to a planar graph or it contains $K_{\aleph_0}$ as an ultra-fat minor.
	\end{conj}
	
	We remark that a positive solution to \cref{conj:Thomassen:ultraFat} implies one for \cref{prob:strengthening:GP7.3+GH4.1+4.2}.
	For this, by \cref{thm:AsymptoticFullGrid}, it suffices to show that the cycle space of every \lf, \qt\ graph that is \qi\ to a planar graph is generated by cycles of bounded length.
	MacManus \cite{M24+b}*{Theorem~A} showed that every \lf, \qt\ graph that is \qi\ to a planar graph is \qi\ to a planar graph that is additionally \qt, and whose cycle space is thus generated by cycles of bounded length by a result of the second author \cite{H18b}*{Theorem 7.2}.
	To show that a positive solution to \cref{conj:Thomassen:ultraFat} implies one for \cref{prob:strengthening:GP7.3+GH4.1+4.2}, it thus remains to prove that the property of having a cycle space that is generated by cycles of bounded length is preserved under quasi-isometries.
	This, however, follows using the standard back-and-forth argument in the proof that finite presentability of finitely generated groups is preserved under quasi-isometries. 
	
	In light of this, we also propose the following weakening of \cref{conj:Thomassen:ultraFat}:
	\begin{conj}
		Let $G$ be a locally finite, \qt\ graph. 
		If the cycle space of $G$ is not generated by cycles of bounded length, then $G$ contains $K_{\aleph_0}$ as an ultra-fat minor.
	\end{conj}

	\section*{Acknowledgement}
	
	We thank the referees for carefully reading the paper and spotting a mistake in the proof of \cref{lem:ThickAndHalfThickComponent}.

	\bibliographystyle{amsplain}
	\bibliography{referencesArXiv}
	
\end{document}